%% file: arxiv_article.tex
\documentclass{amsart}
\usepackage[utf8]{inputenc}
\usepackage[T1]{fontenc}
\usepackage{amsaddr}
\usepackage{amssymb}
\usepackage{bbm}
\usepackage{mathtools}
\usepackage{amsmath}
\usepackage{subfigure}
\usepackage{amssymb}
\usepackage{mathrsfs}

\usepackage{lipsum}
\usepackage{amsfonts}
\usepackage{graphicx}
\usepackage{epstopdf}
\usepackage{algorithmic}
\usepackage{tikz}

\newtheorem{theorem}{Theorem}
\newtheorem{definition}{Definition}
\newtheorem{corollary}{Corollary}
\newtheorem{lemma}{Lemma}
\newtheorem{remark}{Remark}

\newtheorem{example}{Example}

\usepackage[backend=biber]{biblatex}
\bibliography{bibliography.bib}
\usepackage{tikz}
\usetikzlibrary{shapes,arrows,positioning,calc,external}
\tikzset{
	block/.style = {draw, fill=white, rectangle, minimum height=3em, minimum width=3em},
	tmp/.style  = {coordinate}, 
	sum/.style= {draw, fill=white, circle, node distance=1cm, inner sep=0pt},
	input/.style = {coordinate},
	output/.style= {coordinate},
	pinstyle/.style = {pin edge={to-,thin,black}
	}
}

\newcommand{\mat}[1]{\left(\begin{matrix}#1\end{matrix}\right)}
\newcommand{\smat}[1]{\left(\begin{smallmatrix}#1\end{smallmatrix}\right)}

\newcommand{\convw}{\bar{\operatorname{conv}}^{\operatorname{w}}}

\renewcommand{\c}[1]{\mathcal{#1}}
\renewcommand{\b}[1]{\mathbb{#1}}
\newcommand{\s}[1]{\mathscr{#1}}
\newcommand{\bs}[1]{\boldsymbol{#1}}
\newcommand{\intg}[3]{\int_{#1} #2 \,\text{d}#3}

\DeclarePairedDelimiter{\abs}{\lvert}{\rvert}
\DeclarePairedDelimiter{\norm}{\lVert}{\rVert}
\newcommand{\<}{\left\langle}
\renewcommand{\>}{\right\rangle}

\newcommand{\1}{\mathbbm{1}}
\newcommand{\argmin}{\operatorname{argmin}}
\renewcommand{\bar}[1]{\overline{#1}}
\renewcommand{\hat}[1]{\widehat{#1}}
\renewcommand{\tilde}[1]{\widetilde{#1}}
\DeclareMathOperator{\supp}{supp}

\title[Distributionally Robust Optimization with i.i.d. Structure]{Distributionally Robust Optimization over Wasserstein Balls with i.i.d. Structure}
\email{\texttt{\{akharitenko,mfochesato,atsiamis,nikschmid,jlygeros\}@ethz.ch}}
\date{\today}
\keywords{Data-driven optimization, Wasserstein distributionally robust optimization, structured ambiguity set, i.i.d. structure.}

\begin{document}
	\maketitle
	\begin{center}
		\small
		Andrey Kharitenko\textsuperscript{1}, Marta Fochesato\textsuperscript{2},
		Anastasios Tsiamis\textsuperscript{2}, Niklas Schmid\textsuperscript{2} \par and
		John Lygeros\textsuperscript{2} \par \bigskip
		
		\textsuperscript{1}Department of Computer Science, ETH Z\"urich, 8092 Z\"urich, Switzerland \par
		\textsuperscript{2}Automatic Control Laboratory, ETH Z\"urich, 8092 Z\"urich, Switzerland\par \bigskip
\end{center}

\maketitle
% REQUIRED
\begin{abstract}
	We consider distributionally robust optimization problems where the uncertainty is modeled via a structured Wasserstein ambiguity set. 
	Specifically, the ambiguity is restricted to product measures $P^{\otimes N}$, where $P$ lies within a Wasserstein ball centered at a distribution $\hat{P}$. 
	This structure reflects the assumption of independent and identically distributed (i.i.d.) uncertainty components and yields a non-convex ambiguity set that is strictly contained in its unstructured counterpart, thereby reducing conservatism. 
	The resulting optimization problem is generally intractable due to the loss of convexity. 
	We address this by introducing a sequence of tractable convex relaxations, each admitting strong duality, and prove that this sequence converges to the original problem value under suitable conditions. 
	Numerical examples are provided to illustrate the effectiveness of the proposed approach.
	As a byproduct of our proofs, we establish a novel formula, of independent interest, relating the Wasserstein distance of a mixture of product distributions to the Wasserstein distance between its constituent measures. 
\end{abstract}

% REQUIRED
% \begin{MSCcodes}
% ?, ?, ?
% \end{MSCcodes}

\section{Introduction}
\label{sec:introduction}
In stochastic optimization a common goal is to minimize an objective $\Psi$ over a set of feasible decisions $\Theta$, where the objective $\Psi$ is defined as an average of a family of individual uncertainty-affected loss functions $\ell:\Theta \times X \to \mathbb{R}$, with $X$  being a random vector of uncertain parameters defined on a probability space $(X,\Sigma,P)$.
In mathematical terms, a stochastic optimization method evaluates
\begin{align}
	\label{eq:stochasticOptimization}
	\inf_{\theta \in \Theta} \Psi(\theta) \text{\ \ with\ \ } \Psi(\theta) = \mathbb{E}_{x \sim P} \ell(\theta,x)\,.
\end{align}
To avoid trivialities, we assume throughout that the feasible set $\Theta \subseteq \mathbb{R}^m$ and the support set $X \subseteq \mathbb{R}^d$ are non-empty and closed.

Problem \eqref{eq:stochasticOptimization} is ubiquitous in the areas of machine learning, operation research, economics, and automatic control \cite{shapiro2021lectures}.
Unfortunately, the practical deployment of \eqref{eq:stochasticOptimization} is complicated by the fact that the precise form of the underlying distribution $P$ is often unknown and can only be inferred indirectly from past data in the form of a finite number of samples $x_1,\ldots,x_n \in X$ \cite{bertsimas2006robust}. In this case, one can employ statistical methods to infer an estimated (parametric or non-parametric) distribution $\hat{P}$ from the available data. However, solving \eqref{eq:stochasticOptimization} with the estimated $P = \hat{P}$ may yield solutions that display poor out-of-sample performance due to the unavoidable mismatch between the true underlying distribution and $\hat{P}$.

In distributionally robust (stochastic) optimization (DRO), the decision-maker hedges against this mismatch by minimizing the worst-case expected loss $\Psi_{\operatorname{WC}}$ with respect to all distributions in a neighborhood of $\hat{P}$; namely the (unstructured) DRO problem is formulated as
\begin{align}
	\label{eq:stochasticOptimizationRobust}
	\inf_{\theta \in \Theta} \Psi_{\operatorname{WC}}(\theta) \text{\ \ with\ \ } \Psi_{\operatorname{WC}}(\theta) = \sup_{P \in \mathcal{W}}\mathbb{E}_{x \sim P} \ell(\theta,x)\,,
\end{align}
where $\mathcal{W} \subseteq \mathcal{P}(X)$ is a set of distributions, called ambiguity set, on the space $\mathcal{P}(X)$ of all distributions supported on $X$.
If appropriately chosen, the set $\mathcal{W}$ contains the true underlying distribution, which implies that any solution $\theta_*$ to \eqref{eq:stochasticOptimizationRobust} will provide an objective value of $\Psi(\theta_*)$ lower or equal to $\Psi_{\operatorname{WC}}(\theta_*)$ \cite{mohajerin2018data}.
Consequently, the quality of this robust approach heavily depends on the form of the ambiguity set $\mathcal{W}$: while it should be ``large enough'' to include the true distribution (with high confidence), it should not contain superfluous distributions that cannot realistically appear in the problem at hand as this would lead to unnecessarily conservative solutions. 
Typical ambiguity sets that appear in the literature include support-, moment-, or distance-based sets of distributions or mixtures thereof \cite{rahimian2019distributionally}.
While the first two types of ambiguity sets contain all distributions complying with a specified support and moment information (generally, first and second moments) \cite{delage2010distributionally, popescu2007robust}, the distance-based sets include all distributions that are within a certain given ``distance'' of a fixed nominal distribution. 
In the latter setting, the nominal distribution is often obtained through statistical techniques from available empirical data, while the distance, commonly expressed in terms of e.g. the $\phi-$divergence \cite{bayraksan2015data, ben2013robust}, the total variation norm \cite{tzortzis2015dynamic}, the kernel mean embedding \cite{hannah2010nonparametric, bertsimas2020predictive}, or optimal transport based-distances \cite{shafieezadeh2023new} including the celebrated Wasserstein distance \cite{gao2023distributionally, mohajerin2018data}, signifies the ``trust'' in the statistical methods used as well as the obtained data at hand. 
Due to the favorable properties of Wasserstein distance in terms of expressivity and statistical properties allowing for finite-sample guarantees \cite{fournier2015rate}, a significant proportion of the recent literature has focused on so-called Wasserstein balls $\mathcal{W} = \mathbb{B}_\rho(\hat{P})$, i.e., sets containing all distributions that are within some Wasserstein-distance $\rho > 0$ from a nominal distribution $\hat{P}$. 
For finitely supported nominal distributions $\hat{P}$, the optimization of an expectation over a Wasserstein ball, which is convex, can often be reformulated into a finite-dimensional optimization program by means of Lagrange duality \cite{gao2023distributionally,shafieezadeh2023new}, and solved efficiently via off-the-shelf solvers. 

However, in several applications, the distributional uncertainty does not enter the problem at hand in an arbitrary fashion, but rather in a \emph{structured} manner.
One common structure consists of a set $x = (x_1,\ldots,x_N)$ of identically and independently distributed (i.i.d.) random variables $x_1,\ldots,x_N \sim P$, where $P$ is again assumed to belong to a Wasserstein ball $\mathbb{B}_\rho(\hat{P})$.
This results in the \emph{structured Wasserstein distributionally robust optimization} problem of the form
\begin{align}
	\label{eq:stochasticOptimizationRobustStructured}
	\inf_{\theta \in \Theta} \Psi_{\operatorname{S}}(\theta) \text{\ \ with\ \ } \Psi_{\operatorname{S}}(\theta) = \sup_{\bar{P} \in \mathcal{W}} \mathbb{E}_{x \sim \bar{P}} \ell(\theta,x)\,,
\end{align}
where 
\begin{align}
	\label{eq:stochasticOptimizationRobustStructuredSet}
	\mathcal{W} = \{P^{\otimes N} \mid P \in \mathbb{B}_\rho(\hat{P})\}
\end{align} 
is a \emph{structured Wasserstein ambiguity set} containing product distributions of the form $\bar{P} = P^{\otimes N} = P \otimes \cdots \otimes P$ only. 
Problem \eqref{eq:stochasticOptimizationRobustStructured} arises across several domains, such as (i) control of uncertain dynamical systems, which has traditionally assumed stationarity and independence of the additive noise affecting the system dynamics; (ii) strategically robust game theory, where irrationality in the opponents' actions is captured via a distributionally robust ``best response map'' that, in case of disjoint agents action sets, takes the form of \eqref{eq:stochasticOptimizationRobustStructured}; (iii) supply chain optimization and/or inventory management of goods with uniform popularity across buyers and shared demand drivers may be approximated as \eqref{eq:stochasticOptimizationRobustStructured}. 

% Fourth, energy bidding in forward electricity markets of renewable energy producers may be approximated as \eqref{eq:stochasticOptimizationRobustStructured} upon removing seasonality trends in the (uncertain) weather conditions. 

Despite the popularity of this structure in stochastic optimization, handling Problem \eqref{eq:stochasticOptimizationRobustStructured} is computationally challenging. In fact, contrary to the problem of evaluating $\Psi_{\operatorname{WC}}(\theta)$, the evaluation of $\Psi_{\operatorname{S}}(\theta)$ now involves optimizing over the non-convex set $\mathcal{W}$ (the non-convexity arises from the nonlinearity in the expectation operator due to the product structure). In turn, this prohibits the use of standard (strong) duality tools to reformulate \eqref{eq:stochasticOptimizationRobustStructured} as an equivalent finite-dimensional optimization program.

\subsection{Related work}
While the literature concerning the (unstructured) Wasserstein DRO problem in \eqref{eq:stochasticOptimizationRobust} is abundant, little attention has been devoted to the case of its structured counterpart. The closest works to ours are \cite{chaouach2022tightening,chaouach2023structured}, where non-convex ambiguity sets of the form 
\begin{align*}
	\c{W}_{\operatorname{rect}} = \{P_1 \otimes \cdots \otimes P_l \mid P_i \in \b{B}_{\rho_i}(\hat{P}_i) \text{\ for all\ } i=1,\ldots,l\}\,.
\end{align*}
are considered.
For losses $\ell$ that are either additively or multiplicatively separable, a strong duality result is proven, which allows, as in the unstructured case, to compute the worst-case expectation by solving a finite-dimensional convex program.
Moreover, \cite{chaouach2022tightening,chaouach2023structured} also proposes convex overestimations of the above ambiguity set for which strong duality can be shown without requiring the restrictive separability condition.
We will compare our proposed approach in more detail to the latter in Section \ref{sec:structured_sets_comparison}.
On a different direction, \cite{gao2017distributionally} considers the problem of maximizing an expectation functional over a Wasserstein ambiguity set under an additional moment constraints of the form
\begin{align*}
	\c{W}_{\operatorname{mom}} = \b{B}_\rho(\hat{P}) \cap \{P \in \c{P}(\b{R}^d) \mid \b{E}_{x \sim P} (x-\mu)(x-\mu)^\top \leq \Sigma\}\,,
\end{align*}
for a-priori given mean vector $\mu \in \mathbb{R}^d$ and covariance matrix $\Sigma \in \mathbb{S}^{d \times d} \succeq 0$. 
By choosing $\Sigma$ to be diagonal, independence among the lower-dimensional components of the uncertain vector can be enforced. 
The resulting ambiguity set $\c{W}$ retains convexity as it is described as the intersection of convex sets, allowing again to compute the worst-case expectation by means of convex programming for certain classes of loss functions, such as e.g. quadratic or piecewise affine $\ell$.

\subsection{Outline and contributions of the paper}
We consider data-driven DRO problems with Wasserstein ambiguity sets, where the
uncertainty affects the problem in an i.i.d. fashion, as in \eqref{eq:stochasticOptimizationRobustStructured}. 
Specifically, we first examine the problem of evaluating the inner supremum of \eqref{eq:stochasticOptimizationRobustStructured} for a fixed, given $\theta \in \Theta$, i.e. computing the so-called \emph{(primal) uncertainty quantification (UQ) problem}
\begin{align}
	\label{eq:structuredDRO}
	\sup_{\bar{P} \in \mathcal{W}} \mathbb{E}_{x \sim \bar{P}} \ell(x)\,,
\end{align}
where $\c{W}$ is defined in \eqref{eq:stochasticOptimizationRobustStructuredSet}.
After introducing the necessary background in Section \ref{sec:preliminaries}, we establish conditions under which \eqref{eq:structuredDRO} is finite and attains its optimum in Section \ref{sec:structuredDROExistenceFiniteness}.
Moreover, there we also show that \eqref{eq:structuredDRO} can be upper-bounded by a standard Wasserstein uncertainty quantification problem.
Next, in Section \ref{sec:structuredDROConvexUpperBound} we show that the latter upper bound is in general conservative and propose a potentially tighter upper bound based on symmetrization of the corresponding loss $\ell$ that allows for strong duality.
Further, in Section \ref{sec:structuredDROSequenceOfConvexRelaxations}, using the concept of lifting and based on the previous bound, we introduce a nonincreasing sequence of upper bounds (relaxations) on \eqref{eq:structuredDRO} that admit strong duality. 
We investigate the properties of this sequence of relaxations in Section \ref{sec:structuredDRORelaxationGap}.
As a main result, in Section \ref{sec:structuredDROTightnessOfRelaxationSequence}, we show in Theorem \ref{thm:relaxationExactConcave} that the gap between the lower bound of the relaxation sequence and \eqref{eq:structuredDRO} vanishes if $\ell$ is concave. 
Additionally, in Section \ref{sec:structured_sets_comparison} we show that for sets of the form $\c{W}$ our upper bound established in Section \ref{sec:structuredDROConvexUpperBound} is not more conservative than an upper bound provided in \cite{chaouach2023structured}.
Furthermore, in Section \ref{sec:structuredDROOuter} we turn back to Problem \ref{eq:stochasticOptimizationRobustStructured} and show that if $\ell = \ell(\theta,x)$ is convex-concave, then the minimizers of the sequence of relaxed problems, if existent, converge to the set of minimizers of Problem \ref{eq:stochasticOptimizationRobustStructured}.
Finally, in Section \ref{sec:reformulationNumerical} we formulate the strong duality for the relaxation sequence as a second order cone program whenever $\ell$ is a polyhedral loss function and provide some numerical examples.

\section{Notation and Preliminaries}
\label{sec:preliminaries}

In this section we recapitulate some well-known results about the Wasserstein distance and Wasserstein distributionally robust optimization.
For a more throughout introduction we refer to \cite{villani2009optimal,kuhn2019wasserstein}. 
\subsection{Notation} 
For any measurable space $(X,\c{X})$ we denote by $\c{P}(X)$ the set of probability measures on $(X,\c{X})$.
Given two measurable spaces $(X,\c{X})$, $(Y,\c{Y})$, a measurable map $f:X \to Y$ and $P \in \c{P}(X)$, then $f \# P \in \c{P}(Y)$ denotes the pushforward distribution defined by $f \# P(A) = P(f^{-1}(A))$ for all $A \in \c{Y}$.
We denote the product space $X^N = X \times \cdots \times X$ for $N \in \b{N}$ and the projections $\operatorname{pr}_i:X^N \to X$ onto the $i$-th coordinate.
If $M \geq N$ is an integer, then by $\operatorname{pr}_{1:N}^M:X^M \to X^N$ we denote the projection onto the first $N$ factors of $X^M$.
Finally, if $M$ is clear from the context, then it is omitted: $\operatorname{pr}_{1:N} = \operatorname{pr}_{1:N}^M$.
If $X$ is a Polish space, then we set $\c{X} = \c{B}(X)$ to be the Borel $\sigma$-algebra and omit $\c{X}$ from the notation altogether.
For two distributions $P_1, P_2 \in \c{P}(X)$ we denote by $\Gamma(P_1,P_2)$ the set of all couplings of $P_1$ and $P_2$, i.e. all $\Lambda \in \c{P}(X \times X)$ such that $\operatorname{pr}_i \# \Lambda = P_i$ for $i=1,2$.
A (transportation) cost $c:X \times X \to [0,\infty)$ is a lower semi-continuous and symmetric function such that $c(x,x) = 0$ for all $x \in X$.
Then the Wasserstein distance\footnote{Traditionally the term ``Wasserstein distance'' is reserved for the case of $c = d^p$ for some metric $d$, but we extend this terminology for more general costs $c$ throughout this paper.} between two probability measures $P_1, P_2 \in \c{P}(X)$ is defined as 
\begin{align}
	\label{eq:wassersteinDistanceDefinition}
	W_c(P_1,P_2)
	= \inf_{\Lambda \in \Gamma(P_1,P_2)} \intg{}{c(x,y)}{\Lambda(x,y)}\,.
\end{align}
Intuitively the value $W_c(P_1,P_2)$ measures the total cost of moving a ``pile'' of mass distributed according to $P_1$ to a pile distributed according to $P_2$ if the cost of moving a unit of mass located in $x \in X$ to $y \in X$ is $c(x,y)$ \cite{kantorovitch1958translocation, bogachev2012monge, vershik2006kantorovich, villani2009optimal}.
If finite, the infimum in \eqref{eq:wassersteinDistanceDefinition} is attained \cite[Theorem 4.1]{villani2009optimal} and $\Gamma_c(P_1,P_2)$ will denote the set of minimizers.
We say that the cost $c$ satisfies the weak triangle inequality (other terminology: quasi-triangle-inequality or $K$-relaxed triangle inequality) if there exists a constant $K > 0$ such that $c(x,y) \leq K(c(x,z) + c(z,y))$ for all $x,y,z \in X$.
In this case the set $\c{P}_c(X) = \{P \in \c{P}(X) \mid \intg{}{c(x_0,\cdot)}{P} < \infty\}$ is independent of $x_0 \in X$ and convex.
The cost $c$ is called proper if it satisfies the weak triangle inequality and the sublevel sets $\{x \in X \mid c(x_0,x)\leq r\}$ are compact for all $r > 0$ and $x_0 \in X$.
If $d$ is a metric on $X$, then $W_p$ defined by $W_p(P_1,P_2) = W_c(P_1,P_2)^{1/p}$ with $c(x,y) = d(x,y)^p$ is a metric on the set $\c{P}_p(X) := \c{P}_c(X)$ of all distributions with finite $p$-th moment \cite{villani2009optimal,panaretos2020invitation}.
Further for $\hat{P} \in \c{P}_c(X)$ and $\rho > 0$ we will denote $\b{B}_\rho^c(\hat{P}) = \{P \in \c{P}_c(X) \mid W_c(P,\hat{P}) \leq \rho\}$ for the $W_c$-Wasserstein ball of radius $\rho$ around $\hat{P}$.
The Lebesgue spaces on the measure space $(X,P) = (X,\c{B}(X),P)$ are denoted by $L^p(X,P)$ and the bounded, real-valued continuous functions on a topological space $X$ by $C_b(X)$.
For a set $\c{A} \subseteq \c{P}(X)$ of distributions, $\convw \c{A}$ denotes the closed convex hull in the weak topology on $\c{P}(X)$.
The symbol $\1$ denotes the all-ones vector of suitable dimension and $\1_A$ for some set $A \subseteq X$ the indicator function of $A$.
The symbol $\c{S}_N$ denotes symmetric group on $N$ letters, while $\c{S}_\infty$ denotes the set of all bijections $\pi:\b{N} \to \b{N}$ that leave all but finitely many indices invariant.
Moreover, we always identify a $\pi \in \c{S}_N$ with the induced coordinate permutation $\pi:X^N \to X^N:x \mapsto (x_{\pi(k)})_{k=1}^N$.
Finally, for $X = \b{R}^d$ being the Euclidean space and $P \in \c{P}(X)$ we denote the mean, second moment and variance of $P$ by $\b{E}P$, $\b{E}(P^2)$ and $\operatorname{Var}(P)$ respectively.

\subsection{Wasserstein distributionally robust optimization} 
Consider now an uncertainty quantification problem \cite{kuhn2019wasserstein,shafieezadeh2023new} of the type
\begin{align}
	\label{eq:droGeneral}
	\sup_{P \in \b{B}_\rho^c(\hat{P})} \b{E}_{x \sim P} \ell(x)
\end{align}
where $\ell:X \to \b{R}$ is a Borel measurable loss functional and \eqref{eq:droGeneral} corresponds to the inner problem of \eqref{eq:stochasticOptimizationRobust} with $\ell = \ell(\theta,\cdot)$ for a fixed $\theta$.
To study the finiteness and existence of optimizers, define the set of functions with sublinear growth as
\begin{align*}
	\c{D} = \left\{\delta:[0,\infty) \to [0,\infty) \mid \forall r > 0:\, \delta(r) \leq r\,,\; \lim_{r \to \infty} \frac{\delta(r)}{r} = 0\right\}	
\end{align*} 
and consider the following function growth classes:
\begin{align*}
	\c{G}_c(X) &= \{\ell:X \to \b{R} \mid \exists C > 0\,, x_0 \in X\,:\; \ell \leq C(1+c(x_0,\cdot))\}\,,\\
	\c{G}_c^<(X) &= \{\ell:X \to \b{R} \mid \exists \delta \in \c{D}\,, x_0 \in X\,:\; \ell \leq C(1+\delta(c(x_0,\cdot)))\}\,.
\end{align*}
The following result establishes the well-posedness of the problem formulation in  \eqref{eq:droGeneral}.
\begin{theorem}
	\label{thm:droGeneralSolvability}
	For any Borel $\ell:X \to \b{R}$ the value of \eqref{eq:droGeneral} is finite if $c$ satisfies the weak triangle inequality, $\hat{P} \in \c{P}_c(X)$ and $\ell \in \c{G}_c(X)$.
	Moreover, if $c$ is proper, $\ell$ is upper semi-continuous and $\ell \in \c{G}_c^<(X)$, then the optimal value of \eqref{eq:droGeneral} is attained.
\end{theorem}
For the case of $c = d^p$ this result has been established in \cite{yue2022linear}, where the authors considered $\delta(r) = \min\{r^{\tilde{p}/p},r\}$ for some $\tilde{p} \in [1,p)$.
Now, even though the optimization problem \eqref{eq:droGeneral} is infinite-dimensional, it is convex, since the objective is linear and the constraint set $\b{B}_\rho^c(\hat{P})$ convex in the ambiguous distribution $P$.
The following theorem provides a strong duality result \cite{gao2023distributionally,blanchet2019quantifying,shafieezadeh2023new,zhang2024short}.
\begin{theorem}[{\cite[Theorem 1, Example 2]{zhang2024short}, \cite[Section 4.2]{blanchet2019quantifying}}]
	\label{thm:droGeneralDuality}
	Let $\ell \in L^1(X,\hat{P})$ and $\rho > 0$.
	Then the value of \eqref{eq:droGeneral} is equal to
	\begin{align*}
		\sup_{\bar{P} \in \b{B}_\rho^c(\hat{P})} \b{E}_{x \sim \bar{P}} \ell(x) = \inf_{\mu \geq 0} \mu \rho + \b{E}_{z \sim \hat{P}} \phi_\mu(z) 
	\end{align*}
	where the infimum is attained and $\phi_\mu:X \to \bar{\b{R}}: z \mapsto \sup_{x \in X} (\ell(x) - \mu c(x,z))$ is universally measurable\footnote{While this notion is not needed further in this paper, it is required to make sense of the expression $\b{E}_{z \sim \hat{P}} \phi_\mu(z)$, which is defined as the Lebesgue integral w.r.t. the unique extension of $P$ to the universal $\sigma$-algebra on $X$ \cite{bertsekas1996stochastic,bertsekas1978mathematical}. In general $\phi_\mu$ is not Borel-measurable.}.
\end{theorem}

The main importance of this theorem is that it reduces the infinite-dimensional problem \eqref{eq:droGeneral} to a finite-dimensional problem over a single scalar variable $\mu \geq 0$, provided that the expectation of $\phi_\mu$ w.r.t. $\hat{P}$ can be computed explicitly.

\section{Structured Uncertainty Quantification}
\label{sec:structuredDRO}

In this section, we are interested in the inner optimization problems of \eqref{eq:stochasticOptimizationRobustStructured}, namely
\begin{align}
	\label{eq:droStructured}
	\sup_{\bar{P} \in \c{W}} \b{E}_{x \sim \bar{P}} \ell(x)
	= \sup_{P \in \b{B}_\rho^c(\hat{P})} \intg{}{\ell(x)}{P^{\otimes N}(x)}
	= \sup_{\bar{P} \in \c{W}} \intg{}{\ell(x)}{\bar{P}(x)}\,, 
\end{align}
for some $N \in \b{N}$, Borel measurable $\ell:X^N \to \b{R}$ and 
\begin{align}
	\label{eq:structuredWassersteinAmbiguitySet}
	\mathcal{W} = \{P^{\otimes N} \mid P \in \mathbb{B}_\rho^c(\hat{P})\}\,.
\end{align}
In its stated form problem \eqref{eq:droStructured} is non-convex, since the set $\c{W}$ is non-convex.
The following subsections investigate different aspects of this problem in more detail.
We confine the proofs of the subsequent results to the appendix for expositional clarity.
Unless stated otherwise, we assume that the transportation cost $c$ is proper and $\hat{P} \in \c{P}_c(X)$ throughout. 

% \begin{figure}    
%     \resizebox{\textwidth}{!}{ 
%     % \input{Mathematics of Data Science/graphics.tikz}
% }
% \caption{Overview of the main results of Section 3.}
% \end{figure}
% \todo[inline]{MF: Fix figure at the end}

% \todo[inline]{AK: Do we need this list? Didn't we mention this in the introduction or contributions?  MF: it is not necessary, it was an attempt to give a preview of the structure of the section cause it is very long. But we can 1) remove it; 2) adapt it and move it to the contributions (to have more detailed contributions)}

\subsection{Finiteness and existence of optimizers}
\label{sec:structuredDROExistenceFiniteness}

First we delve into when \eqref{eq:droStructured} is well-defined, finite and attains its solution on $\c{W}$.
For $N \in \b{N}$ we write $c^N:X^N \times X^N \to [0,\infty):(x,y) \mapsto \sum_{i=1}^N c(x_i,y_i)$ for the $N$-lift of the transport cost $c$.
In some cases we will omit the superscript $N$ and write again $c$ instead of $c^N$ for brevity. 
%Note that the $N$-lift of $c = d^p$ is of the same form as well, namely $c^N = (d_N)^p$, where $d_N(x,y) = (\sum_{i=1}^N d(x_i,y_i)^p)^{1/p}$.
%Here $d_N$ is a metric which induces the product topology on $X^N$.

To start, we note that for $N \in \b{N}$ the set $\c{W}$ can be overestimated by an ordinary Wasserstein ball w.r.t. the lifted transportation cost.

\begin{lemma}
	\label{lem:wassersteinBallProductInclusion}
	For any transportation cost $c:X \times X \to [0,\infty]$, $\hat{P} \in \c{P}(X)$, $\rho \geq 0$ and $N \in \b{N}$ it holds that $P^{\otimes N} \in \b{B}_{N \rho}^{c^N}(\hat{P}^{\otimes N})$ for all $P \in \b{B}_\rho^c(\hat{P})$.
\end{lemma}
Thus 
\begin{align}
	\label{eq:structuredAmbigutySetOverestimateBasic}
	\c{W} \subseteq \b{B}_{N\rho}^{c^N}(\hat{P}^{\otimes N})\,. 
\end{align} 
and it always holds that
\begin{align}
	\label{eq:nonlinearDROUpperBoundBasic}
	\sup_{\bar{P} \in \c{W}} \b{E}_{x \sim \bar{P}} \ell(x)
	\leq \sup_{\bar{P} \in \b{B}_{N\rho}^{c^N}(\hat{P}^{\otimes N})} \b{E}_{x \sim \bar{P}} \ell(x)\,,
\end{align}
and we can use a Theorem \ref{thm:droGeneralSolvability} to conclude that the value of \eqref{eq:droStructured} is finite if $\ell \in \c{G}_{c^N}(X^N)$, or equivalently, 
\begin{align}
	\label{eq:productSpaceGrowthBoundSum}
	\ell(x) &\leq C\left(1+\sum_{k=1}^N c(x_0,x_k)\right) \text{\ \ for all\ \ } x \in X^N \text{\ for some\ } C > 0\,,\; x_0 \in X\,. 
\end{align}
However, because of the product structure of $\c{W}$ it is actually enough to require argument-wise growth conditions.
For this purpose define the following function growth classes
\begin{align*}
	&\hat{\c{G}}_{c,N}(X^N) \\
	&\quad = \left\{\ell:X^N \to \b{R} \mid \exists C > 0\,, x_0 \in X\,, \forall x \in X^N:\; \ell(x) \leq C\prod_{k=1}^N (1+c(x_0,x_k))\right\}\,,\\
	&\hat{\c{G}}_{c,N}^<(X^N) \\ 
	&\quad = \left\{\ell:X^N \to \b{R} \mid  \exists \delta \in \c{D}\,, x_0 \in X\,, \forall x \in X^N:\; \ell(x) \leq C\prod_{k=1}^N (1+\delta(c(x_0,x_k)))\right\}\,.
\end{align*}
Note that in general $\c{G}_{c^N}(X^N) \subsetneq \hat{\c{G}}_{c,N}(X^N)$ and $\c{G}_{c^N}^<(X^N) \subsetneq \hat{\c{G}}_{c,N}^<(X^N)$.
Then we have the following analogon of Theorem \ref{thm:droGeneralSolvability}.
\begin{theorem}
	\label{thm:droStructuredSolvability}
	For any Borel $\ell:X^N \to \b{R}$ the value of \eqref{eq:droGeneral} is finite if $c$ satisfies the weak triangle inequality, $\hat{P} \in \c{P}_c(X)$ and $\ell \in \hat{\c{G}}_{c,N}(X^N)$.
	Moreover, if $c$ is proper, $\ell$ is upper semi-continuous and $\ell \in \hat{\c{G}}_{c,N}^<(X^N)$, then the optimal value of \eqref{eq:droGeneral} is attained.
\end{theorem}

\subsection{A convex upper bound}
\label{sec:structuredDROConvexUpperBound}

In the previous section we have seen that \eqref{eq:nonlinearDROUpperBoundBasic} states an upper bound on \eqref{eq:droStructured}.
In general this bound will be conservative as the following example shows.
\begin{example}
	Let $X = \b{R}$, $N = 2$, $c(x_1,x_2) = \abs{x_1-x_2}^2$, $\hat{P} = \delta_0$ and $\ell(x_1,x_2) = \frac{1}{2}(x_1-x_2)^2$.
	Then the value of the structured problem is 
	\begin{align*}
		\sup_{P \in \b{B}_\rho^c(\hat{P})} \b{E}_{x \sim P^{\otimes 2}} \ell(x)
		= \sup_{P \in \b{B}_\rho^c(\delta_0)} \operatorname{Var}(P)
		= \sup_{\b{E}(P^2) \leq \rho} \b{E}(P^2) - (\b{E}P)^2
		\leq \rho\,,
	\end{align*}
	with $P = \frac{1}{2}\delta_{-\sqrt{\rho}} + \frac{1}{2}\delta_{\sqrt{\rho}}$ showing that in fact the latter bound is attained. 
	On the other hand, the value of the right hand side of \eqref{eq:nonlinearDROUpperBoundBasic} can be computed using Theorem \ref{thm:droGeneralDuality} and is given by
	\begin{align*}
		\sup_{\bar{P} \in \b{B}_{2\rho}^{c^2}(\delta_0^{\otimes 2})} \b{E}_{x \sim \bar{P}} \ell(x)
		= \inf_{\mu \geq 0} \left[ 2\rho \mu + \sup_{x \in \b{R}^2} \left(\frac{1}{2}(x_1-x_2)^2 - \mu\norm{x}^2\right)\right]
		= 2\rho\,.
	\end{align*}
\end{example}
The next example shows that the conservatism of this bound can be arbitrary large in a relative sense.
\begin{example}
	\label{ex:conservatismUnstructuredBound}
	Let $X = \b{R}$, $N = 2$, $c(x_1,x_2) = \abs{x_1-x_2}^2$, $\hat{P} = \frac{1}{2}\delta_1 + \frac{1}{2} \delta_{-1}$ and $\ell(x_1,x_2) = -x_1 x_2$.
	Then the value of the structured problem is 
	\begin{align*}
		\sup_{P \in \b{B}_\rho^c(\hat{P})} \b{E}_{x \sim P^{\otimes 2}} \ell(x)
		= \sup_{P \in \b{B}_\rho^c(\hat{P})} -(\b{E}P)^2 
		= 0\,.
	\end{align*}
	The value of the right hand side of \eqref{eq:nonlinearDROUpperBoundBasic} is, after some elementary calculations, given by
	\begin{align*}
		\sup_{\bar{P} \in \b{B}_{2\rho}^{c^2}(\hat{P}^{\otimes 2})} \b{E}_{x \sim \bar{P}} \ell(x)
		&= \inf_{\mu \geq 0} \left[ 2\rho \mu + \frac{1}{4} \sum_{\hat{\xi} \in \{\pm 1\}^2}\sup_{x \in \b{R}^2} \left(-x_1 x_2 - \mu\norm{x - \hat{\xi}}^2\right)\right] \\
		&= \inf_{\mu > \frac{1}{2}} \left[ 2 \rho \mu - \frac{2\mu}{1-4\mu^2}\right] \\
		%&= \rho \sqrt{1+\frac{1+\sqrt{1+8\rho}}{2\rho}} \left(1+\frac{2}{1+\sqrt{1+8\rho}}\right)
		&\geq \rho\,.
	\end{align*}
	Hence, while the value of the structured problem is constant, the value of the latter bound grows asymptotically linearly as the ball radius increases.
\end{example}

In view of the last two examples we turn to the question whether we can obtain better bounds on \eqref{eq:droStructured} than \eqref{eq:nonlinearDROUpperBoundBasic}.
This corresponds to obtaining tighter convex overestimations of the set $\c{W}$ than the one given by \eqref{eq:structuredAmbigutySetOverestimateBasic}.
Clearly, the best such overestimation is the convex hull $\operatorname{conv} \c{W}$ and indeed even more is true. 
\begin{lemma}
	\label{lem:convexHullClosedDRO}
	If $\ell \in L^1(X^N,P^{\otimes N})$ for all $P \in \b{B}_\rho^c(\hat{P})$, then
	\begin{align*}
		\sup_{\bar{P} \in \c{W}} \b{E}_{x \sim \bar{P}} \ell(x)
		= \sup_{\bar{P} \in \operatorname{conv} \c{W}} \b{E}_{x \sim \bar{P}} \ell(x)
		= \sup_{\bar{P} \in \convw \c{W}} \b{E}_{x \sim \bar{P}} \ell(x)\,,
	\end{align*}
	where $\convw \c{W}$ is the closure of $\operatorname{conv} \c{W}$ in the weak topology of $\c{P}(X)$.
\end{lemma}

The nontrivial part is the second equality, since the first equality is simply a consequence of the expectation being linear in the distribution.
Note that by Theorem \ref{thm:droStructuredSolvability} the condition $\ell \in L^1(X^N,P^{\otimes N})$ for all $P \in \b{B}_\rho^c(\hat{P})$ holds if e.g. $\ell \in \hat{\c{G}}_{c,N}(X^N)$.

While this shows that our original non-convex UQ problem can be reformulated as a convex problem, it is very difficult to exhibit an explicit description of the sets $\operatorname{conv} \c{W}$ or $\convw \c{W}$ that would make the latter computationally tractable.
Hence we proceed to search for convex sets that overestimate $\convw \c{W}$, while still being a subset of $\b{B}_{N\rho}^{c^N}(\hat{P}^{\otimes N})$ and allowing for strong duality in the sense of Theorem \ref{thm:droGeneralDuality}.
For this we explore \eqref{eq:droStructured} further and define
\begin{align}
	\label{eq:droStructuredBasicUpperBound}
	\operatorname{S}(\ell)
	= \sup_{\bar{P} \in \c{W}} \b{E}_{x \sim \bar{P}} \ell(x)\,, \quad
	\operatorname{U}(\ell)
	= \sup_{\bar{P} \in \b{B}_{N\rho}^{c^N}(\hat{P}^{\otimes N})} \b{E}_{x \sim \bar{P}} \ell(x)\,.
\end{align}
Then, \eqref{eq:nonlinearDROUpperBoundBasic} translates simply to $\operatorname{S}(\ell) \leq \operatorname{U}(\ell)$ for any Borel $\ell:X^N \to \b{R}$.
Note that $\operatorname{U}(\ell)$ is a UQ problem of the type \eqref{eq:droGeneral} and hence can be computed by using Theorem \ref{thm:droGeneralDuality}.

\subsection{A tighter convex upper bound}
\label{sec:structuredDROTighterConvexUpperBound}
While the bound \eqref{eq:nonlinearDROUpperBoundBasic} can be quite conservative, we can establish tighter bounds by using the following, general, principle:
Suppose that there exists some class $\c{F}$ of transformations $F:\b{R}^{X^N} \to \b{R}^{X^N}$ such that $\operatorname{id} \in \c{F}$, $F(\ell)$ is Borel measurable for any Borel measurable $\ell:X^N \to \b{R}$ and 
\begin{align}
	\label{eq:invariantTransformations}
	\operatorname{S}(F(\ell))
	= \operatorname{S}(\ell) \text{\ \ for all\ \ } F \in \c{F} \text{\ and Borel\ } \ell:X^N \to \b{R}\,.
\end{align}
Then, for any Borel $\ell:X^N \to \b{R}$, it holds that 
\begin{align*}
	\operatorname{S}(\ell) 
	= \operatorname{S}(F(\ell))
	\leq \operatorname{U}(F(\ell))
\end{align*}
and hence
\begin{align}
	\label{eq:nonlinearDROUpperBoundOptimization}
	\operatorname{S}(\ell) \leq \inf_{F \in \c{F}} \operatorname{U}(F(\ell))
\end{align}
is a potentially tighter upper bound on $\operatorname{S}(\ell)$ than $\operatorname{U}(\ell)$.
Similar arguments have been used in \cite{packard1993complex} to obtain convex upper bounds on the structured singular value in the domain of control theory.
To obtain such a class $\c{F}$ for our problem, we make the following, crucial observation: If $\pi \in \c{S}_N$ is a permutation, then the transformation $F_\pi:\ell \mapsto \ell_\pi$ with $\ell_\pi(x) = \ell(\pi(x))$, where $\pi(x) = (x_{\pi(1)},\ldots,x_{\pi(N)})$ for $x \in X^N$, satisfies
\begin{align*}
	\operatorname{S}(F_\pi(\ell)) 
	= \sup_{\bar{P} \in \c{W}} \b{E}_{x \sim \bar{P}} \ell_\pi(x)
	= \sup_{P \in \b{B}_\rho^c(\hat{P})} \b{E}_{x \sim P^{\otimes N}} \ell_\pi(x)
	= \sup_{P \in \b{B}_\rho^c(\hat{P})} \b{E}_{x \sim P^{\otimes N}} \ell(x)
	= \operatorname{S}(\ell)\,,
\end{align*}
since $\pi \# P^{\otimes N} = P^{\otimes N}$ for any permutation $\pi$.
Moreover, since $\b{E}_{x \sim \bar{P}} \ell$ is linear in $\ell$, it follows that \eqref{eq:invariantTransformations} holds for
\begin{align}
	\label{eq:permutationConvexSet}
	\c{F}
	= \operatorname{conv} \{F_\pi \mid \pi \in \c{S}_N\}
	= \left\{F_\alpha = \sum_{\pi \in \c{S}_N} \alpha_\pi F_\pi \mid \alpha \in \Delta^{\c{S}_N} \right\}\,,
\end{align}
where $\Delta^{\c{S}_N} = \{\alpha = (\alpha_\pi)_{\pi \in \c{S}_N} \in [0,1]^{\c{S}_N} \mid \sum_{\pi \in \c{S}_N} \alpha_\pi = 1\}$ is the standard simplex in the $N!$-dimensional Euclidean space. 
This implies that \eqref{eq:nonlinearDROUpperBoundOptimization} holds for the above class $\c{F}$ of transformations.
In this specific case we can actually establish the following 
\begin{lemma}
	\label{lem:nonlinearDROUpperBoundOptimization}
	For any Borel measurable $\ell:X^N \to \b{R}$ it holds that 
	\begin{align*}
		\inf_{F \in \c{F}} \operatorname{U}(F(\ell))
		= \min_{\alpha  \in \Delta^{\c{S}_N}} \operatorname{U}(F_\alpha(\ell))
		= \operatorname{U}(\ell_{\operatorname{sym}})\,,
	\end{align*}
	where the \emph{symmetrization} of a function is defined by
	\begin{align}
		\label{eq:symmetrizationFunction}
		\ell_{\operatorname{sym}}(x) = \frac{1}{N!}\sum_{\pi \in \c{S}_N} \ell(\pi(x))\,, 
	\end{align}
	i.e. the infimum in \eqref{eq:nonlinearDROUpperBoundOptimization} is attained in the uniform coefficient $\alpha = (1/N!,\ldots,1/N!) \in \Delta^{\c{S}_N}$.
\end{lemma}

Thus instead of optimizing over the class $\c{F}$ as in \eqref{eq:nonlinearDROUpperBoundOptimization}, we can replace the latter by the single upper bound
\begin{align}
	\label{eq:nonlinearDROUpperBoundSymmetrization}
	\operatorname{S}(\ell) \leq \operatorname{U}(\ell_{\operatorname{sym}})\,.
\end{align}

\begin{example}
	\label{ex:symmetrization}
	Let $X = \b{R}$, $N = 2$, $c(x_1,x_2) = \abs{x_1-x_2}^2$, $\hat{P} = \frac{1}{2}\delta_1 + \frac{1}{2} \delta_{-1}$ and $\ell(x_1,x_2) = -2x_1^2 - 2x_1 x_2$.
	Then the value of the structured problem is given by
	\begin{align*}
		\operatorname{S}(\ell)
		= \operatorname{S}(\ell_{\operatorname{sym}})
		= \sup_{P \in \b{B}_\rho^c(\hat{P})} -2(\b{E}(P^2) + (\b{E}P)^2)
		= -2(1-\min(\sqrt{\rho},1))^2
		\leq 0\,.
	\end{align*}
	The value of the unstructured bound is, after some elementary calculations,
	\begin{align*}
		\operatorname{U}(\ell)
		= \inf_{\mu > \sqrt{2}-1} 2\rho \mu + \frac{2\mu(1-\mu)}{\mu(\mu+2)-1} 
		\geq (\sqrt{2}-1)2\rho - 2\,, 
	\end{align*}
	while the value of the symmetrized bound is
	\begin{align*}
		\operatorname{U}(\ell_{\operatorname{sym}})
		= 2(2\sqrt{2\min(\rho,1/2)} - 2\min(\rho,1/2)-1) \leq 0\,.
	\end{align*}
	See Figure \ref{fig:exampleSymmetrizationAndLifting1} for a graphical illustration.
\end{example}

\begin{figure}
	\centering
	\subfigure[Objective values for Example \ref{ex:symmetrization}]{%
		\includegraphics[width=0.46\textwidth]{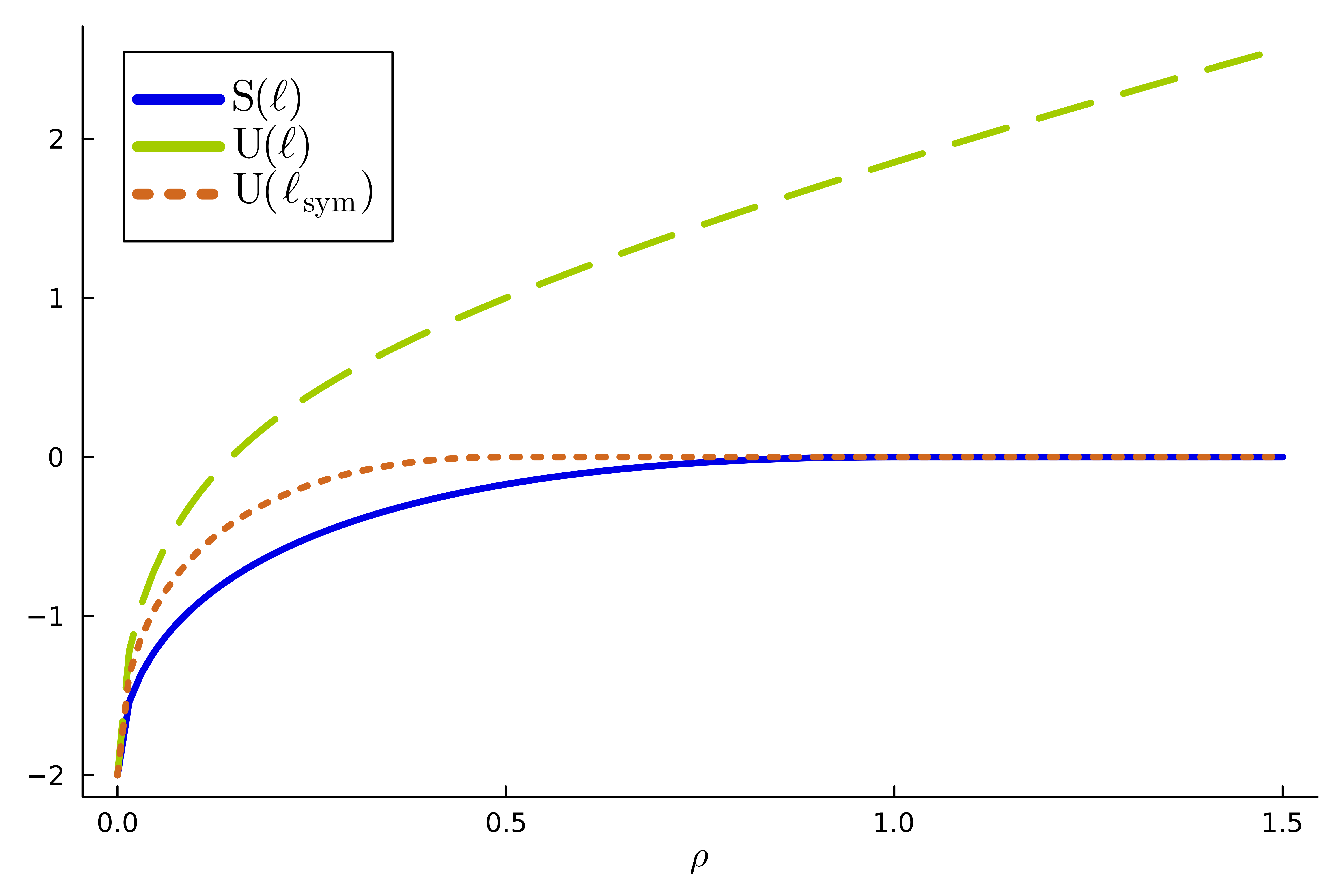}
		\label{fig:exampleSymmetrizationAndLifting1}
	}
	\subfigure[Objective values for Example \ref{ex:liftedRelaxation}]{%
		\includegraphics[width=0.46\textwidth]{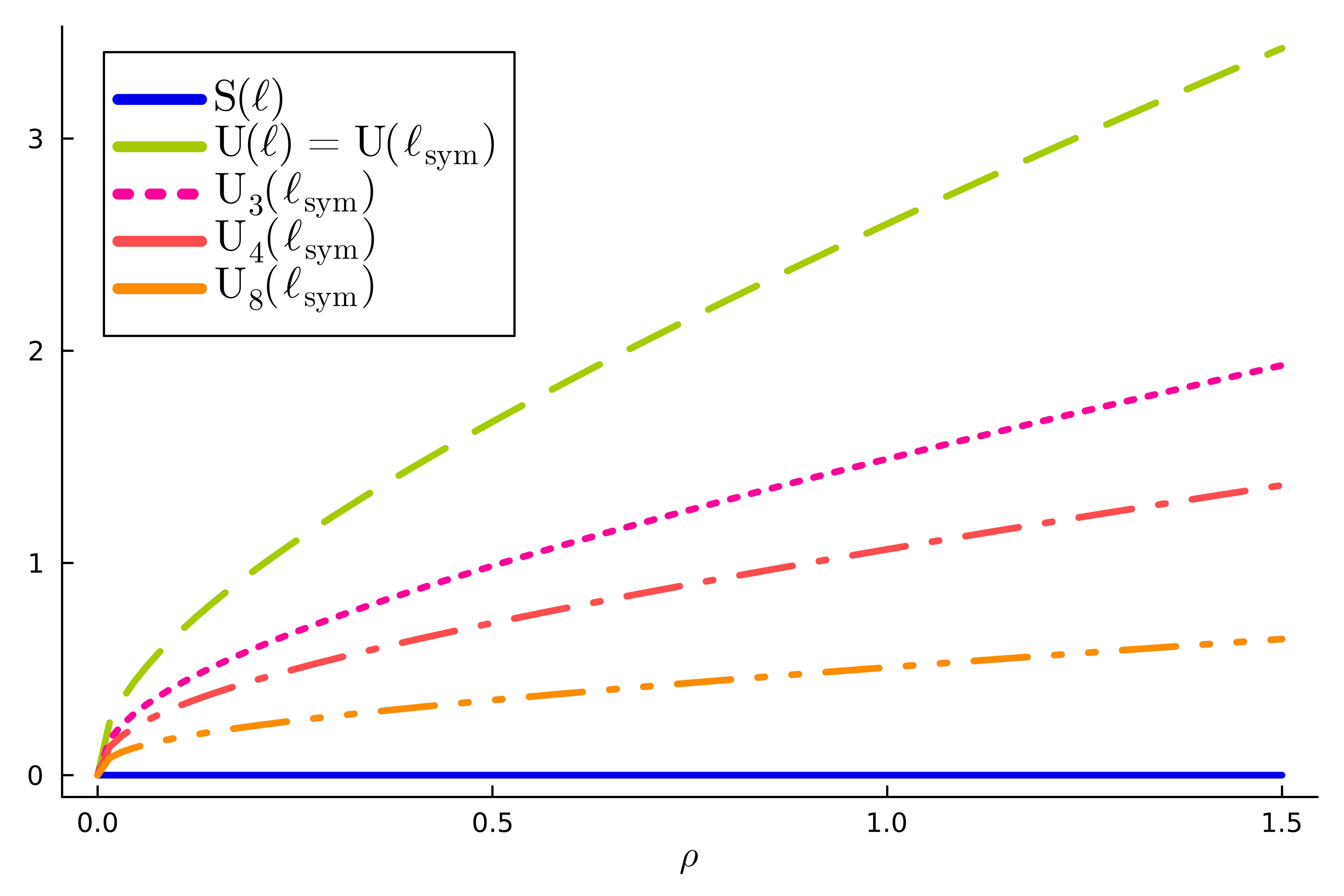}
		\label{fig:exampleSymmetrizationAndLifting2}
	}
	\caption{Objective values for Examples \ref{ex:symmetrization} and \ref{ex:liftedRelaxation}}
	\label{fig:exampleSymmetrizationAndLifting}
\end{figure}
For the purpose of interpreting the upper bound $\operatorname{U}(\ell_{\operatorname{sym}})$ in terms of an overestimation of the set $\convw \c{W}$ 
we need the following
\begin{definition}
	A distribution $\bar{P} \in \c{P}(X^N)$ is called \emph{symmetric} (other terminology: exchangeable or permutation-invariant) if for any permutation $\pi \in \c{S}_N$ it holds that $\pi \# \bar{P} = \bar{P}$.
	The set of symmetric distributions on $X^N$ is denoted by $\c{P}_{\operatorname{sym}}(X^N)$.
\end{definition} 
One important example are powers of distributions $\bar{P} = P^{\otimes N}$, which are always symmetric.
Additionally and in analogy to \eqref{eq:symmetrizationFunction} let us write
\begin{align}
	\label{eq:symmetrizationMeasure}
	\bar{P}_{\operatorname{sym}} := \frac{1}{N!}\sum_{\pi \in \c{S}_N} \pi\#\bar{P}
\end{align}
for the \emph{symmetrization} of the distribution $\bar{P}$.
This yields the following
\begin{theorem}
	\label{thm:symmetrizationUpperBound}
	Suppose that $\ell \in \c{G}_{c^N}(X^N)$ is Borel measurable.
	Then
	\begin{align}
		\label{eq:symmetrizationUpperBound}
		\begin{aligned}
			\operatorname{U}(\ell_{\operatorname{sym}})
			= \sup_{\bar{P} \in \b{B}_{N\rho}^{c^N}(\hat{P}^{\otimes N})_{\operatorname{sym}}} \intg{}{\ell}{\bar{P}}\,,
		\end{aligned}
	\end{align}
	where 
	\begin{align}
		\label{eq:symmetricSet}
		\b{B}_{N\rho}^{c^N}(\hat{P}^{\otimes N})_{\operatorname{sym}} 
		:= \{\bar{P}_{\operatorname{sym}} \mid \bar{P} \in \b{B}_{N\rho}^{c^N}(\hat{P}^{\otimes N})\}
		= \b{B}_{N\rho}^{c^N}(\hat{P}^{\otimes N}) \cap \c{P}_{\operatorname{sym}}(X^N)\,,
	\end{align}
\end{theorem}
From \eqref{eq:symmetrizationUpperBound} we make two observations:
Firstly, the upper bound $\operatorname{U}(\ell_{\operatorname{sym}})$ corresponds to optimizing over the set \eqref{eq:symmetricSet}, which contains $\convw \c{W}$, as it is weakly closed, convex and contains $\c{W}$.
Secondly, we see that even though the set $\b{B}_{N\rho}^{c^N}(P^{\otimes N})_{\operatorname{sym}}$ is not a Wasserstein ball itself (it is the intersection of the Wasserstein ball $\b{B}_{N\rho}^{c^N}(P^{\otimes N})$ with $\c{P}_{\operatorname{sym}}(X^N)$), \emph{one can still evaluate} the corresponding UQ problem
\begin{align}
	\label{eq:droSymmetric}
	\sup_{\bar{P} \in \b{B}_{N\rho}^{c^N}(P^{\otimes N})_{\operatorname{sym}}} \intg{}{\ell}{\bar{P}}
\end{align}
using Theorem \ref{thm:droGeneralDuality} by evaluating $\operatorname{U}(\ell_{\operatorname{sym}})$, i.e. by substituting $\ell$ with $\ell_{\operatorname{sym}}$ and replacing $\b{B}_{N\rho}^{c^N}(P^{\otimes N})_{\operatorname{sym}}$ by $\b{B}_{N\rho}^{c^N}(P^{\otimes N})$ in \eqref{eq:droGeneral}.

\subsection{Sequence of convex relaxations}
\label{sec:structuredDROSequenceOfConvexRelaxations}

It is interesting to ask now whether we can find even tighter overestimations of $\convw \c{W}$ than $\b{B}_{N\rho}^{c^N}(P^{\otimes N})_{\operatorname{sym}}$.
In view of \eqref{eq:symmetricSet} the latter question is equivalent to determining a property that distinguishes powers of measures in the set of all symmetric measures.
The crucial observation is that powers of powers of measures are again powers of measures and hence symmetric, while the same does not have to be true for general symmetric measures.
This latter viewpoint hints us to lift our problem and consider instead of \eqref{eq:droStructured} the problem
\begin{align}
	\label{eq:droGeneralLifted}
	\sup_{P \in \b{B}_\rho^c(\hat{P})} \intg{}{\ell\circ \operatorname{pr}_{1:N}^M(x)}{P^{\otimes M}(x)}\,,
\end{align}
where $N \leq M \leq \infty$ is a ``lifting parameter'' and $\operatorname{pr}_{1:N}^M:X^M \to X^N$ is the projection onto the first $N$ components, i.e. in \eqref{eq:droGeneralLifted} we see the function $\ell$ on $X^N$ as a function $\hat{\ell} = \ell \circ \operatorname{pr}_{1:N}^M$ on $X^M$.
Analogously we define the quantities
\begin{align*}
	\operatorname{S}_M(\hat{\ell})
	:= \sup_{P \in \b{B}_\rho^c(\hat{P})} \b{E}_{x \sim P^{\otimes M}} \hat{\ell}(x)
	\leq 
	\operatorname{U}_M(\hat{\ell})
	:= \sup_{\bar{P} \in \b{B}_{M\rho}^{c^M}(\hat{P}^{\otimes M})} \b{E}_{x \sim \bar{P}} \hat{\ell}(x)\,.
\end{align*}
Note that \eqref{eq:droGeneralLifted} has the same optimal value as \eqref{eq:droStructured}, since integrating w.r.t. extra factors $P$ does not change the value of the integral, i.e. $\operatorname{S}_N(\ell) = \operatorname{S}_M(\ell \circ \operatorname{pr}_{1:N}^M)$ for any $M \geq N$.
In view of \eqref{eq:nonlinearDROUpperBoundSymmetrization} it then follows that
\begin{align*}
	\operatorname{S}(\ell)
	= \operatorname{S}_N(\ell)
	= \operatorname{S}_M(\hat{\ell})
	\leq \operatorname{U}_M(\hat{\ell}_{\operatorname{sym}}) \text{\ \ for\ \ } N \leq M < \infty\,.
\end{align*}
Note that here $\hat{\ell}_{\operatorname{sym}}$ denotes the symmetrization of $\ell$ viewed as a function $X^M \to \b{R}$ and thus itself depends on $M$.
To summarize, for each $M \geq N$ we obtain an upper bound
\begin{align}
	\label{eq:droStructuredRelaxationSequence}
	\operatorname{U}_M^{\operatorname{sym}}(\ell) := \operatorname{U}_M((\ell\circ \operatorname{pr}_{1:N}^M)_{\operatorname{sym}}) 
\end{align}
on $\operatorname{S}(\ell)$.
It then holds that\footnote{the use of the notation $\operatorname{U}_\infty^{\operatorname{sym}}(\ell)$ is justified by Lemma \ref{lem:nestedRelaxationSets} stated below.}
\begin{align}
	\label{eq:droStructuredRelaxationSequenceInfimum}
	\operatorname{S}(\ell) 
	\leq \operatorname{U}_\infty^{\operatorname{sym}}(\ell) := \inf_{M \geq N} \operatorname{U}_M^{\operatorname{sym}}(\ell)\,.
\end{align} 
We stress again that, contrary to the original problem $\operatorname{S}(\ell)$, \emph{we can evaluate each} $\operatorname{U}_M^{\operatorname{sym}}(\ell)$ via Theorem \ref{thm:droGeneralDuality} by using the relation \eqref{eq:droStructuredRelaxationSequence} (see also Section \ref{sec:structuredDROPolyhedralLoss}).
For notational brevity we also set $\operatorname{U}_0^{\operatorname{sym}}(\ell) := \operatorname{U}(\ell)$ for the non-symmetrized bound defined as the right side of \eqref{eq:nonlinearDROUpperBoundBasic}.
The following theorem establishes when $\operatorname{U}_M^{\operatorname{sym}}(\ell)$ is finite and attained.
\begin{theorem}
	\label{thm:droRelaxedSolvability}
	The value $\operatorname{U}_M^{\operatorname{sym}}(\ell)$ is finite if $\ell \in \c{G}_{c^N}(X^N)$. 
	Additionally, if $\ell \in \c{G}_{c^N}^<(X^N)$, then $\operatorname{U}_M^{\operatorname{sym}}(\ell)$ is attained, i.e. there exists some $\bar{P} \in \b{B}_{M\rho}^{c^M}(\hat{P}^{\otimes M})_{\operatorname{sym}}$ such that $\intg{}{\ell\circ\operatorname{pr}_{1:N}^M}{\bar{P}} = \operatorname{U}_M^{\operatorname{sym}}(\ell)$.
\end{theorem}

\begin{example}
	\label{ex:liftedRelaxation}
	Let $X$, $N$, $c$, $\hat{P}$ and $\ell$ be as in Example \ref{ex:conservatismUnstructuredBound}.
	After some elementary calculations (see Appendix) we obtain
	\begin{align*}
		\operatorname{U}_M^{\operatorname{sym}}(\ell)
		&= \inf_{\mu > \frac{1}{M(M-1)}} M(\rho-1) \mu + \frac{(M-1)M^2\mu^2}{(M-1)M\mu-1}\left(1 - \frac{1}{(M-1)(1+M\mu)}\right) \\
		&\leq \inf_{\mu > \frac{1}{M(M-1)}} M(\rho-1) \mu + \frac{(M-1)M^2\mu^2}{(M-1)M\mu-1} \\
		&= \frac{(\sqrt{\rho}+1)^2}{M-1}\,,
	\end{align*}
	i.e. for a fixed $\rho > 0$ it follows $\lim_{M \to \infty} \operatorname{U}_M^{\operatorname{sym}}(\ell) = 0 = \operatorname{S}(\ell)$. 
	See Figure \ref{fig:exampleSymmetrizationAndLifting2} for a graphical illustration.
\end{example}

In view of Example \ref{ex:liftedRelaxation} it is now natural to ask under which conditions the inequality in \eqref{eq:droStructuredRelaxationSequence} is an equality, i.e. when the sequence of \emph{relaxations} is tight and there is no relaxation gap.
For this we will investigate this gap in the next section in more detail.

\subsection{Relaxation gap}
\label{sec:structuredDRORelaxationGap}

To understand the gap between $\operatorname{S}(\ell)$ and $\operatorname{U}_\infty^{\operatorname{sym}}(\ell)$ it is helpful to first understand the lifting relaxation $\operatorname{U}_M^{\operatorname{sym}}(\ell)$ in terms of overestimations of $\convw \c{W}$.
For this purpose we unfold the definitions to obtain
\begin{align*}
	\begin{gathered}
		\operatorname{U}_M^{\operatorname{sym}}(\ell)
		= \sup_{\bar{P} \in \b{B}_{M\rho}^{c^M}(\hat{P}^{\otimes M})} \intg{}{(\ell\circ \operatorname{pr}_{1:N}^M)_{\operatorname{sym}}}{\bar{P}}
		= \sup_{\bar{P} \in \b{B}_{M\rho}^{c^M}(\hat{P}^{\otimes M})} \intg{}{\ell\circ \operatorname{pr}_{1:N}^M}{\bar{P}_{\operatorname{sym}}} \\
		= \sup_{\bar{P} \in \b{B}_{M\rho}^{c^M}(\hat{P}^{\otimes M})} \intg{}{\ell}{(\operatorname{pr}_{1:N}^M \# \bar{P}_{\operatorname{sym}})} 
		= \sup_{\bar{P} \in \c{U}_M} \intg{}{\ell}{\bar{P}}\,,
	\end{gathered}
\end{align*}
with
\begin{align}
	\label{eq:relaxationSetFiniteIndex}
	\c{U}_M := \{\operatorname{pr}_{1:N}^M \# \bar{P} \mid \bar{P} \in \b{B}_{M\rho}^{c^M}(\hat{P}^{\otimes M})_{\operatorname{sym}}\} \subseteq \c{P}(X^N)\,,
\end{align}
i.e. we overestimate $\convw \c{W}$ by sets of the form \eqref{eq:relaxationSetFiniteIndex}, which consist of all marginal distributions onto the first $N$ factors of symmetric distributions in $\b{B}_{M\rho}^{c^M}(\hat{P}^{\otimes M})$.
The following lemma establishes some properties of the sets \eqref{eq:relaxationSetFiniteIndex}. 
\begin{lemma}
	\label{lem:nestedRelaxationSets}
	For each $M \geq N$
	\begin{enumerate}
		\item[(i)] the set $\c{U}_M$ is convex and weakly compact
		
		\item[(ii)] it holds that $\convw\c{W} \subseteq \c{U}_M$
		
		\item[(iii)] it holds that $\c{U}_{M+1} \subseteq \c{U}_M$
	\end{enumerate}
\end{lemma}
%
%\todo[inline]{Add projection of a 2D-Gaussian distribution: What covariance matrices are allowed for Gaussians inside $\c{U}_M$?}
%
%\begin{example}
%	To obtain a better intuition about the sets $\c{U}_M$, let us look at its intersections with $$	
%\end{example}

%\begin{figure}
%	\centering
%	\begin{subfigure}[t]{0.49\textwidth}
%		\centering	\includegraphics[width=1.3\textwidth]{../graphics/amb_set_comparison.png}
%		\caption{Sets $\c{U}_M$ for $M=2,3,4$ as well as $\c{W}$ and $\operatorname{conv}\c{W}$}	\label{fig:amb_set_comparison1}
%	\end{subfigure}
%	\begin{subfigure}[t]{0.49\textwidth}
%		\centering		\includegraphics[width=1.3\textwidth]{../graphics/amb_set_comparison_infty_less.png}
%		\caption{Sets $\c{U}_M$ for $M=2,3,4$ and $M = \infty$}		\label{fig:amb_set_comparison2}
%	\end{subfigure}
%	\caption{Space of probabilities $p_{11}, p_{12}, p_{21}$ for which the distribution $\bar{P} = p_{11} \delta_{(1,1)} + p_{12} \delta_{(1,-1)} + p_{21} \delta_{(-1,1)} + p_{22} \delta_{(-1,-1)} \in \c{P}(\b{R}^2)$ with $p_{22} = 1-p_{11} - p_{12} - p_{21}$ belongs to certain subsets sets of $\b{B}_{2\rho}^c(P^{\otimes 2})$.
%		Here we have set $X = \b{R}$, $\rho = 1$, $c(x,y) = \abs{x-y}^2$ and $P = \frac{1}{2}\delta_1 + \frac{1}{2}\delta_{-1}$ and $N = 2$.
%		Note that because every distribution in the above sets is symmetric, it follows that $p_{21} = p_{12}$ and hence we only plot $p_{12} + p_{21} = 2p_{12}$.}
%	\label{fig:amb_set_comparison}
%\end{figure}

As an immediate corollary, we obtain that the sequence of upper bounds $\operatorname{U}_M^{\operatorname{sym}}(\ell)$ is \emph{nonincreasing}.
\begin{corollary}
	\label{cor:decreasingRelaxationSequence}
	For all $M \in \b{N}$ with $M \geq N$ it holds that $\operatorname{U}_{M+1}^{\operatorname{sym}}(\ell) \leq \operatorname{U}_M^{\operatorname{sym}}(\ell)$. 
	In particular $\operatorname{U}_\infty^{\operatorname{sym}}(\ell) = \lim_{M \to \infty} \operatorname{U}_M^{\operatorname{sym}}(\ell)$.
\end{corollary}
Since we are interested in the latter limit behavior, we consider the following set
\begin{align*}
	\c{U}_{\infty}
	= \bigcap_{M \geq N} \c{U}_M\,,
\end{align*}
which is convex, weakly compact and contains $\convw \c{W}$, as well as the corresponding optimization problem
\begin{align}
	\label{eq:droStructuredLimitSet}
	\hat{\operatorname{U}}_\infty^{\operatorname{sym}}(\ell) := \sup_{\bar{P} \in \c{U}_\infty} \b{E}_{x \sim \bar{P}} \ell(x)\,.
\end{align}
%Despite having a similar interpretation, $\hat{\operatorname{U}}_\infty^{\operatorname{sym}}(f)$ and $\operatorname{U}_\infty^{\operatorname{sym}}(f)$ are \emph{potentially different}, since the latter is the defined as the limit of the relaxation sequence that optimizes over the set $\c{U}_M$ at each lifting parameter $M$, while the former is defined as an optimization problem over the intersection of all $\c{U}_M$ \emph{simultaneously}.
%Now, the question is how both of the latter quantities relate to our relaxation sequence.
Since $\convw \c{W} \subseteq \c{U}_\infty \subseteq \c{U}_M$ for all $M \geq N$, it follows that
\begin{align*}
	\operatorname{S}(\ell)
	\leq \hat{\operatorname{U}}_\infty^{\operatorname{sym}}(\ell) 
	\leq \sup_{\bar{P} \in \c{U}_M} \b{E}_{x \sim \bar{P}} \ell(x)
	= \operatorname{U}_M^{\operatorname{sym}}(\ell) 
	\text{\ \ for all\ } M \geq N\,,
\end{align*}
and thus always $\hat{\operatorname{U}}_\infty^{\operatorname{sym}}(\ell) \leq \operatorname{U}_\infty^{\operatorname{sym}}(\ell)$.
The next theorem gives conditions when the latter inequality is an equality.
\begin{theorem}
	\label{thm:limitSetContinuous}
	If $\ell:X^N \to \b{R}$ is upper semi-continuous and $\ell \in \c{G}_{c^N}^<(X^N)$, then
	\begin{align*}
		\hat{\operatorname{U}}_\infty^{\operatorname{sym}}(\ell) = \operatorname{U}_\infty^{\operatorname{sym}}(\ell)\,.
	\end{align*}
\end{theorem}
In view of Theorem \ref{thm:limitSetContinuous} it is interesting to investigate $\c{U}_\infty$ and the optimization problem \eqref{eq:droStructuredLimitSet} more in detail. 
We have the following alternative description of the latter set:
\begin{lemma}
	\label{lem:relaxationSetAlternativeDescription}
	It holds that
	\begin{align*}
		\c{U}_{\infty}
		= \{\operatorname{pr}_{1:N} \# \bar{P} \mid \bar{P} \in \hat{\c{W}}_\infty\}\,,
	\end{align*}
	where
	\begin{align*}
		\hat{\c{W}}_\infty
		= \{\bar{P} \in \c{P}_{\operatorname{sym}}(X^\infty) \mid \forall M \in \b{N}\,:\; \operatorname{pr}_{1:M}\#\bar{P} \in \b{B}_{M\rho}^{c^M}(\hat{P}^{\otimes M})\}\,.
	\end{align*}
\end{lemma}
%Figure \ref{fig:amb_set_comparison2} provides a visualization of a projection of the set $\c{U}_{\infty}$.
Hence we can characterize the limiting behavior of the overestimation
\begin{align*}
	\convw \c{W}
	\subseteq \c{U}_{\infty}
	= \bigcap_{M \geq N} \c{U}_M
\end{align*}
by the first $N$ marginal distributions of all symmetric distributions on the infinite product space $X^\infty$ with the corresponding first $M$ marginals belonging to a Wasserstein ball with radius growing linearly in $M$ and radius $\rho$.
The question is now of how the optimization over the sets $\convw \c{W}$ and $\c{U}_{\infty}$ can be related.
For this purpose note that
\begin{align}
	\label{eq:relaxationOriginalSetConvexHull}
	\convw \c{W} = \operatorname{pr}_{1:N} \# \c{W}_\infty\,, \qquad 
	\c{U}_{\infty} = \operatorname{pr}_{1:N} \# \hat{\c{W}}_\infty\,,
\end{align}
with $\c{W}_\infty = \convw\{P^{\otimes \infty} \mid P \in \b{B}_\rho^c(\hat{P})\}$.
As a main result of this paper we show the following Theorem \ref{thm:mainRelaxedSetRepresentation} that relates the sets $\c{W}_\infty$ and $\hat{\c{W}}_\infty$.\footnote{In the following theorem we use the concept of mixtures of probability measures: If $X$ is a Polish space, then so is $\c{P}(X)$ \cite[15.15 Theorem]{guide2006infinite}.
	Thus it is reasonable to consider $\c{P}(\c{P}(X))$, the set of probability measures $\nu$ defined on the set of Borel sets of $\c{P}(X)$ with the weak topology.
	If we have any such measure $\nu$ and $F:\c{P}(X) \to \c{P}(Y)$ is a Borel measurable map, with $Y$ being another Polish space, then one can define a measure $\bar{P}$ on $Y$ by
	\begin{align*}
		\bar{P}(A) 
		= \intg{\c{P}(X)}{F(\tilde{P})(A)}{\nu(\tilde{P})} \text{\ \ for Borel\ } A \subseteq Y\,.
	\end{align*}
	We call this measure the $\nu$-mixture of $F$ and write $\intg{}{F(\tilde{P})}{\nu(\tilde{P})} = \bar{P}$.
	In the result that follows we pick the Polish space $Y = X^\infty$ and $F(P) = P^{\otimes \infty}$ and point the reader to the appendix for further technical considerations.}
\begin{theorem}[Main result]
	\label{thm:mainRelaxedSetRepresentation}
	Let $c$ be a proper transportation cost, $\rho > 0$ and $\hat{P} \in \c{P}_c(X)$.
	Then
	\begin{align*}
		\c{W}_\infty &= \left\{\intg{}{P^{\otimes \infty}}{\nu(P)} \mid \nu \in \c{P}(\c{P}(X))\,,\; W_c(\hat{P},\cdot) \leq \rho \text{\ $\nu$-almost-surely}  \right\}\,, \\
		\hat{\c{W}}_\infty &= \left\{\intg{}{P^{\otimes \infty}}{\nu(P)} \mid \nu \in \c{P}(\c{P}(X))\,,\; \b{E}_{P \sim \nu} W_c(\hat{P},P) \leq \rho \right\}\,.
	\end{align*}
\end{theorem}
Let us reformulate in words what Theorem \ref{thm:mainRelaxedSetRepresentation} is saying:
First, it states that both sets that appear in \eqref{eq:relaxationOriginalSetConvexHull} can be expressed as mixtures of i.i.d. product distributions.
This mixture corresponds to first drawing randomly the distribution $P$ according to $\nu$ and then to randomly draw countably many elements from $X$ independently according to $P$.
%For instance, suppose we have a box $B$ of fair and unfair coins, i.e. distributions on $X = \{H,T\}$ with $H$ being heads and $T$ being tails. 
%Then $\bar{P}$ corresponds to first picking a random coin from $B$ according to $\nu$ and then throwing this coin infinitely many times and writing down the results. 
Furthermore, Theorem \ref{thm:mainRelaxedSetRepresentation} says that the difference between $\c{W}_\infty$ and $\hat{\c{W}}_\infty$ is in the mixture distribution $\nu$ only.
For $\c{W}_\infty$ the mixture is required to be concentrated in the Wasserstein ball $\b{B}_\rho^c(\hat{P})$ around $\hat{P}$ with radius $\rho$, while distributions drawn according to the mixture in $\hat{\c{W}}_\infty$ are allowed to be outside of $\b{B}_\rho^c(\hat{P})$ as long as in \emph{expectation} the Wasserstein distance to $\hat{P}$ is less than $\rho$. 
The proof of Theorem \ref{thm:mainRelaxedSetRepresentation} hinges on Theorem \ref{thm:wasserseinProductIntegral} given in the appendix, which establishes an interesting relationship between the Wasserstein distance of a mixture distribution and the mixture of the Wasserstein distances of its constituents.\\
%Going back to our coin example, this would correspond to selecting a fixed coin $P$ (say the fair coin with $P(H) = P(T) = \frac{1}{2}$) in the box $B$ and require that we randomly draw coins $\tilde{P}$ that are not too far from $P$ in a certain sense. 
%For $\c{W}_\infty$ this ``sense'' would the hard constraint on the Wasserstein distance to $P$ being less than $\rho$, while for $\hat{\c{W}}_\infty$ this ``sense'' would mean that on \emph{average} the Wasserstein distance to $P$ is less than $\rho$.\\
\\
Let us look at Theorem \ref{thm:mainRelaxedSetRepresentation} from the angle of the so-called de Finetti's theorem \cite[Theorem~10.10.19]{bogachev2007measure}:
\begin{theorem}[De Finetti]
	\label{thm:deFinetti}
	Let $X$ be a Polish space. 
	Then
	\begin{align*}
		\c{P}_{\operatorname{sym}}(X^\infty) = \left\{\intg{}{P^{\otimes \infty}}{\nu(P)} \mid \nu \in \c{P}(\c{P}(X))\right\}\,.
	\end{align*}
	Moreover, for each $\bar{P} \in \c{P}_{\operatorname{sym}}(X^\infty)$ the representing mixture $\nu$ is unique. 
\end{theorem} 
This theorem shows that the set of all symmetric measures on the infinite product space $X^\infty$ is precisely the set of mixtures of product distributions.
In other words, any sequence of random variables that is \emph{exchangeable}, i.e. its distribution is unchanged under permutations of its elements, is the mixture of i.i.d. sequences of random variables.
%Observe that in the above theorem the only nontrivial inclusion is $\subseteq$, since clearly mixtures of i.i.d. sequences are exchangeable.
%Also note that the above theorem is false if one replaces $\infty$ by a finite integer $N \in \b{N}$.
%For instance, given an urn containing balls of different color, the distribution of colors when drawing $N$ balls without replacement is exchangeable, but not a mixture of i.i.d. random variables. 
%This is the main reason why we have worked with the infinite product space $X^\infty$ rather than $X^N$.\\
Equipped with Theorem \ref{thm:deFinetti} we can rewrite the statement of Theorem \ref{thm:mainRelaxedSetRepresentation} as
\begin{align*}
	\c{W}_\infty &= \left\{\bar{P} \in \c{P}_{\operatorname{sym}}(X^\infty) \mid W_c(\hat{P},\cdot) \leq \rho \text{\ $\nu_{\bar{P}}$-almost-surely}  \right\}\,, \\
	\hat{\c{W}}_\infty &= \left\{\bar{P} \in \c{P}_{\operatorname{sym}}(X^\infty) \mid \b{E}_{P \sim \nu_{\bar{P}}} W_c(\hat{P},P) \leq \rho \right\}\,,
\end{align*}
where $\nu_{\bar{P}}$ denotes the unique mixture distribution of the symmetric measure $\bar{P}$.\\
\\
Now, \eqref{eq:relaxationOriginalSetConvexHull} and Theorem \ref{thm:mainRelaxedSetRepresentation} together imply:
\begin{align}
	\label{eq:relaxationComparison}
	\operatorname{S}(\ell)
	= \sup_{\bar{P} \in \c{U}} \b{E}_{x \sim \bar{P}} \ell(x)
	\text{\ \ and\ \ }
	\hat{\operatorname{U}}_\infty^{\operatorname{sym}}(\ell)
	= \sup_{\bar{P} \in \c{U}_\infty} \b{E}_{x \sim \bar{P}} \ell(x)
\end{align}
with
\begin{align*}
	\c{U} &:= \left\{\intg{}{P^{\otimes N}}{\nu(P)} \mid \nu \in \c{P}(\c{P}(X))\,,\; W_c(\hat{P},\cdot) \leq \rho \text{\ $\nu$-almost-surely}  \right\}\,, \\
	\c{U}_\infty &= \left\{\intg{}{P^{\otimes N}}{\nu(P)} \mid \nu \in \c{P}(\c{P}(X))\,,\; \b{E}_{P \sim \nu} W_c(\hat{P},P) \leq \rho \right\}\,.
\end{align*}
Having obtained a good understanding the nature on the limit of the sequence of relaxations we can ask now which conditions on $\ell$ would suffice to imply that there is no relaxation gap, i.e. that $\operatorname{S}(f) = \hat{\operatorname{U}}_\infty^{\operatorname{sym}}(\ell)$.
Let us one last time rewrite \eqref{eq:relaxationComparison}, this time in terms of the mixing measures $\nu$ directly.
We have
\begin{align}
	\label{eq:relaxationComparisonIntegral}
	\operatorname{S}(\ell)
	= \sup_{\substack{\nu \in \c{P}(\c{P}(X)) \\ W_c(\hat{P},\cdot) \leq \rho \text{\ $\nu$-a.s.}}} \intg{}{F_\ell(P)}{\nu(P)}\,,\quad
	\hat{\operatorname{U}}_\infty^{\operatorname{sym}}(\ell)
	= \sup_{\substack{\nu \in \c{P}(\c{P}(X)) \\ \b{E}_{P \sim \nu} W_c(\hat{P},P) \leq \rho }} \intg{}{F_\ell(P)}{\nu(P)}\,,
\end{align}
where $F_\ell(P) = \intg{}{\ell}{P^{\otimes N}}$.
Thus, the question on the absence of a relaxation gap is equivalent to asking when a constraint can be replaced by its expectation.
%In general the expected version of a constraint is weaker than the constraint itself as the following example shows.
%\begin{example}
%	Consider the case of maximizing the ReLU function $F:\b{R} \to \b{R}: x \mapsto \max\{x,0\}$ over the set $T = [-2,0]$.
%	Clearly $\max_{x \in T} F(x) = 0$.
%	Moreover, this set can be represented by e.g. the following constraint: $T = \{x \in \b{R} \mid (x+1)^2 \leq 1\}$.
%	Then averaged optimization set is $\hat{T} = \{\nu \in \c{P}(\b{R}) \mid \b{E}_{x \sim \nu} (x+1)^2 \leq 1\}$.
%	Then by taking $\nu_t = (1-\alpha) \delta_{-1} + \alpha \delta_t$ for $t \geq 0$ such that $\alpha = \frac{1}{(1+t)^2}$, we obtain $\nu_t \in \hat{T}$ and hence
%	\begin{align*}
%		\sup_{\nu \in \hat{T}} \b{E}_{x \sim \nu} F(x)
%		\geq \sup_{t \geq 0} \b{E}_{x \sim \nu_t} F(x) 
%		= \sup_{t \geq 0} \frac{t}{(t+1)^2}
%		= \frac{1}{4}
%		> 0 
%		= \max_{x \in T} F(x)\,.
%	\end{align*} 
%\end{example}
%While the last example exposes the main issue of replacing a constraint by its average, namely that it might be worth violating the constraint to obtain a strictly better objective value, it is not completely analogous to the situation in this section.
%Indeed, instead of dealing with functions $F$ defined over the real numbers $\b{R}$, we deal with functions defined over probability measures $\c{P}(X)$. 
%However, even in this case the gap can indeed be non-zero, as the following example shows.
The following example shows that in general the relaxation gap is not zero.
\begin{example}
	\label{ex:infiniteGap}
	Let us consider $X = \b{R}$ and $N = 2$ with $\hat{P} = \delta_0$ and $\ell:X^2 \to \b{R}: (x_1,x_2) \mapsto x_1 x_2^2$ and $\rho > 0$.
	Then we have for the structured ambiguity set $\c{W} = \{P^{\otimes 2} \mid P \in \b{B}_\rho^c(\hat{P})\}$ that 
	\begin{align*}
		\operatorname{S}(f)
		= \sup_{P \in \b{B}_\rho^c(\hat{P})} \intg{}{\ell(x_1,x_2)}{P^{\otimes 2}(x_1,x_2)}
		= \sup_{P \in \b{B}_\rho^c(\hat{P})} (\b{E}P)(\b{E}(P^2))
		= \rho^{3/2}\,,
	\end{align*}
	where we have exploited Theorem \ref{thm:droGeneralDuality} to obtain that 
	\begin{align*}
		\sup_{P \in \b{B}_\rho^c(\hat{P})} (\b{E}P) = \rho^{1/2} \text{\ \ and\ \ } \sup_{P \in \b{B}_\rho^c(\hat{P})} \b{E}(P^2) = \rho
	\end{align*}
	and noted that both attained at $P = \delta_{\sqrt{\rho}}$.
	We claim that $\hat{\operatorname{U}}_\infty^{\operatorname{sym}}(f) = \infty$ and thus $\operatorname{U}_M^{\operatorname{sym}}(f) = \infty$ for all $M \geq 2$.
	Indeed, the sequence of distributions $\bar{P}_n = (1-\frac{\rho}{n^2}) \delta_0^{\otimes \infty} + \frac{\rho}{n^2} \delta_n^{\otimes \infty}$ satisfies $\operatorname{pr}_{1:M}\# \bar{P}_n = (1-\frac{\rho}{n^2}) \delta_0^{\otimes M} + \frac{\rho}{n^2} \delta_n^{\otimes M}  \in \b{B}_{M\rho}^{c^M}(P^{\otimes M})$ for all $M \geq 1$, since 
	\begin{align*}
		\intg{}{\intg{}{c^M(x,y)}{\delta_0^{\otimes M}(y)}}{\operatorname{pr}_{1:M}\# \bar{P}_n(x)} 
		= \intg{}{\sum_{i=1}^M x_i^2}{\operatorname{pr}_{1:M}\# \bar{P}_n(x)} 
		= \frac{\rho}{n^2} M n^2 
		= M\rho	\,.
	\end{align*}
	Moreover, it holds that 
	\begin{align*}
		\intg{}{\hat{\ell}_{\operatorname{sym}}(x)}{\bar{P}_n(x)} 
		= \frac{\rho}{n^2} \ell(n,n) = \rho n \to \infty \text{\ for\ } n \to \infty\,.
	\end{align*}
	Thus $\operatorname{U}_M^{\operatorname{sym}}(\ell) = \infty$ for all $M \geq N = 2$ and the relaxation gap is infinite.
\end{example}

Note that in this example $\ell$ does not satisfy the growth assumption $\ell \in \c{G}_{c^2}(X^2)$, which is sufficient for the finiteness of the unstructured UQ problem \eqref{eq:droGeneral}, but not necessary for the finiteness of the structured version \eqref{eq:droStructured}.
The question of whether there exists an example with a non-zero relaxation gap such that \eqref{eq:droGeneral} is finite, remains open.

\subsection{Tightness of the relaxation sequence}
\label{sec:structuredDROTightnessOfRelaxationSequence}

Intuitively, replacing a constraint by its average will not increase the optimal value if any ``average feasible'' point can be replaced by a single point that satisfies the original constraint without decreasing its objective value.
If the latter ``replacement'' is done by taking the average, i.e. taking the mean $\b{E}_{x \sim \nu} x$ of $\nu$ itself, then Jensen's inequality
\begin{align*}
	F(\b{E}_{x \sim \nu} x) \geq \b{E}_{x \sim \nu} F(x) \text{\ \ for all concave\ } F\,,
\end{align*}    
shows that concavity of the objective function $F$ is sufficient for the absence of a relaxation gap.
The following theorem gives a sufficient condition for the function $F_\ell(P) = \intg{}{\ell}{P^{\otimes N}}$ to be convex in the sense of the usual linear structure on $\c{P}(X)$.
For this we need to introduce the following notion \cite[Definition 8.1]{wendland2004scattered}: A function $k:X \times X \to \b{R}$ is said to be conditionally negative definite (c.n.d.) if for any $n \in \b{N}$, $\gamma \in \b{R}^n$ and $x \in X^n$ the matrix $(k(x_i,x_j))_{i,j=1}^n$ is negative semi-definite on the subspace $\1^\top \gamma = 0$, i.e. if 
\begin{align*}
	\sum_{i=1}^n \gamma_i = 0 \Longrightarrow \sum_{i,j=1}^n \gamma_i \gamma_j k(x_i,x_j) \leq 0\,.
\end{align*}
The result reads now as follows.
\begin{theorem}
	\label{thm:concaveLinearStructureExact}
	Let $N \geq 2$ and let $\ell:X^N \to \b{R}$ be such that $\abs{\ell} \in \c{G}_{c^N}(X^N)$ and the symmetrization $\ell_{\operatorname{sym}}$ is c.n.d. in the first two (and hence any two) of its variables, i.e. for all $\bar{x} \in X^{N-2}$ the map $X\times X \to \b{R}: (x_1,x_2) \mapsto \ell_{\operatorname{sym}}(x_1,x_2,\bar{x})$ is c.n.d.
	Then $F_\ell$ is concave on $\c{P}_c(X)$, i.e.
	\begin{align*}
		F_\ell(\alpha P_1 + (1-\alpha) P_2) \geq \alpha F_\ell(P_1) + (1-\alpha) F_\ell(P_2) \text{\ \ for\ } P_1, P_2 \in \c{P}_c(X)\,,\; \alpha \in [0,1]\,.
	\end{align*}
	Moreover, in this case
	\begin{align*}
		\operatorname{S}(\ell)
		= \hat{\operatorname{U}}_\infty^{\operatorname{sym}}(\ell)\,,
	\end{align*}
	and if $\ell \in \c{G}_{c^N}^<(X^N)$, then
	\begin{align*}
		\operatorname{S}(\ell)
		= \lim_{M \to \infty} \operatorname{U}_M^{\operatorname{sym}}(\ell)\,.
	\end{align*}
\end{theorem}
%In our case the objective function is $F_\ell(P) = \intg{}{\ell}{P^{\otimes N}}$ is unfortunately not concave in the sense of the usual linear structure of $\c{P}_c(X)$, i.e. in general 
%\begin{align*}
%	F_\ell(\alpha P_1 + (1-\alpha) P_2) \not \geq \alpha F_\ell(P_1) + (1-\alpha) F_\ell(P_2) \text{\ \ for\ } P_1, P_2 \in \c{P}_c(X)\,,\; \alpha \in [0,1]\,.
%\end{align*}
\begin{example}
	\label{ex:conditionallyNegativeDefinite}
	Let $X$, $N$, $c$ and $\ell$ be as in Example \ref{ex:liftedRelaxation}, $\ell(x_1,x_2) = -x_1 x_2$.
	Then $\ell$ is (conditionally) negative definite, since if $x \in X^n$ and $\gamma \in \b{R}^n$ are such that $\1^\top \gamma = 0$, then
	\begin{align*}
		\sum_{i,j=1}^n \gamma_i \gamma_j \ell(x_i,x_j)
		= - \left(\sum_{i=1}^n \gamma_i x_i \right)^2 
		\leq 0\,.
	\end{align*}
	Thus, the conclusion of Theorem \ref{thm:concaveLinearStructureExact} is consistent with Example \ref{ex:liftedRelaxation}.
\end{example}
Unfortunately, verifying conditional negative definiteness in practice can be challenging.
It turns out, however, that if $X$ is additionally a vector space and $\ell:X^N \to \b{R}$ is concave, then $F_\ell$ is \emph{geodesically concave} in the sense of the Wasserstein space on $\c{P}_c(X)$.
%For this we first need to introduce the concept of a generalized geodesic \cite{ambrosio2008gradient}.
%\begin{definition}
%	Let $X$ be a vector space and $P_1,P_2,\hat{P} \in \c{P}_c(X)$.
%	Then a generalized geodesic (from $P_1$ to $P_2$, with base point $\hat{P}$) is a map $\gamma:[0,1] \to \c{P}_c(X)$ of the form $\gamma(t) = ((1-t)\operatorname{pr}_1 + t \operatorname{pr}_2) \# \Lambda$, where $\Lambda \in \Gamma(P_1,P_2,\hat{P})$ is such that $\operatorname{pr}_{1,3} \# \Lambda \in \Gamma_c(P_1,\hat{P})$ and $\operatorname{pr}_{2,3} \# \Lambda \in \Gamma_c(P_2,\hat{P})$.
%\end{definition}
%Concavity is now understood as follows \cite{ambrosio2008gradient}. 
%\begin{definition}
%	A function $F:\c{P}_c(X) \to \b{R}$ is said to be concave w.r.t. generalized geodesics with base point $\hat{P} \in \c{P}_c(X)$ if for any two $P_1,P_2 \in \c{P}_c(X)$ there exists a generalized geodesic from $P_1$ to $P_2$ with base point $\hat{P}$ such that
%	\begin{align*}
%		F(\gamma(t)) \geq (1-t) F(P_1) + t F(P_2) \text{\ for all\ } t \in [0,1]\,.
%	\end{align*}
%\end{definition}
We will not introduce the concepts of geodesic convexity or concavity here, since they will not be needed any further in this paper, but just note that the so-called \emph{Wasserstein geodesics} \cite{ambrosio2008gradient} provide a way of interpolating between two distributions $P_1$ and $P_2$ in $\c{P}_c(X)$ that is \emph{different} from $\alpha P_1 + (1-\alpha) P_2$.
Thus, if the average w.r.t. $\nu$ is understood in a different, namely \emph{geodesic} sense, then we can indeed show the relaxation gap again to be zero.
The price to pay, however, is that the underlying feasible set $\b{B}_\rho^c(\hat{P})$ is convex w.r.t. this new notion of convexity as well, which requires an additional assumption on the transportation cost $c$.
The following theorem summarizes these observations and is the main result of this section.
\begin{theorem}
	\label{thm:relaxationExactConcave}
	Suppose that $X$ is a convex subset of a vector space and $c$ is proper such that $c(x,\cdot):X \to [0,\infty)$ is convex for all $x \in X$.
	If $\ell \in \c{G}_{c^N}(X^N)$ is upper semi-continuous and concave, then
	\begin{align*}
		\operatorname{S}(\ell)
		= \hat{\operatorname{U}}_\infty^{\operatorname{sym}}(\ell)\,.
	\end{align*}
	If additionally $\ell \in \c{G}_{c^N}^<(X^N)$, then
	\begin{align*}
		\operatorname{S}(\ell)
		= \lim_{M \to \infty} \operatorname{U}_M^{\operatorname{sym}}(\ell)\,.
	\end{align*}
\end{theorem}

Note that the convexity assumption on the transportation cost is naturally satisfied if e.g. $X = \b{R}^n$, $c = d^p$ with $d(x,y) = \norm{x-y}$ for some $p \in [1,\infty)$ and norm $\norm{\cdot}$ on $X$.
\begin{remark}
	Note that Theorem \ref{thm:concaveLinearStructureExact}	and Theorem \ref{thm:relaxationExactConcave} are not subcases of each other and rely on different notions of convex interpolation. 
	The problem of determining the largest class of functions $\ell$ with a zero relaxation gap remains open.
\end{remark}

\subsection{Comparison to \cite{chaouach2022tightening,chaouach2023structured}}
\label{sec:structured_sets_comparison}

As already noted in the introduction to this section, Wasserstein ambiguity sets with additional structure have not been considered much in the literature with one notable exception given by \cite{chaouach2022tightening,chaouach2023structured}, where non-convex uncertainty sets of the form 
\begin{align}
	\label{eq:wassersteinHyperrectangle}
	\begin{aligned}
		\c{W}_{\operatorname{rect}}
		&= \b{B}_{\rho_1}^{c_1}(\hat{P}_1) \otimes \cdots \otimes  \b{B}_{\rho_N}^{c_N}(\hat{P}_N) \\
		&= \{P_1 \otimes \cdots \otimes P_N \mid P_i \in \b{B}_{\rho_i}^{c_i}(\hat{P}_i)\,,\; i=1,\ldots,N\}
	\end{aligned}
\end{align}
are considered and termed Wasserstein hyperrectangles\footnote{in fact, \cite{chaouach2023structured} allows for different spaces $X_k$ in each coordinate, but we will focus on the case $X_k = X$ for all $k=1,\ldots,N$}.
Note that even in the case of $\rho_i = \rho$, $c_i = c$ and $\hat{P}_i = \hat{P}$ for all $i=1,\ldots,N$ sets of the form \eqref{eq:wassersteinHyperrectangle} are different from sets \eqref{eq:structuredWassersteinAmbiguitySet} in that they allow each factor $P_i \in \b{B}_\rho^c(\hat{P})$ to depend on the index $i$.
Thus clearly in this case
\begin{align*}
	\c{W} = \{P^{\otimes N} \mid P \in \b{B}_\rho^c(\hat{P})\} \subsetneq \c{W}_{\operatorname{rect}} = \b{B}_\rho^c(\hat{P}) \otimes \cdots \otimes \b{B}_\rho^c(\hat{P})\,.
\end{align*}
Now, in \cite{chaouach2023structured} a duality result is established for UQ problems w.r.t. sets \eqref{eq:wassersteinHyperrectangle} and functions $\ell:X^N \to \b{R}$ that are additively and multiplicatively \emph{separable}, i.e. are of the form
\begin{align*}
	\ell(x)
	= \sum_{k=1}^N \ell_k(x_k) \text{\ \ or\ \ } \ell(x)
	= \prod_{k=1}^N \ell_k(x_k)\,, \qquad x \in X^N
\end{align*}
for some $\ell_k:X \to \b{R}$, $k=1,\ldots,N$, respectively.
However, this requirement of the loss $\ell$ is very restrictive and in this case
\begin{align*}
	\intg{}{\ell(x)}{(P_1 \otimes \ldots \otimes P_N)(x)}
	=
	\begin{cases}
		\sum_{k=1}^N \intg{}{\ell_k(x_k)}{P_k(x_k)} & \text{if\ }\ell(x)
		= \sum_{k=1}^N \ell_k(x_k) \\
		\prod_{k=1}^N \intg{}{\ell_k(x_k)}{P_k(x_k)} & \text{if\ }\ell(x)
		= \prod_{k=1}^N \ell_k(x_k)
	\end{cases}\,,
\end{align*}
which implies that the structured UQ problem
\begin{align*}
	\sup_{\bar{P} \in \c{W}_{\operatorname{rect}}} \intg{}{\ell(x)}{\bar{P}(x)}
\end{align*}
splits into $N$ decoupled and unstructured problems of the form \eqref{eq:droGeneral}.
For this reason we will not state the latter duality here explicitly.
Further \cite{chaouach2023structured} also consider the so-called \emph{multitransport hyperrectangles} that are defined for given $\check{P} \in \c{P}(X^N)$, costs $c_k$ and radii $\rho_k$ as\footnote{Here $\operatorname{pr}_{k,k}:X^N \times X^N \to X \times X$ is the projection onto the $k$-th coordinates in each $X^N$. See Appendix \ref{sec:additionalNotation} for additional notation.}
\begin{align*}
	\c{W}_{\operatorname{multi}}
	= \left\{\bar{P} \in \c{P}(X^N) \mid \exists \Lambda \in \Gamma(\check{P},\bar{P})\,,\;\intg{}{c_k}{\operatorname{pr}_{k,k}\#\Lambda} \leq \rho_k \text{\ for\ } k=1,\ldots,N\right\}
\end{align*}
Contrary to $\c{W}_{\operatorname{rect}}$, the set $\c{W}_{\operatorname{multi}}$ is convex. 
It is then shown that if $\check{P} = \hat{P}_1\otimes \ldots \otimes \hat{P}_N$, then $\c{W}_{\operatorname{rect}} \subseteq \c{W}_{\operatorname{multi}} \subseteq \b{B}_\rho^{c^N}(\check{P})$ for $\rho = \sum_{i=1}^N \rho_i$ and this time $c^N(x,y) = \sum_{i=1}^N c_i(x_i,y_i)$.
Additionally the following corresponding duality result is established \cite[Theorem 6.4]{chaouach2023structured}.
\begin{theorem}[\cite{chaouach2023structured}]
	Let $c_i = d_i^p$ for $p \in [1,\infty)$, some metrics $d_i$ on $X$ that metrisize $X$ and suppose that $\ell\in L^1(\hat{P})$ is upper semi-continuous.
	Then for $\c{W}_{\operatorname{multi}}$ as defined above it holds
	\begin{align}
		\label{eq:droMultitransportDuality}
		\sup_{\bar{P} \in \c{W}_{\operatorname{multi}}} \b{E}_{x \sim \bar{P}} \ell(x)
		= \inf_{\mu \in [0,\infty)^N} \mu^\top \hat{\rho} + \b{E}_{z \sim \check{P}} \hat{\phi}_{\mu}(z)\,,
	\end{align}
	where $\hat{\rho} = (\rho_1,\ldots,\rho_N)^\top$, the infimum is attained and $\hat{\phi}_{\mu}(z) = \sup_{x \in X^N} (\ell(x) - \sum_{i=1}^N \mu_i c_i(x_i,z_i))$.
\end{theorem}
In the case of $\rho_i = \rho$ and $c_i = c = d^p$ and $P_i = P$ for all $i=1,\ldots,N$ we have
\begin{align*}
	\c{W} \subseteq \c{W}_{\operatorname{rect}} \subseteq \c{W}_{\operatorname{multi}}\,.
\end{align*}
Thus, a natural question is if the UQ problem over $\c{W}_{\operatorname{multi}}$ provides a better or different approximation to \eqref{eq:droStructured} than our lifting scheme developed in the previous sections.
The following result shows that this is not the case and that the quantity \eqref{eq:droMultitransportDuality} is at least as large as the quantity \eqref{eq:droSymmetric}, i.e. the our first relaxation upper bound $\operatorname{U}_N^{\operatorname{sym}}(f)$ is not more conservative than the value given by \eqref{eq:droMultitransportDuality}.
\begin{theorem}
	\label{thm:multitransportSymmetricInequality}
	When $\rho_i = \rho$,  $c_i = c = d^p$ and $\hat{P}_i = \hat{P}$ for all $i=1,\ldots,N$ as well as $\check{P} = \hat{P}^{\otimes N}$, then for any Borel measurable $\ell:X^N \to \b{R}$ we have
	\begin{align*}
		\sup_{\bar{P} \in \c{W}_{\operatorname{multi}}} \b{E}_{x \sim \bar{P}} \ell(x) 
		\geq \operatorname{U}_N^{\operatorname{sym}}(\ell) 
		= \operatorname{U}(\ell_{\operatorname{sym}})\,.
	\end{align*}
\end{theorem}
Hence there is no purely theoretical advantage of using multitransport hyperrectangles $\c{W}_{\operatorname{multi}}$ over our relaxation approach to approximate the structured UQ problem \eqref{eq:droStructured}.
On the other hand, \eqref{eq:droMultitransportDuality} does not require the symmetrization $\ell_{\operatorname{sym}}$ of the loss $\ell$, which can be expensive (see Section \ref{sec:reformulationNumerical}).

\section{Relaxations of the distributionally robust optimization problem}
\label{sec:structuredDROOuter}
In this section we return to the problem \eqref{eq:stochasticOptimizationRobustStructured} and examine the implications of replacing its inner uncertainty quantification problem by its relaxed version, i.e. replacing \eqref{eq:stochasticOptimizationRobustStructured}, where $\Psi_{\operatorname{S}}(\theta) = \operatorname{S}(\ell(\theta,\cdot))$ in the notation of the previous section, by the problem
\begin{align}
	\label{eq:structuredDROOuterRelaxed}
	\inf_{\theta \in \Theta} \Psi_{\operatorname{U}}^M(\theta) \text{\ \ with\ \ } \Psi_{\operatorname{U}}^M(\theta) = \operatorname{U}_M^{\operatorname{sym}}(\ell(\theta, \cdot))\,,
\end{align}
where this time we explicitly see $\ell:\Theta \times X^N \to \b{R}$ as a function of two arguments $(\theta,x) \in \Theta \times X^N$.
The following result relates the solutions of the relaxed problem \eqref{eq:structuredDROOuterRelaxed} to solutions of \eqref{eq:stochasticOptimizationRobustStructured}.
\begin{theorem}
	\label{thm:structuredDROOuterTightness}
	Suppose that $X$ is a convex subset of a vector space, $c:X \times X \to [0,\infty)$ is proper and $c(x,\cdot)$ is convex for all $x \in X$.
	Further, let $\Theta \subseteq \b{R}^p$ and $\ell:\Theta \times X^N \to \b{R}$.
	If $\ell(\theta,\cdot) \in \c{G}_{c^N}^<(X^N)$ is concave and upper semi-continuous for all $\theta \in \Theta$, then it holds
	\begin{align}
		\label{eq:structuredDROOuterValueConvergence}
		\lim_{M \to \infty} \inf_{\theta \in \Theta} \Psi_{\operatorname{U}}^M(\theta) = \inf_{\theta \in \Theta} \Psi_{\operatorname{S}}(\theta)\,. 
	\end{align}
	Suppose additionally that $\Theta$ is compact, convex with $\operatorname{int}(\Theta) \neq \emptyset$ and that $\ell(\cdot,x)$ is convex and lower semi-continuous on $\Theta$ for all $x \in X^N$ and that the family $\{-\min\{\ell(\theta,\cdot),0\}\}_{\theta \in \Theta}$ is $\bar{P}$-integrable \emph{uniformly} in $\bar{P} \in \c{W}$.
	Then the infimum on the right hand side of \eqref{eq:structuredDROOuterValueConvergence} is attained and given any sequence $(\epsilon_M)_{M = N}^\infty$ with $0 < \epsilon_M \searrow 0$ for $M \to \infty$ and arbitrary $\epsilon_M$-approximate minimizers
	\begin{align*}
		\theta_M^* 
		\in \epsilon_M\text{-}\argmin_{\theta \in \Theta} \Psi_{\operatorname{U}}^M(\theta) 
		:= \left\{\theta \in \Theta \mid \Psi_{\operatorname{U}}^M(\theta) < \inf_\Theta \Psi_{\operatorname{U}}^M + \epsilon_M\right\}\,,
	\end{align*} 
	it holds that
	\begin{align*}
		\theta^* \in \argmin_{\theta \in \Theta} \Psi_{\operatorname{S}}(\theta)
	\end{align*}
	for any limit point $\theta^*$ of $(\theta_M^*)_{M=N}^\infty$.
\end{theorem}

\begin{remark}
	The uniform integrability assumption on $\{-\min\{\ell(\theta,\cdot),0\}\}_{\theta \in \Theta}$ is used to establish the lower semi-continuity of $\Psi_{\operatorname{S}}$ on $\Theta$ only and is satisfied if e.g. there exists a constant $C > 0$ and $x_0 \in X^N$ such that
	\begin{align*}
		\ell(\theta,x) \geq -C(1+c^N(x,x_0)) \text{\ \ for all\ } x \in X^N\,,\; \theta \in \Theta\,,
	\end{align*}
	i.e. $-\ell(\theta,\cdot)$ ``belong uniformly in $\theta \in \Theta$'' to the class $\c{G}_{c^N}(X^N)$.  
\end{remark}

Thus Theorem \ref{thm:structuredDROOuterTightness} provides a guarantee that (possibly suboptimal) solutions of the relaxed problem \eqref{eq:structuredDROOuterRelaxed} will approximate the true minimizers of \eqref{eq:stochasticOptimizationRobustStructured} if $\ell$ is a saddle (i.e. convex-concave) function with certain integrability and continuity properties.
Note, however, that even without any convexity or concavity assumptions on $\ell$, it can be beneficial to solve \eqref{eq:structuredDROOuterRelaxed} as a substitute to \eqref{eq:stochasticOptimizationRobustStructured}.
Indeed, suppose that we want to find a decision vector $\theta^* \in \Theta$ that satisfies a robustness margin given by a threshold $R \in \b{R}$, i.e.
\begin{align*}
	\sup_{\bar{P} \in \c{W}} \b{E}_{x \sim \bar{P}} \ell(\theta^*,x) \leq R\,.
\end{align*} 
If we now solve instead \eqref{eq:structuredDROOuterRelaxed} for some fixed $M \geq N$ to (sub)optimality and obtain a decision vector $\theta_M^* \in \Theta$ with an objective value $\Psi_{\operatorname{U}}^M(\theta_M^*) \leq R$, then after implementing $\theta_M^*$ in our original problem \eqref{eq:stochasticOptimizationRobustStructured}, we obtain
\begin{align*}
	\sup_{\bar{P} \in \c{W}} \b{E}_{x \sim \bar{P}} \ell(\theta_M^*,x)
	= \Psi_{\operatorname{S}}(\theta_M^*) 
	\leq \Psi_{\operatorname{U}}^M(\theta_M^*)
	\leq R\,,
\end{align*}
i.e. we can take $\theta^* = \theta_M^*$ to meet the desired specifications.
%The first part of Theorem \ref{thm:structuredDROOuterTightness} just that if $\ell$ is concave in the second argument, then, provided the orgi there will exist some $M \geq N$ such that we indeed will find such a decision $\theta_M^*$.
\begin{remark}
	\label{rem:nonmonotonicRelaxationOuter}
	While Theorem \ref{thm:structuredDROOuterTightness} implies that $\lim_{M \to \infty} \Psi_{\operatorname{S}}(\theta_M^*) = \inf_\Theta \Psi_{\operatorname{S}}$ whenever $\ell$ is concave in the second argument, this convergence is, in general, not monotone (see Section \ref{sec:structuredDRONumerical} for a numerical example).
	As such, after having obtained $\theta_M^* \in \epsilon_M\text{-}\argmin_\Theta \Psi_{\operatorname{U}}^M$ for $M \in \{N,\ldots,M_{\max}\}$ with some maximal $M_{\max} \geq N$, it is advisable to take $\theta_{M^*}^*$ with $M^* = \argmin_{M \in \{N,\ldots,M_{\max}\}} \Psi_{\operatorname{U}}^M(\theta_M^*)$ as the final decision vector.
\end{remark}

\section{Tractable reformulations and numerical examples}
\label{sec:reformulationNumerical}

In this section we reformulate each uncertainty quantification problem in the proposed sequence of relaxations into an explicit finite-dimensional convex program and provide some numerical examples.

\subsection{Finite-dimensional reformulations using Lagrange duality} 
\label{sec:structuredDROPolyhedralLoss}

When the nominal distribution $\hat{P}$ is discrete, i.e.
\begin{align}
	\label{eq:discreteMeasure}
	\hat{P} = \sum_{i=1}^{n_{\hat{P}}} \hat{p}_i \delta_{\hat{\xi}_i}\,,
\end{align} 
the value of $\operatorname{U}_M^{\operatorname{sym}}(\ell)$ as defined in \eqref{eq:droStructuredRelaxationSequence} can be evaluated for any $M \geq N$ by reformulating \eqref{eq:droStructuredRelaxationSequence} into a semi-infinite program via Theorem \ref{thm:droGeneralDuality} in the following way.
We first note that
\begin{align*}
	\operatorname{U}_M^{\operatorname{sym}}(\ell)
	&= \operatorname{U}_M((\ell \circ \operatorname{pr}_{1:N}^M)_{\operatorname{sym}}) \\
	&= \sup_{\bar{P} \in \b{B}_{M\rho}^{c^M}(\hat{P}^{\otimes M})} \b{E}_{x \sim \bar{P}} (\ell \circ \operatorname{pr}_{1:N}^M)_{\operatorname{sym}}(x) \\
	&= \inf_{\mu \geq 0} M\rho \mu + \sum_{\bs{i} \in \bs{\hat{I}}} \hat{p}_{\bs{i}} \sup_{x \in X^M} \left((\ell \circ \operatorname{pr}_{1:N}^M)_{\operatorname{sym}}(x) - \mu c^M(x,\hat{\xi}_{\bs{i}})\right)\,,
\end{align*}
where $\bs{\hat{I}} = \{1,\ldots,n_{\hat{P}}\}^M$, $\hat{p}_{\bs{i}} = \hat{p}_{i_1}\cdots \hat{p}_{i_M}$, $\hat{\xi}_{\bs{i}} = (\hat{\xi}_{i_1},\ldots,\hat{\xi}_{i_M})$ when $\bs{i} = (i_1,\ldots,i_M) \in \bs{\hat{I}}$ and the last equality holds due to Theorem \ref{thm:droGeneralDuality} as well as the fact that
\begin{align*}
	\hat{P}^{\otimes M}
	= \sum_{\bs{i} \in \bs{\hat{I}}} \hat{p}_{\bs{i}} \delta_{\hat{\xi}_{\bs{i}}}
\end{align*}
is also discrete.
By introducing slack variables $\sigma_{\bs{i}} \in \b{R}$ for $\bs{i} \in \hat{I}$, we obtain the semi-infinite program
\begin{align}
	\label{eq:droRelaxationReformulationSemiInfinite}
	\begin{aligned}
		\operatorname{U}_M^{\operatorname{sym}}(\ell)
		&= \inf_{\mu \geq 0} M\rho \mu + \sum_{\bs{i} \in \bs{\hat{I}}} \hat{p}_{\bs{i}} \sigma_{\bs{i}} \\
		&\quad\text{such that\ }
		\forall \bs{i} \in \bs{\hat{I}}\,,\; x \in X^M\,:\; (\ell \circ \operatorname{pr}_{1:N}^M)_{\operatorname{sym}}(x) - \mu c^M(x,\hat{\xi}_{\bs{i}}) \leq \sigma_{\bs{i}}\,.
	\end{aligned}
\end{align}
When $X = \b{R}^n$ is Euclidean and the loss $\ell$ and cost $c$ belong to some particular classes of functions, the latter constraints can be reformulated more explicitly.
In view of Theorem \ref{thm:relaxationExactConcave} we focus on concave polyhedral loss function $\ell$ and cost $c(x,y) = \norm{x-y}$ for some norm $\norm{\cdot}$ on $\b{R}^n$.

\subsubsection{Concave polyhedral loss}
Suppose that $\ell:(\b{R}^n)^N \to \b{R}$ concave and polyhedral, i.e. of the form
\begin{align}
	\label{eq:polyhedralConcaveLoss}
	\ell(x)
	= \min_{\smat{a \\ b} \in \c{H}} a^\top x + b
	= \min_{h \in \c{H}} h^\top \mat{x \\ 1}\,,
\end{align}
where $\c{H} \subseteq (\b{R}^n)^N \times \b{R}$ is a polytope.
Additionally, suppose that $c(x,y) = \norm{x-y}$ for some norm $\norm{\cdot}$ on $\b{R}^d$.
Before we state our reformulation, we need to introduce some additional notation.
By $\bs{\c{L}} = \{\bs{l} \in \{1,\ldots,M\}^N \mid \bs{l}_k \neq \bs{l}_l \text{\ for\ } k \neq l\}$ we denote the set of non-repeating $N$-tuples of integers in $\{1,\ldots,M\}$ and by $E_{\bs{l}} \in \b{R}^{nN \times nM}$ the component selection matrix such that $E_{\bs{l}} x = x_{\bs{l}} := (x_{\bs{l}_1},\ldots,x_{\bs{l}_N}) \in (\b{R}^n)^N$.
Moreover, we define an equivalence relation on $\bs{\hat{I}}$ by $\bs{i} \sim \bs{j}$ if and only if there exists $\pi \in \c{S}_M$ with $\bs{i}_k = \bs{j}_{\pi(k)}$ for all $k \in \{1,\ldots,M\}$.
Let $\bs{\hat{\c{I}}} = \bs{\hat{I}}/\sim$ be the set of equivalence classes with $\bs{\iota} = [\bs{i}]_{\sim}$ being the equivalence class of $\bs{i}$ and let $\upsilon:\bs{\hat{\c{I}}} \to \bs{\hat{I}}$ denote an arbitrary, but fixed selection such that $\upsilon(\bs{\iota}) \in \bs{\iota}$ for all $\bs{\iota} \in \bs{\hat{\c{I}}}$. 
Finally, for any norm $\norm{\cdot}$ we denote by $\norm{\cdot}_*$ its dual norm. 
The following result holds.

\begin{theorem}
	\label{thm:polyhedralConcaveLossProgram}
	Let $\rho > 0$, $\hat{P}$ be as in \eqref{eq:discreteMeasure}, the loss as in \eqref{eq:polyhedralConcaveLoss} and $c(x,y) = \norm{x-y}$ for some norm $\norm{\cdot}$ on $X = \b{R}^n$.
	Then
	\begin{align}
		\label{eq:polyhedralConcaveLossProgram}
		\operatorname{U}_M^{\operatorname{sym}}(\ell)
		= \inf \; \mu M \rho + \sum_{\bs{\iota} \in \bs{\hat{\c{I}}}} \hat{p}_{\bs{\iota}} \sigma_{\bs{\iota}} 
	\end{align}		
	where the infimum is taken w.r.t. the following set of variables
	\begin{align*}
		\mu \geq 0\,,\; \sigma_{\bs{\iota}}\,,\; b_{\bs{l},\bs{\iota}}  \in \b{R} \,, \; z_{\bs{\iota}} = (z_{\bs{\iota}}^j)_{j=1}^M \in (\b{R}^n)^M\,,\; a_{\bs{l},\bs{\iota}} \in (\b{R}^n)^N \quad \forall \bs{l} \in \bs{\c{L}}\,,\; \bs{\iota} \in \bs{\hat{\c{I}}}\,,
	\end{align*}
	under the (convex) constraints
	\begin{align*}
		0 \leq \sigma_{\bs{\iota}} + z_{\bs{\iota}}^\top \hat{\xi}_{\upsilon(\bs{\iota})} + \frac{(M-N)!}{M!} \sum_{\bs{l} \in \bs{\c{L}}} b_{\bs{l},\bs{\iota}}\,,\quad
		\norm{z_{\bs{\iota}}^j}_* \leq \mu  \\
		z_{\bs{\iota}} = \frac{(M-N)!}{M!} \sum_{\bs{l} \in \bs{\c{L}}} E_{\bs{l}}^\top a_{\bs{l},\bs{\iota}}\,, \quad
		\mat{a_{\bs{l},\bs{\iota}} \\ b_{\bs{l},\bs{\iota}}} \in -\c{H}\,,
	\end{align*}
	for $\bs{l} \in \bs{\c{L}}$, $\bs{\iota} \in \bs{\hat{\c{I}}}$ and $j=1,\ldots,M$.
	In particular, the optimal value of the right hand side of \eqref{eq:polyhedralConcaveLossProgram} is independent of the choice of the selector $\upsilon$. 
\end{theorem}

\begin{remark}
	The program \eqref{eq:polyhedralConcaveLossProgram} involves (for $M \to \infty$)
	\begin{align*}
		O(\abs{\bs{\c{L}}} \cdot \abs{\bs{\hat{\c{I}}}} \cdot M) = O\left(\frac{M!}{(M-N)!} \cdot \binom{M+n_{\hat{P}}-1}{M} \cdot M\right) = O(M^{N + n_{\hat{P}}})
	\end{align*}
	decision variables and constraints.
	While this is polynomial in the lifting parameter $M$, it can be solved only for small values of $N$ and $n_{\hat{P}}$.
\end{remark}
%
%\subsubsection{Convex polyhedral loss}
%\todo[inline]{Convex polyhedral}
%
%
%\subsubsection{Quadratic loss}
%\todo[inline]{Quadratic loss}

\subsection{Numerical results} 

We conclude this paper by illustrating the proposed method on synthetic examples.

\subsubsection{Structured uncertainty quantification}
\label{sec:structuredUQNumerical}

Let $X = \b{R}$, $N = 2$, $c(x,y) = \norm{x-y}_2$ and $\ell(x) = \min \{ 2 x_1 + 5x_2, -5 x_1 + 2 x_2\}$, with $x = (x_1, x_2) \sim P^\star \otimes P^\star$, where $P^\star = 0.3 \delta_{-0.9} + 0.7 \delta_{1.1}$.
We assume that $P^\star$ is not known, and that instead we have a nominal distribution $\hat{P} = 0.25 \delta_{-1} + 0.75 \delta_{1}$, which could have been obtained from e.g. expert knowledge.
In this case it holds\footnote{In real practice this value is not known and the tuning of the radius of the ambiguity set is typically done via e.g. cross-validation} $W_c(P^\star, \hat{P}) \approx 0.19$. 
To hedge against the distributional ambiguity, we the structured ambiguity set $\c{W}$ \eqref{eq:structuredWassersteinAmbiguitySet} around $\hat{P}$ with $\rho = 0.2$. 
We compute the optimal value of the relaxed problem \eqref{eq:droGeneralLifted} for different values of the lifting parameter $M$ using Theorem \ref{thm:polyhedralConcaveLossProgram}. 
For comparison, we compute the optimal value of the unstructured DRO problem that uses the ambiguity set in \eqref{lem:wassersteinBallProductInclusion} with the same radius. 
The results are shown in Figure \ref{fig:numericsUQ}.
We observe that the relaxation gap decreases monotonically as $M \to \infty$, corroborating our theoretical findings. 
The substantial gap between the structured and the unstructured DRO value at $M = 50$ when a plateau is reached confirms that \eqref{lem:wassersteinBallProductInclusion} is a loose over-approximation of the original non-convex structured ambiguity set.

\begin{figure}
	\centering
	\includegraphics[width=0.7\textwidth]{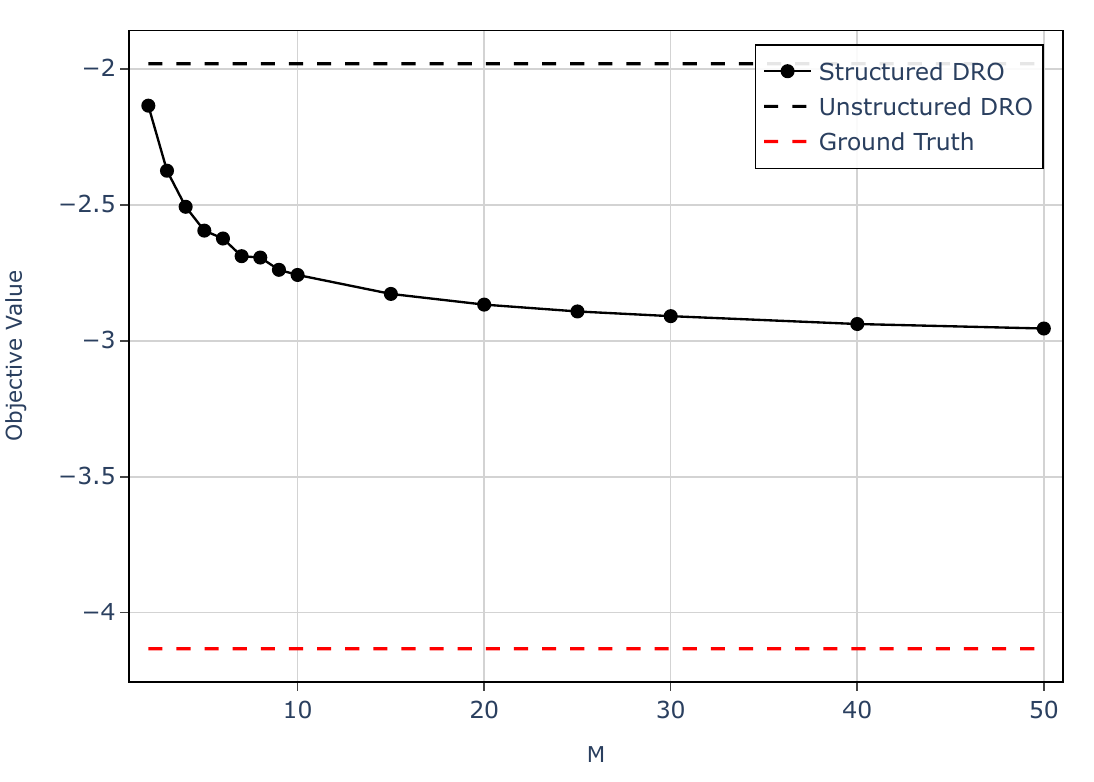}
	\caption{Relaxation gap of the structured DRO problem for different values of $M$ and $\rho$.}
	\label{fig:numericsUQ}
\end{figure}

\subsubsection{Structured distributionally robust optimization}
\label{sec:structuredDRONumerical}

Consider $X = \b{R}$, $N = 2$, $c(x,y) = \norm{x-y}_2$, $\Theta = [-3,3] \subseteq \b{R}$ as well as
\begin{align*}
	\ell(\theta,x)
	= \min_{h \in \c{H}(\theta)} h^\top \mat{x \\ 1} \text{\ \ for\ } \theta \in \Theta\,,\; x \in \b{R}^2\,,
\end{align*} 
where the polytope $\c{H}(\theta) \subseteq \b{R}^3$ is given by $\c{H}(\theta) = \{h \in \b{R}^3 \mid Wh \leq g(\theta)\}$ for 
\begin{align*}
	W = \mat{1 & 1 & -1 \\ -1 & 0 & 0 \\ 3/2 & -1/2 & -1/2 \\ 0 & 0 & 1}\,, \quad
	g(\theta) = \mat{1-\theta \\ \theta \\ 1/2-\theta \\ \theta}\,.
\end{align*}
It is easily verified that $\ell$ is a saddle function and satisfies the assumptions of Theorem \ref{thm:structuredDROOuterTightness}.
Solving the relaxed problems \eqref{eq:structuredDROOuterRelaxed} with $\hat{P} = \frac{1}{2}\delta_1 + \frac{1}{2} \delta_{-1}$ and $\rho = 1/4$ to optimality yields a sequence of decisions $(\theta_M^*)_{M \geq N} \subseteq \Theta$ and objective values $\Psi_{\operatorname{U}}^M(\theta_M^*) = \inf_\Theta \Psi_{\operatorname{U}}^M$, which are depicted in Figure \ref{fig:numericsDROObjectiveDecision}. 
We also depict the sequence of objective values $\Psi_{\operatorname{S}}(\theta_M^*)$ of the original problem \eqref{eq:stochasticOptimizationRobustStructured}\footnote{Since we cannot exactly compute $\Psi_{\operatorname{S}}$, we approximate it by $\Psi_{\operatorname{U}}^M$ for a large, fixed $M \geq N$}.
We observe that both $\Psi_{\operatorname{U}}^M(\theta_M^*)$ and $\Psi_{\operatorname{S}}(\theta_M^*)$ converge to a limiting value for $M \to \infty$, which, by Theorem \ref{thm:structuredDROOuterTightness}, is equal to $\inf_\Theta \Psi_{\operatorname{S}}$.
However, we can also see that the convergence is highly non-monotone in $M$, which motivates Remark \ref{rem:nonmonotonicRelaxationOuter}.

\begin{figure}
	\centering
	\subfigure[Relaxted objective values $\Psi_{\operatorname{U}}^M(\theta_M^*)$ and original objective values $\Psi_{\operatorname{S}}(\theta_M^*)$]{%
		\includegraphics[width=0.46\textwidth]{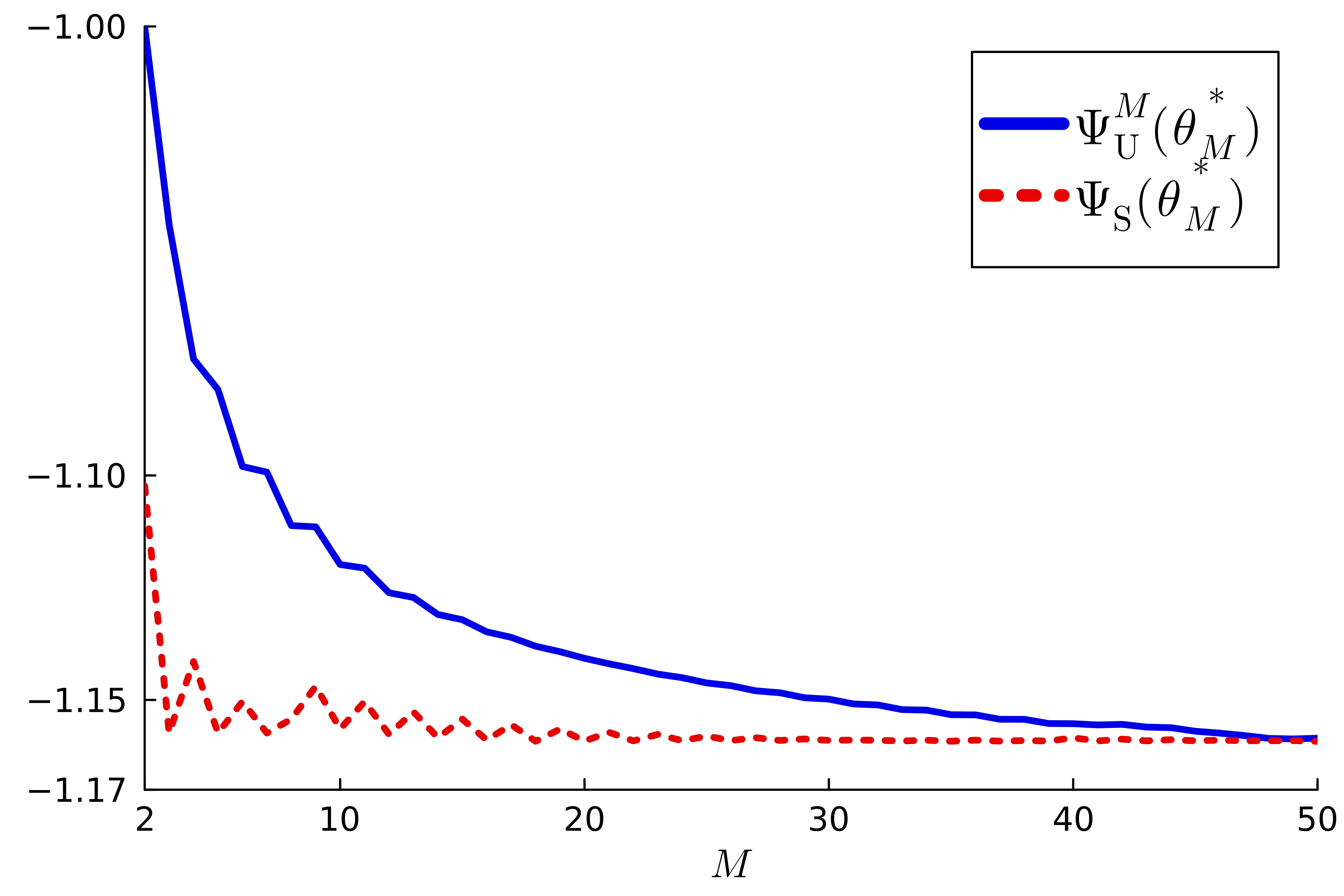}
	}
	\subfigure[Optimal decision vector $\theta_M^* \in \Theta$]{%
		\includegraphics[width=0.46\textwidth]{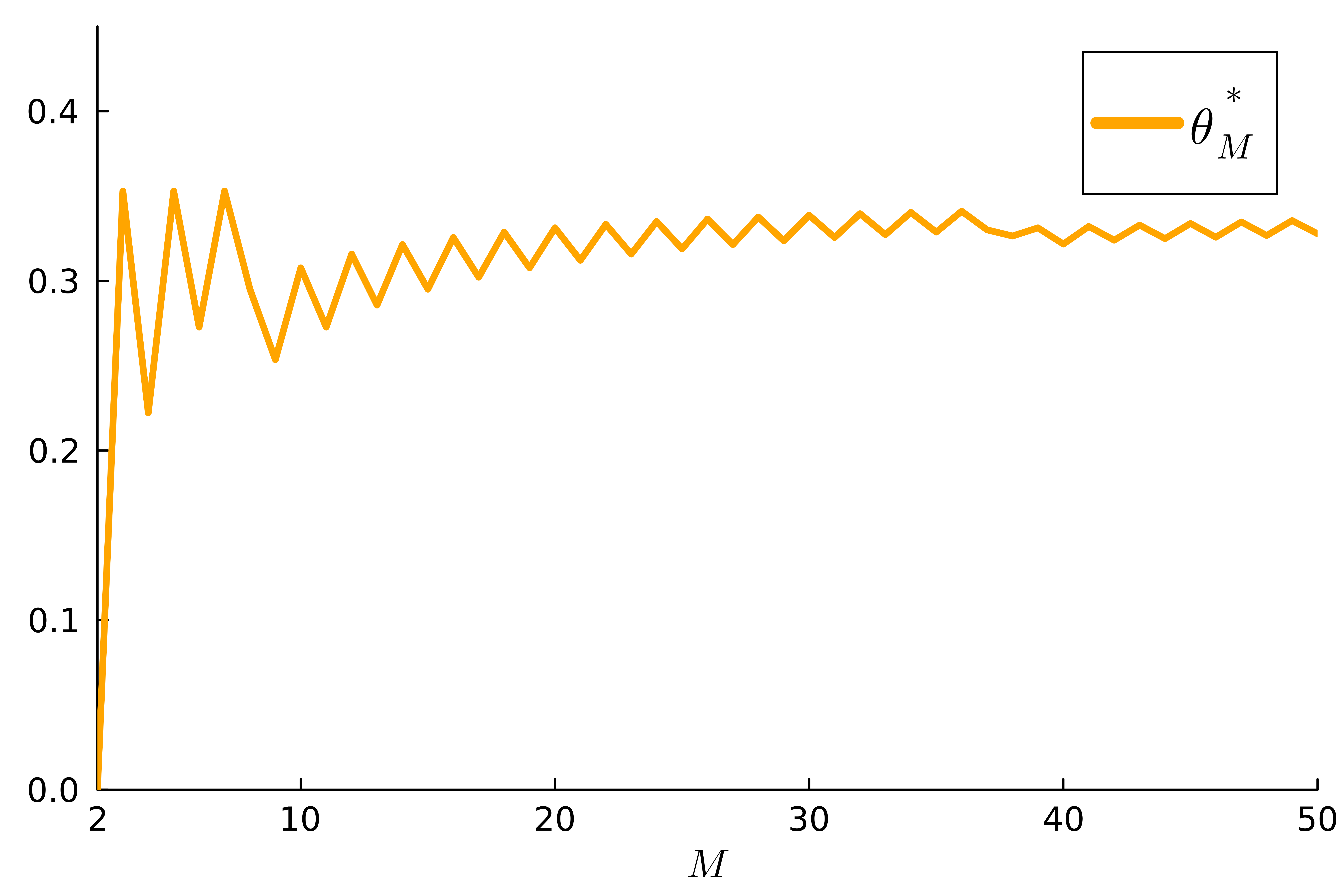}
	}
	\caption{Objective values and optimal decision vectors for the example from Section \ref{sec:structuredDRONumerical}.}
	\label{fig:numericsDROObjectiveDecision}
\end{figure}

\section{Conclusions and Outlook}
In this paper, we focused on Wasserstein DRO formulations where the uncertain vector exhibits an i.i.d. structure. 
By exploiting this structure, we construct a structured ambiguity set that only contains product distributions. 
To solve the resulting non-convex program, we devise a sequence of convex relaxations that, under mild conditions on the loss function, converge to the optimal solution of the original non-convex problem.
Our numerical results certify how structured ambiguity sets can capture uncertainty in a more effective manner than unstructured ambiguity sets, ultimately improving the overall decision-making. 
Future work includes (i) the analysis of the relaxation gap for convex loss functions, (ii) addressing the computational tractability issues of data-driven problems where the product of empirical distributions is supported on a exponential amount of points and finally (iii) extending the results to exploit invariance w.r.t. groups $G$ other than the symmetric group $\c{S}_N$.

\section*{Acknowledgments}

The first author would like to thank Ariel Neufeld (NTU Singapore), Zebang Shen (ETH Zurich) and Alexander Schell (ETH Zurich) for some helpful comments regarding different parts of the manuscript.

%\THEEndNotes
\begingroup \parindent 0pt \parskip 4ex
\def\enotesize{\normalsize} 
%\theendnotes
\endgroup

% Appendix here
% Options are (1) APPENDIX (with or without general title) or
%             (2) APPENDICES (if it has more than one unrelated sections)
% Outcomment the appropriate case if necessary
%
% \begin{APPENDIX}{<Title of the Appendix>}
% \end{APPENDIX}
%
%   or
%

\appendix

\input{Appendix.tex}

\begingroup
\renewcommand*{\bibfont}{\small}
\printbibliography
\endgroup

\end{document}

%% file: Appendix.tex
\section{Additional Notation}
\label{sec:additionalNotation}

In this appendix we use the following additional notation:
By $\operatorname{pr}_{i,j}:X^N \times X^N \to X \times X$ for $i,j=1,\ldots,N$ we denote the projection onto the $i$-th coordinate in the first factor $X^N$ and on the $j$-th coordinate in the second factor $X^N$.
The projections $\operatorname{pr}_{1:N,1:N}^M:X^M \times X^M \to X^N \times X^N$, $\operatorname{pr}_{1:N,j}^M:X^M \times X^M \to X^N \times X$ are defined similarly.
The special cases $\operatorname{pr}_{1:N,\emptyset}^M:X^M \times X^M \to X^N$ and $\operatorname{pr}_{\emptyset,1:N}^M:X^M \times X^M \to X^N$ denote the projection on the first and second factors only, respectively. 
If $M = \infty$ we drop the superscript in the previous notations.
For $P_1,\ldots,P_N \in \c{P}(X)$ we write $\Gamma(P_1,\ldots,P_N)$ for all $\Lambda \in \c{P}(X^N)$ such that $\operatorname{pr}_i \# \Lambda = P_i$ for $i=1,\ldots,N$.
The set $\Gamma(P_1,\ldots,P_N)$ is called the set of multi-marginals couplings between $P_1,\ldots,P_N$. 

\section{Background on measures in Polish spaces}

Here we collect some results on measures in Polish spaces.
First we note that if $X$ is a Polish space, then so is $\c{P}(X)$ with the topology of weak convergence of measures (i.e. convergence in distribution) \cite[15.15 Theorem]{guide2006infinite}.
%It can be shown that the set $\c{P}_p(X)$ is complete in the topology induced by the $p$-th Wasserstein metric $W_p$ \cite{villani2009optimal,panaretos2020invitation}, but that $\c{P}_p(X) \subseteq \c{P}(X)$ is in general not closed in the weak topology (e.g. there are distributions with finite $p$-th moment that converge in distribution to a probability measure without any moments, e.g. take the truncated Cauchy distribution).
While $\c{P}_c(X)$ is in general not closed as a subset of $\c{P}(X)$, it can be shown that, under the assumptions on the cost $c$ made in Section \ref{sec:preliminaries}, for $\hat{P} \in \c{P}_c(X)$ the function $\c{P}(X) \to [0,\infty]:P \mapsto W_c(P,\hat{P})$ is lower-semicontinuous in the weak topology and hence Borel measurable \cite[Paragraph right after Corollary 2.2.2]{panaretos2020invitation}.
This implies that for each $\rho > 0$ and $x_0 \in X$ the set $\c{P}_c(X,\rho) = \{P \mid W_c(\delta_{x_0},P) \leq \rho\}$ is closed as a subset of $\c{P}(X)$ (and independent of $x_0$ because of the weak triangle inequality for $c$) and hence $\c{P}_c(X) = \bigcup_{\rho \in \b{N}} \c{P}_c(X,\rho) \subseteq \c{P}(X)$ is Borel measurable as a countable union of closed sets.
Next, if $c$ is proper (i.e. metric sublevel sets are compact), then, for any $\rho > 0$, the set $\c{P}_c(X,\rho)$ is compact in the weak topology \cite{yue2022linear,panaretos2020invitation} as is the corresponding Wasserstein ball $\b{B}_\rho^c(P)$ whenever $P \in \c{P}_c(X)$ \cite{yue2022linear}.\\
Now we turn to mixtures of measures: Since $\c{P}(X)$ is Polish, we can define $\c{P}(\c{P}(X))$ as the set of Borel measures (the Borel $\sigma$-algebra on $\c{P}(X)$ is induced by the weak topology) on $\c{P}(X)$.
Now, given any Polish space $Y$ and a Borel subset $A \subseteq Y$ the map $\c{P}(Y) \to [0,1]: P \mapsto P(A) = \intg{}{\1_A}{P}$ is Borel measurable by \cite[Theorem 15.13]{guide2006infinite}.
Hence for any Borel measurable $F:\c{P}(X) \mapsto \c{P}(Y)$ the composition $\c{P}(X) \to [0,1]: P \mapsto F(P)(A)$ is Borel measurable as well. 
Then for any $\nu \in \c{P}(\c{P}(X))$ the expression $\bar{P}(A) := \intg{}{F(P)(A)}{\nu(P)}$ is well-defined and defines a map from the Borel subsets of $Y$ to $[0,1]$, which is finitely additive.
We show that $\bar{P}$ is $\sigma$-continuous from below \cite{klenke2013probability}, from which it follows that $\bar{P}$ is countably additive:
If $A, A_i \subseteq Y$ are Borel such that $A_i \nearrow A$ for $i \to \infty$, then $\intg{}{F(A_i)}{\nu(P)} \nearrow \intg{}{F(A)}{\nu(P)}$ for $i \to \infty$ by the monotone convergence theorem.
Hence $\bar{P}$ is a Borel measure and we write $\intg{}{F(P)}{\nu(P)} = \bar{P}$.
Furthermore, we often need the following identity that relates the integration w.r.t. a mixture distribution to the integration w.r.t. the mixture measure itself.
\begin{theorem}
	\label{thm:disintegrationMixture}
	For any Borel measurable $f:X \to \b{R}$, $\nu \in \c{P}(\c{P}(X))$ and $F:\c{P}(X) \to \c{P}(Y)$ the following holds: $f$ is $F(P)$-integrable for $\nu$-almost all $P \in \c{P}(X)$ iff $f$ is $\intg{}{F(P)}{\nu(P)}$-integrable and in this case 
	\begin{align*}
		\intg{}{f}{\left(\intg{}{F(P)}{\nu(P)}\right)}
		= \intg{}{\left(\intg{}{f}{(F(P))}\right)}{\nu(P)}\,.
	\end{align*}
\end{theorem}
\begin{proof}
	Let us define the map $\kappa:\c{P}(X) \times \c{B}(Y) \to [0,1]: (P,A) \mapsto F(P)(A)$, where $\c{B}(Y)$ denotes the Borel $\sigma$-algebra of $Y$. 
	Then $\kappa$ is a Markov (transition) kernel.
	Indeed, $\kappa(P,\cdot) = F(P)$ is a probability distribution for fixed $P \in \c{P}(X)$ by definition of $F$. 
	For $A \in \c{B}(Y)$ the map $\kappa(\cdot,A):\c{P}(X) \to [0,1]$ can be written as the composition $\kappa(\cdot,A) = \delta_A \circ F$, where $\delta_A:\c{P}(Y) \to [0,1]: Q \mapsto Q(A)$. 
	Since $F$ and $\delta_A$ are measurable by Lemma \ref{lem:powerMeasureWeaklyContinuous} and \cite[Theorem 15.13]{guide2006infinite}, respectively, so is $\kappa(\cdot,A)$.
	The claim then follows directly from \cite[Corollary 16]{schwartz1975lectures}.
\end{proof}

\section{Proofs of Section \ref{sec:preliminaries}}

\begin{proof}[Proof of Theorem \ref{thm:droGeneralSolvability}]
	If $\ell \in \c{G}_c$ and $c$ satisfies the weak triangle inequality, then by Lemma \ref{lem:wassersteinBallUniformlyBounded} it follows that
	\begin{align*}
		\b{E}_P \ell 
		\leq C \b{E}_P (1+c(x_0,\cdot))
		\leq C (1+\b{E}_P c(x_0,\cdot))
		\leq C \left(1+\sup_{\tilde{P} \in \b{B}_\rho^c(\hat{P})} \intg{}{c(x_0,\cdot)}{\tilde{P}} \right) < \infty
	\end{align*}
	for any $P \in \b{B}_\rho^c(\hat{P})$.
	Now suppose that $c$ is proper, $\ell \in \c{G}_c^<$ is upper semi-continuous.
	Since $c$ is proper, $\b{B}_\rho^c(\hat{P}) \subseteq \c{P}(X)$ is weakly compact by Lemma \ref{lem:properCostCompactWassersteinBall}.
	We will show that the map $P \mapsto \b{E}_P \ell$ can be uniformly approximated by weakly upper semi-continuous functions on $\b{B}_\rho^c(\hat{P})$, which will imply that its restriction to $\b{B}_\rho^c(\hat{P})$ is weakly upper semi-continuous itself and hence attains a maximum by the extreme value theorem.
	For this purpose, let $\delta \in \c{D}$, $x_0 \in X$ and $C > 0$ as in the definition of $\c{G}_c^<$.
	Define for any $r > 0$ the function $\ell_r = \min\{\ell,C(1+r)\}$.
	Clearly then for any $x \in X$ it holds that
	\begin{align*}
		\ell(x) - \ell_r(x)
		&\leq 
		\begin{cases}
			0\,, & c(x_0,x) \leq r \\
			\max\{0,\ell(x) - C(1+r)\}\,, & c(x_0,x) > r
		\end{cases}
		\\
		&\leq 
		\begin{cases}
			0\,, & c(x_0,x) \leq r \\
			C \delta(c(x_0,x))\,, & c(x_0,x) > r
		\end{cases}\,,
	\end{align*}
	where the first follows from $\ell \leq C (1+c(x_0,\cdot))$ and the second from the stronger $\ell \leq C (1+\delta(c(x_0,\cdot)))$.
	Hence if $P \in \b{B}_\rho^c(\hat{P})$ it follows for $r \to \infty$ that
	\begin{gather*}
		\abs{\b{E}_P\ell - \b{E}_P\ell_r}
		\leq \b{E}_P \abs{\ell - \ell_r}
		\leq C\intg{c(x_0,\cdot) > r}{\delta(c(x_0,\cdot))}{P} \\
		= C\intg{c(x_0,\cdot) > r}{\frac{\delta(c(x_0,\cdot))}{c(x_0,\cdot)}c(x_0,\cdot)}{P} 
		\leq C\left(\sup_{\substack{x \in X \\ c(x_0,x) > r}} \frac{\delta(c(x_0,x))}{c(x_0,x)} \right) \intg{X}{c(x_0,\cdot)}{P} \\
		\leq C \underbrace{\left(\sup_{t > r} \frac{\delta(t)}{t}\right)}_{\to 0} \underbrace{\left(\sup_{P \in \b{B}_\rho^c(\hat{P})}  \intg{X}{c(x_0,\cdot)}{P}\right)}_{< \infty}\,,
	\end{gather*}
	with the last claims following from the definition of $\c{D}$ and Lemma \ref{lem:wassersteinBallUniformlyBounded}.
	Hence the maps $P \mapsto \b{E}_P \ell$ can be uniformly approximated by maps of the form $P \mapsto \b{E}_P \ell_r$ on $\b{B}_\rho^c(\hat{P})$.
	By \cite[15.5 Theorem]{guide2006infinite} the latter maps are weakly upper semi-continuous, which finishes the proof.
\end{proof}

\section{Proof of Section \ref{sec:structuredDROExistenceFiniteness}}

\begin{proof}[Proof of Theorem \ref{thm:droStructuredSolvability}]
	The proof is similar to the proof of Theorem \ref{thm:droGeneralSolvability}.
	If $\ell \in \hat{\c{G}}_{c,N}$ and $c$ satisfies the weak triangle inequality, then by Lemma \ref{lem:wassersteinBallUniformlyBounded} it follows that
	\begin{align*}
		\b{E}_{P^{\otimes N}} \ell 
		&\leq C \b{E}_{x \sim P^{\otimes N}} \prod_{k=1}^N (1+c(x_0,x_k)) \\
		&\leq C \prod_{k=1}^N (1+\b{E}_{x_k \sim P} c(x_0,x_k)) 
		\leq C \left(1+\sup_{P \in \b{B}_\rho^c(\hat{P})} \intg{}{c(x_0,\cdot)}{P}\right)^N
		< \infty
	\end{align*}
	for any $P \in \b{B}_\rho^c(\hat{P})$.
	Now suppose that $c$ is proper, $\ell \in \hat{\c{G}}_{c,N}^<$ is upper semi-continuous.
	Then $\b{B}_\rho^c(\hat{P}) \subseteq \c{P}(X)$ is weakly compact by Lemma \ref{lem:properCostCompactWassersteinBall}.
	We will show that the map $P \mapsto \b{E}_{P^{\otimes N}} \ell$ can be uniformly approximated by weakly upper semi-continuous functions on $\b{B}_\rho^c(\hat{P})$, which will imply that its restriction to $\b{B}_\rho^c(\hat{P})$ is weakly upper semi-continuous and hence attains its maximum.
	For this purpose, let $\delta \in \c{D}$, $x_0 \in X$ and $C > 0$ as in the definition of $\hat{\c{G}}_{c,N}^<$.
	Moreover, let $\hat{C} = \sup_{P \in \b{B}_\rho^c(\hat{P})} \intg{}{c(x_0,\cdot)}{P}$, which is finite by Lemma \ref{lem:wassersteinBallUniformlyBounded}.
	Define for any $r > 0$ the function $\ell_r:X^N \to \b{R}$ by $\ell_r(x) = \min\{\ell(x),C(1+r)^N\}$ for $x \in X^N$.
	Let $\c{I} = 2^{\{1,\ldots,N\}}$ be the power set of $\{1,\ldots,N\}$ and consider for any $r > 0$ the following function
	\begin{align*}
		I_r:X^N \to \c{I}: x \mapsto \{i \in \{1,\ldots,N\} \mid c(x_0,x_i) > r\}\,.
	\end{align*}
	Let us partition $X^N$ into $X^N = \bigsqcup_{I \in \c{I}} I_r^{-1}(I)$.
	Note that each $I_r^{-1}(I)$ can be written as $I_r^{-1}(I) = X_{I,1} \times \cdots \times X_{I,N}$ with $X_{I,i} = \{x_i \in X \mid c(x_0,x_i) > r\}$ if $i \in I$ and $X_{I,i} = \{x_i \in X \mid c(x_0,x_i) \leq r\}$ if $i \notin I$.
	Now, if $x \in I_r^{-1}(I) \cap \{\ell > C(1+r)^N\}$, then we have 
	\begin{align*}
		&\ell(x) - \ell_r(x) \\
		&\leq C\prod_{k=1}^N (1+\delta(c(x_0,x_k))) - C(1+r)^N \\
		&\leq C\left[\prod_{k \in I} (1+\delta(c(x_0,x_k)))\prod_{k \notin I} (1+c(x_0,x_k)) - (1+r)^N\right] \\
		&\leq C\left[\prod_{k \in I} (1+\delta(c(x_0,x_k)))\prod_{k \notin I} (1+c(x_0,x_k)) - (1+r)^{\abs{I}}\prod_{k \notin I} (1+c(x_0,x_k))\right] \\
		&= C\left[\prod_{k \in I} (1+\delta(c(x_0,x_k))) - (1+r)^{\abs{I}}\right] \prod_{k \notin I} (1+c(x_0,x_k)) \\
		&\leq C\left[\prod_{k \in I} (1+\delta(c(x_0,x_k))) - 1\right] \prod_{k \notin I} (1+c(x_0,x_k)) \,.
	\end{align*}
	On the other hand, if $x \in \{\ell \leq C(1+r)^N\}$, then $\ell(x) - \ell_r(x) = 0$.
	Note that $I_r^{-1}(\emptyset) \subseteq \{\ell \leq C(1+r)^N\}$.
	Integrating w.r.t. $P \in \b{B}_\rho^c(\hat{P})$ yields
	\begin{align*}
		\intg{X^N}{\ell(x) - \ell_r(x)}{P^{\otimes N}(x)}
		\leq \sum_{I \in \c{I}} \intg{I_r^{-1}(I)}{\ell(x) - \ell_r(x)}{P^{\otimes N}(x)}\,.
	\end{align*}
	Let us estimate every summand as follows
	\begin{gather*}
		\intg{I_r^{-1}(I)}{\ell(x) - \ell_r(x)}{P^{\otimes N}(x)}
		= \intg{I_r^{-1}(I) \cap \{\ell > C(1+r)^N\}}{\ell(x) - \ell_r(x)}{P^{\otimes N}(x)} \\
		\leq C \intg{I_r^{-1}(I)}{\left[\prod_{k \in I} (1+\delta(c(x_0,x_k))) - 1\right] \prod_{k \notin I} (1+c(x_0,x_k))}{P^{\otimes N}(x)}	\,.	
	\end{gather*}
	Let us use the representation $I_r^{-1}(I) = X_{I,1} \times \cdots \times X_{I,N}$ to obtain (here $I^c$ is the complement of $I$)
	\begin{align*}
		&\intg{I_r^{-1}(I)}{\left[\prod_{k \in I} (1+\delta(c(x_0,x_k))) - 1\right] \prod_{k \notin I} (1+c(x_0,x_k))}{P^{\otimes N}(x)} \\
		&\leq \left(\intg{\prod_{k \in I} X_{I,k}}{\left[\prod_{k \in I} (1+\delta(c(x_0,x_k))) - 1\right]}{P^{\otimes \abs{I}}(x_I)} \right) \\
		&\qquad \cdot \left( \intg{\prod_{k \notin I} X_{I,k}}{\prod_{k \notin I} (1+c(x_0,x_k))}{P^{\otimes \abs{I^c}}(x_{I^c})}\right) \\
		&\leq \left(\intg{\prod_{k \in I} X_{I,k}}{\left[\prod_{k \in I} (1+\delta(c(x_0,x_k))) - 1\right]}{P^{\otimes \abs{I}}(x_I)} \right)(1+\hat{C})^{\abs{I^c}}\,.
	\end{align*}
	To estimate the first factor, we note that
	\begin{align*}
		\prod_{k \in I} (1+\delta(c(x_0,x_k))) - 1
		= \sum_{J \in 2^I \setminus \{\emptyset\}} \prod_{k \in J} \delta(c(x_0,x_k)) \,.
	\end{align*}
	and that for any $J \in 2^I \setminus \{\emptyset\}$ it holds that
	\begin{align*}
		\intg{\prod_{k \in I} X_{I,k}}{\prod_{k \in J} \delta(c(x_0,x_k))}{P^{\otimes \abs{I}}(x_I)} 
		= \prod_{k \in J} \intg{X_{I,k}}{\delta(c(x_0,x_k))}{P(x_k)}
	\end{align*}
	and, since $k \in J \subseteq I$, also
	\begin{align*}
		\intg{X_{I,k}}{\delta(c(x_0,x_k))}{P(x_k)}
		&= \intg{c(x_0,\hat{x}) > r}{\frac{\delta(c(x_0,\hat{x}))}{c(x_0,\hat{x})}c(x_0,\hat{x})}{P(\hat{x})} \\
		&\leq \left(\sup_{t > r} \frac{\delta(t)}{t}\right)\intg{}{c(x_0,\hat{x})}{P(\hat{x})}\\
		&\leq \hat{C}\hat{\delta}(r)\,,
	\end{align*}
	where we have defined $\hat{\delta}(r) =\sup_{t > r} \frac{\delta(t)}{t}$.
	In all we obtain uniformly for $P \in \b{B}_\rho^c(\hat{P})$
	\begin{align*}
		\abs{\b{E}_{P^{\otimes N}} \ell - \b{E}_{P^{\otimes N}} \ell_r}
		&\leq \intg{X^N}{(\ell(x) - \ell_r(x))}{P^{\otimes N}(x)} \\
		&\leq C \sum_{I \in \c{I}} \left((1+\hat{C})^{N-\abs{I}} \sum_{J \in 2^I \setminus\{\emptyset\}} \hat{C}^{\abs{J}}\hat{\delta}(r)^{\abs{J}}\right)\,,
	\end{align*}
	and the right hand side converges to $0$ as $r \to \infty$, since $\delta \in \c{D}$.
	Hence the maps $P \mapsto \b{E}_{P^{\otimes N}} \ell$ can be uniformly approximated by maps of the form $P \mapsto \b{E}_{P^{\otimes N}} \ell_r$ on $\b{B}_\rho^c(\hat{P})$.
	By \cite[15.5 Theorem]{guide2006infinite} and Theorem \ref{lem:powerMeasureWeaklyContinuous} the latter maps are weakly upper semi-continuous, which finishes the proof.
\end{proof}

\begin{proof}[Proof of Lemma \ref{lem:wassersteinBallProductInclusion}]
	Let $P \in \b{B}_\rho^c(\hat{P})$.
	Then $W_c(P,\hat{P}) \leq \rho$ and by Lemma \ref{lem:wassersteinDistanceProduct} it follows that $W_{c^N}(P^{\otimes N},\hat{P}^{\otimes N}) \leq N \rho$.
	Hence $P^{\otimes N} \in \b{B}_{N \rho}^{c^N}(\hat{P}^{\otimes N})$.
\end{proof}

\section{Proofs of Section \ref{sec:structuredDROConvexUpperBound}}

\begin{proof}[Proof of Lemma \ref{lem:convexHullClosedDRO}]
	The first equality is just a consequence of $\bar{P} \mapsto \b{E}_{x \sim \bar{P}} \ell(x)$ being linear.
	To show the second equality, we first note that $\leq$ is clearly true, since $\operatorname{conv} \c{W} \subseteq \convw \c{W}$.
	Therefore we actually only need to show $\geq$. 
	However, we will show equality directly by using Lemma \ref{lem:integralConvexHull} stated below.
	By Lemma \ref{lem:integrationFunctionalBorel} applied to the space $X^\infty$ and the Borel measurable $\ell\circ \operatorname{pr}_{1:N}:X^\infty \to \b{R}$, the function $F$ defined by $F:P \mapsto \intg{}{\ell \circ \operatorname{pr}_{1:N}}{P^{\otimes \infty}}$ whenever $\ell \circ \operatorname{pr}_{1:N}$ is $P^{\otimes \infty}$-integrable, is itself Borel measurable.
	Let $\bar{P} = \intg{}{P^{\otimes \infty}}{\nu(P)} \in \convw \c{W}$.
	Then 
	\begin{align*}
		\intg{}{\ell\circ \operatorname{pr}_{1:N}}{\bar{P}}
		= \intg{}{F(P)}{\nu(P)}\,,
	\end{align*}
	if $F$ is $\nu$-integrable.
	Then by Lemma \ref{lem:existenceFinitelySupportedMeasure} there exists a finitely supported measure $\hat{\nu} \in \c{P}(\b{B}_\rho^c(\hat{P}))$ such that
	\begin{align*}
		\intg{}{F(P)}{\nu(P)} 
		= \intg{}{F(P)}{\hat{\nu}(P)}
		= \intg{}{\ell\circ \operatorname{pr}_{1:N}}{\check{P}}\,,
	\end{align*}
	where $\check{P} = \intg{}{P^{\otimes \infty}}{\hat{\nu}(P)} \in \operatorname{conv} \c{W}$, i.e. we can always substitute a general mixing measure $\nu$ by a finitely supported one without changing the value of the expectation.
	This shows that equality in this lemma holds provided that $F$ is $\nu$-integrable, i.e. $\ell\circ \operatorname{pr}_{1:N}$ is $P^{\otimes \infty}$-integrable for $\nu$-a.a. $P \in \c{P}(X)$.
	Since $\operatorname{supp} \nu \subseteq \b{B}_\rho^c(\hat{P})$, we see that it is sufficient to require that $\ell\circ \operatorname{pr}_{1:N}$ is $P^{\otimes \infty}$-integrable for all $P \in \b{B}_\rho^c(\hat{P})$ or equivalently that $\ell$ is $P^{\otimes N}$-integrable for all $P \in \b{B}_\rho^c(\hat{P})$.
\end{proof}

\section{Proofs of Section \ref{sec:structuredDROTighterConvexUpperBound}}

\begin{proof}[Proof of Lemma \ref{lem:nonlinearDROUpperBoundOptimization}]
	First we claim that the infimum is attained.
	Indeed, for each $\bar{P} \in \b{B}_{N\rho}^{c^N}(\hat{P}^{\otimes N})$ the function $\alpha \mapsto \intg{}{F_\alpha(\ell)}{\bar{P}}$ is lower-semicontinuous (in fact, its just linear in $\alpha$).
	Since the supremum of lower semi-continuous functions is again lower semi-continuous \cite{willard2012general}, it follows that $\alpha \mapsto \operatorname{U}(F_\alpha(\ell))$ attains its minimum over the compact set $\Delta^{\c{S}_N}$.
	Now assume that the minimum is attained in some $\alpha^* \in \Delta^{\c{S}_N}$.
	%Now, if $\bar{P} \in \b{B}_{N\rho}^c(P^{\otimes N})$, then $\bar{P}_{\operatorname{sym}} := \frac{1}{N!}\sum_{\pi \in \c{S}_N} \pi\#\bar{P} \in \b{B}_{N\rho}^c(P^{\otimes N})$.
	Given any $\bar{P} \in \b{B}_{N\rho}^{c^N}(\hat{P}^{\otimes N})$, define the measure $\bar{P}_{\operatorname{sym}}$ by \eqref{eq:symmetrizationMeasure}.
	Because $\bar{P}_{\operatorname{sym}}$ is always symmetric and belongs to $\b{B}_{N\rho}^{c^N}(\hat{P}^{\otimes N})$ by Lemma \ref{lem:symmetricCoupling}, it holds that 
	\begin{align*}
		\operatorname{U}(F_{\alpha^*}(\ell)) 
		\geq \intg{}{F_{\alpha^*}(\ell)}{\bar{P}_{\operatorname{sym}}} 
		= \sum_{\pi \in \c{S}_N} \alpha_\pi^* \intg{}{\ell_\pi}{\bar{P}_{\operatorname{sym}}}
		= \sum_{\pi \in \c{S}_N} \alpha_\pi^* \intg{}{\ell}{\bar{P}_{\operatorname{sym}}}
		= \intg{}{\ell}{\bar{P}_{\operatorname{sym}}}\,.
	\end{align*}
	As such, $\sup_{\bar{P} \in \b{B}_{N\rho}^{c^N}(\hat{P}^{\otimes N})} \intg{}{\ell}{\bar{P}_{\operatorname{sym}}}$ is a lower bound for $\operatorname{U}(F_{\alpha^*}(\ell))$.
	But clearly for the particular $\alpha = (1/N!,\ldots,1/N!)$ it holds that (see \eqref{eq:symmetrizationPushforwardIdentity})
	\begin{align*}
		\operatorname{U}(F_{\alpha}(\ell)) = \sup_{\bar{P} \in \b{B}_{N\rho}^{c^N}(\hat{P}^{\otimes N})} \intg{}{\ell}{\bar{P}_{\operatorname{sym}}}	
	\end{align*}
	 i.e. thus the infimum is also attained in $\alpha$.
\end{proof}

\begin{proof}[Proof of Example \ref{ex:symmetrization}]
	For the first part we have $\ell_{\operatorname{sym}}(x_1,x_2) = -x_1^2 -2x_1 x_2 - x_2^2$ and thus
	\begin{align*}
		\operatorname{S}(\ell_{\operatorname{sym}})
		= \sup_{P \in \b{B}_\rho^c(\hat{P})} -2(\b{E}(P^2) + (\b{E}P)^2)
	\end{align*}
	is clear.
	Furthermore, by Theorem \ref{thm:droGeneralDuality} we see that
	\begin{align*}
		\sup_{P \in \b{B}_\rho^c(\hat{P})} -2\b{E}(P^2)
		= \inf_{\mu \geq 0} \rho \mu + \b{E}_{z \sim \hat{P}} \phi_\mu(z)\,,
	\end{align*}
	with $\phi_\mu(z) = \sup_{x \in \b{R}} (-2x^2 - \mu (x-z)^2)$ satisfying $\phi_\mu(z) = \phi_\mu(-z)$.
	Using $\hat{P} = \frac{1}{2}\delta_1 + \frac{1}{2}\delta_{-1}$ we also note that 
	\begin{align*}
		\b{E}_{z \sim \hat{P}} \phi_\mu(z)
		= \frac{1}{2} \phi_\mu(1) + \frac{1}{2} \phi_\mu(-1)
		= \phi_\mu(1)
		= -\frac{2\mu}{2+\mu}
	\end{align*}
	where the last equality follows by elementary computations.
	Then 
	\begin{align*}
		\inf_{\mu \geq 0} \rho \mu + \b{E}_{z \sim \hat{P}} \phi_\mu(z) 
		&= \inf_{\mu \geq 0} \rho \mu - \frac{2\mu}{2+\mu} \\
		&=
		\begin{cases}
			 -2(1-\sqrt{\rho})^2\,, & \rho \leq 1\,, \\
			 0\,, & \rho > 1\,,
		\end{cases} \\
		&= -2(1-\min(\sqrt{\rho},1))^2
	\end{align*}
	where the second equality follows again by elementary calculations.
	This establishes that $\operatorname{S}(\ell_{\operatorname{sym}}) \leq -2(1-\min(\sqrt{\rho},1))^2$.
	Now, the distribution $P = \frac{1}{2} \delta_{1-\min(\sqrt{\rho},1)} + \frac{1}{2} \delta_{-(1-\min(\sqrt{\rho},1))}$ is contained in the Wasserstein ball $\b{B}_\rho^c(\hat{P})$ and satisfies $-2\b{E}(P^2) = -2(1-\min(\sqrt{\rho},1))^2$ and $\b{E}P = 0$, which shows that actually $\operatorname{S}(\ell_{\operatorname{sym}}) = -2(1-\min(\sqrt{\rho},1))^2$.
	The calculation of $\operatorname{U}(\ell)$ and $\operatorname{U}(\ell_{\operatorname{sym}})$ follows by similar arguments and Theorem \ref{thm:droGeneralDuality}.
\end{proof}

\begin{proof}[Proof of Theorem \ref{thm:symmetrizationUpperBound}]
We have
\begin{align}
	\label{eq:symmetrizationPushforwardIdentity}
	\begin{gathered}
		\operatorname{U}(\ell_{\operatorname{sym}})
		= \sup_{\bar{P} \in \b{B}_{N\rho}^{c^N}(\hat{P}^{\otimes N})}\frac{1}{N!}\sum_{\pi \in \c{S}_N} \intg{}{\ell\circ \pi}{\bar{P}} 
		= \sup_{\bar{P} \in \b{B}_{N\rho}^{c^N}(\hat{P}^{\otimes N})}\frac{1}{N!}\sum_{\pi \in \c{S}_N} \intg{}{\ell}{\pi \# \bar{P}} \\
		= \sup_{\bar{P} \in \b{B}_{N\rho}^{c^N}(\hat{P}^{\otimes N})} \intg{}{\ell}{\bar{P}_{\operatorname{sym}}} 
		= \sup_{\bar{P} \in \b{B}_{N\rho}^{c^N}(\hat{P}^{\otimes N})_{\operatorname{sym}}} \intg{}{\ell}{\bar{P}}\,,
	\end{gathered}
\end{align}
where the first equality follows by the definition of the symmetrization of a function $\ell$ and the linearity of the integral, the second equality from the pushforward identity of the Lebesgue integral and the last equality from the linearity of the integral in the measure and the definition \eqref{eq:symmetrizationMeasure}.
Now, to see \eqref{eq:symmetricSet}, we first note that $\supseteq$ trivially holds, since for any $\bar{P} \in \b{B}_{N\rho}^{c^N}(\hat{P}^{\otimes N}) \cap \c{P}_{\operatorname{sym}}(X^N)$ we have $\bar{P}_{\operatorname{sym}} = \bar{P}$.
To see $\subseteq$, pick any $\bar{P} \in \b{B}_{N\rho}^{c^N}(\hat{P}^{\otimes N})$.
Let $\epsilon > 0$ be arbitrary and let $\Lambda \in \Gamma(\bar{P},\hat{P}^{\otimes N})$ be a coupling such that $\intg{}{c^N}{\Lambda} \leq N\rho + \epsilon$.
Then by Lemma \ref{lem:symmetricCoupling} there exists some $\Lambda_{\operatorname{sym}} \in \Gamma(\bar{P}_{\operatorname{sym}},\hat{P}^{\otimes N})$ with $\intg{}{c^N}{\Lambda_{\operatorname{sym}}} \leq N\rho + \epsilon$.
Since $\epsilon > 0$ was arbitrary, it follows that $\bar{P}_{\operatorname{sym}} \in \b{B}_{N\rho}^{c^N}(P^{\otimes N})$.
\end{proof}

\section{Proofs of Section \ref{sec:structuredDROSequenceOfConvexRelaxations}}

\begin{proof}[Proof of Theorem \ref{thm:droRelaxedSolvability}]
%	In view of Theorem \ref{thm:droGeneralSolvability} and the definition of $\operatorname{U}_M^{\operatorname{sym}}(\ell)$ it is sufficient to show the following two implications: $\ell \in \c{G}_{c^N}(X^N)$ (resp. $\in \c{G}_{c^N}^<(X^N)$) implies $(\ell \circ \operatorname{pr}_{1:N}^M)_{\operatorname{sym}} \in \c{G}_{c^M}(X^M)$ (resp. $\in \c{G}_{c^M}^<(X^M)$) for all $M \geq N$.
%	Indeed, if $c$ is proper and $\hat{P} \in \c{P}_c(X)$, then $c^M$ is proper and $\hat{P}^{\otimes M} \in \c{P}_{c^M}(X^M)$, after which one can apply Theorem \ref{thm:droGeneralSolvability} to $(\ell \circ \operatorname{pr}_{1:N}^M)_{\operatorname{sym}}$.
	We show the implication involving $\c{G}_{c^N}^<(X^N)$ only, since the implication involving $\c{G}_{c^N}(X^N)$ is shown analogously.
	Suppose that $\ell \in \c{G}_{c^N}^<(X^N)$, i.e. there exists some $\delta \in \c{D}$, $\hat{x} \in X^N$ and $C > 0$ such that $\ell(x) \leq C(1+\delta(c^N(x,\hat{x})))$ for all $x \in X^N$.
	Without the loss of generality we can assume that $\delta$ is monotone, as otherwise we can replace it by $\tilde{\delta}(r) = \sup_{t \in [0,r]} \delta(t)$, $\tilde{\delta} \in \c{D}$.
	Let $x_0 \in X$ be an arbitrary point and define $\bar{x} \in X^M$ by $\bar{x}_i = \hat{x}_i$ if $i \in \{1,\ldots,N\}$ and $\bar{x}_i = x_0$ for $i \in \{N+1,\ldots,M\}$.  
	Then
	\begin{align*}
		c^N(\operatorname{pr}_{1:N}^M(x),\hat{x})
		= \sum_{i=1}^N c(x_i,\hat{x}_i)
		\leq \sum_{i=1}^M c(x_i,\bar{x}_i)
		= c^M(x,\bar{x})\,,
	\end{align*}
	and thus
	\begin{align*}
		\ell \circ \operatorname{pr}_{1:N}^M(x) 
		\leq C(1+\delta(c^N(\operatorname{pr}_{1:N}^M(x),\hat{x})))
		\leq C(1+\delta(c^M(x,\bar{x})))
	\end{align*}
	for any $x \in X^M$, i.e. $\ell \circ \operatorname{pr}_{1:N}^M \in \c{G}_{c^M}^<(X^M)$.
	Moreover, since $c$ is proper and $\hat{P} \in \c{P}_c(X)$, the lifted cost $c^M$ is proper and $\hat{P}^{\otimes M} \in \c{P}_{c^M}(X^M)$.
	This implies, as shown in the proof of Theorem \ref{thm:droGeneralSolvability}, that $\bar{P} \mapsto \intg{}{\ell \circ \operatorname{pr}_{1:N}^M}{\bar{P}}$ is upper semi-continuous on the compact set $\b{B}_{M\rho}^{c^M}(\hat{P}^{\otimes M})$.
	Because $\b{B}_{M\rho}^{c^M}(\hat{P}^{\otimes M})_{\operatorname{sym}}$ is a closed subset, the former map attains its maximum therein, which concludes the proof.
\end{proof}

\begin{proof}[Proof of Example \ref{ex:liftedRelaxation}]
	We show the first equality, since the second inequality is obvious and the third is a straightforward one-dimensional optimization problem.
	In this proof we denote by $\1 \in \b{R}^M$ the all-ones vector.
	First we note that for the given $\ell$ we have for $x \in \b{R}^M$ that 
	\begin{align*}
		(\ell\circ \operatorname{pr}_{1:2}^M)_{\operatorname{sym}}(x)
		= \frac{(M-2)!}{M!}\sum_{\substack{l_1,l_2=1 \\ l_1 \neq l_2}}^M -x_{l_1}x_{l_2} 
		= \frac{1}{M(M-1)} x^\top (I-\1 \1^\top) x		
	\end{align*}
	and hence for any $\hat{\xi} \in \b{R}^M$
	\begin{align*}
		(\ell\circ \operatorname{pr}_{1:2}^M)_{\operatorname{sym}}(x) - \mu \norm{x-\hat{\xi}}^2
		= \mat{x \\ 1}^\top \mat{\frac{1}{M(M-1)}(I-\1\1^\top) - \mu I & \mu \hat{\xi} \\ \mu \hat{\xi}^\top & -\mu M}\mat{x \\ 1}\,.
	\end{align*}
	Now, noting that $\hat{P}^{\otimes M} = \frac{1}{2^M} \sum_{\hat{\xi} \in \{\pm 1\}^M} \delta_{\hat{\xi}}$, we obtain
	\begin{align}
		\label{eq:exampleLiftingSupremum}
		\begin{aligned}
		&\inf_{\mu \geq 0} M \rho \mu + \b{E}_{\hat{\xi} \sim \hat{P}^{\otimes M}} \phi_\mu(\hat{\xi}) \\
		&\quad= \inf_{\mu \geq 0} M \rho \mu + \frac{1}{2^M} \sum_{\hat{\xi} \in \{\pm 1\}^M} \sup_{x \in \b{R}^M} \mat{x \\ 1}^\top \mat{\frac{1}{M(M-1)}(I-\1\1^\top) - \mu I & \mu \hat{\xi} \\ \mu \hat{\xi}^\top & -\mu M}\mat{x \\ 1}
		\end{aligned}
	\end{align}
	and thus we need to evaluate the inner supremum.
	For this we use the following, elementary lemma, the proof of which we omit.
	\begin{lemma}
		Let $q: \b{R}^M \to \b{R}: x \mapsto \smat{x \\ 1}^\top \smat{Q & s \\ s^\top & r} \smat{x \\ 1}$ be a quadratic function.
		Then $q_* = \sup_{x \in \b{R}^M} q(x) < \infty$ iff $Q \leq 0$ and $\operatorname{ker} Q \subseteq \operatorname{ker} s^\top$. 
		In this case the supremum is given by $q_* = -s^\top y_* + r$ with $y_*$ being any solution to $Qy_* = s$.
		If $Q < 0$, then $q_* = -s^\top Q^{-1} s + r$
	\end{lemma}
	Thus the supremum in \eqref{eq:exampleLiftingSupremum} is finite only if $\frac{1}{M(M-1)}(I-\1\1^\top) - \mu I \leq 0$, which is equivalent to $\mu \geq \frac{1}{M(M-1)}$.
	If equality $\mu = \frac{1}{M(M-1)}$ holds, then the second condition in the lemma can only be satisfied if $\hat{\xi} \in \operatorname{ran}(\1 \1^\top)$, i.e. if $\hat{\xi} \in \{\1,-\1\}$ and thus the total sum is infinite as $M \geq 2$.
	Hence we only need to consider $\mu > \frac{1}{M(M-1)}$.
	Using the Sherman-Morrison identity we obtain 
	\begin{align*}
		\left(\tfrac{1}{M(M-1)}(I-\1\1^\top) - \mu I\right)^{-1}
 		= \left(\tfrac{1}{M(M-1)} - \mu\right)^{-1} \left(I - \tfrac{1}{(M-1)(1+M\mu)}\1 \1^\top\right)\,,
	\end{align*}
	and thus by the lemma the supremum is 
	\begin{align*}
		&\left(\mu-\tfrac{1}{M(M-1)}\right)^{-1} \mu^2 \hat{\xi}^\top \left(I - \tfrac{1}{(M-1)(1+M\mu)}\1 \1^\top\right)\hat{\xi} - \mu M \\
		&\quad=  \tfrac{M(M-1)}{M(M-1)\mu-1} \mu^2 \left(I - \tfrac{1}{(M-1)(1+M\mu)}\abs{\hat{\xi}^\top \1}^2\right) - \mu M\,.
	\end{align*}
	Thus, substituting the above quantity into \eqref{eq:exampleLiftingSupremum} yields
	\begin{align*}
		\inf_{\mu > \frac{1}{M(M-1)}} M(\rho-1) \mu +  \tfrac{M(M-1)}{M(M-1)\mu-1} \mu^2 \left(I - \tfrac{1}{(M-1)(1+M\mu)}\frac{1}{2^M} \sum_{\hat{\xi} \in \{\pm 1\}^M} \abs{\hat{\xi}^\top \1}^2\right)\,.
	\end{align*}
	Finally, we note that, there are exactly $\binom{M}{k}$ many $\hat{\xi} \in \{\pm 1\}^M$ such that $\operatorname{spin}(\hat{\xi}) = 1$, where $\operatorname{spin}(\hat{\xi}) = \abs{\{i \in \{1,\ldots,M\} \mid \hat{\xi}_i = 1\}}$, and that $\hat{\xi}^\top \1 = 2 \operatorname{spin}(\hat{\xi}) - M$.
	Thus the summation in the last expression becomes
	\begin{align*}
		\frac{1}{2^M} \sum_{\hat{\xi} \in \{\pm 1\}^M} \abs{\hat{\xi}^\top \1}^2
		&= \frac{1}{2^M} \sum_{k=0}^M \binom{M}{k} (2k-M)^2 \\
		&= \b{E}_{\zeta \sim \operatorname{Binom}_M(\tfrac{1}{2})} (2\zeta-M)^2 \\
		&= 4 \b{E}_{\zeta \sim \operatorname{Binom}_M(\tfrac{1}{2})} \zeta^2 - 4 M \b{E}_{\zeta \sim \operatorname{Binom}_M(\tfrac{1}{2})} \zeta + M^2 \\
		&= M\,.
	\end{align*}
	Substituting this into the above infimum and performing some algebraic simplifications yields the claimed equality.
\end{proof}

\section{Proofs of Section \ref{sec:structuredDRORelaxationGap}}

\begin{proof}[Proof of Lemma \ref{lem:nestedRelaxationSets}]
	(i) Since $c$ is proper, so is $c^M$ and thus $\b{B}_{M\rho}^{c^M}(\hat{P}^{\otimes M})$ is weakly compact.
	The set $\c{P}_{\operatorname{sym}}(X^M)$ is weakly closed, for it can be represented as an intersection of weakly closed sets, 
	\begin{align*}
		\c{P}_{\operatorname{sym}}(X^M) 
		= \bigcap_{\pi \in \c{S}_M} \left\{\bar{P} \in \c{P}(X^M) \mid \bar{P} = \pi \# \bar{P} \right\}\,.
	\end{align*}
	Therefore $\b{B}_{M\rho}^{c^M}(\hat{P}^{\otimes M})_{\operatorname{sym}} = \b{B}_{M\rho}^{c^M}(\hat{P}^{\otimes M}) \cap \c{P}_{\operatorname{sym}}(X^M)$ is, as an intersection of a weakly compact and weakly closed set, itself weakly compact.
	Then $\c{U}_M$, being the compact image of the latter set under the weakly continuous map $\bar{P} \mapsto \operatorname{pr}_{1:N}^M \# \bar{P}$ is weakly compact.\\
	(ii) It suffices to show that $\c{W} \subseteq \c{U}_M$ and that $\c{U}_M$ is convex.
	The latter follows from the convexity of $\b{B}_{M\rho}^{c^M}(\hat{P}^{\otimes M})_{\operatorname{sym}}$ and the former from the fact that $\operatorname{pr}_{1:N}^M \# P^{\otimes M} = P^{\otimes N}$ with $P^{\otimes M} \in \b{B}_{M\rho}^{c^M}(\hat{P}^{\otimes M})_{\operatorname{sym}}$ for all $P \in \b{B}_\rho^c(\hat{P})$ by Lemma \ref{lem:wassersteinBallProductInclusion}.\\
	(iii) Let $\tilde{P} \in \c{U}_{M+1}$.
	Then $\tilde{P} = \operatorname{pr}_{1:N}^{M+1} \# \bar{P}'$ for some $\bar{P}' \in \b{B}_{(M+1)\rho}^{c^{M+1}}(\hat{P}^{\otimes (M+1)})_{\operatorname{sym}}$.
	Since $\bar{P}', \hat{P}^{\otimes (M+1)} \in \c{P}_{\operatorname{sym}}(X^{M+1})$, by Lemma \ref{lem:symmetricCoupling} there exists a \\ $\hat{\Lambda} \in \hat{\Gamma}_{\operatorname{sym}}(\hat{P}^{\otimes (M+1)},\bar{P}')$ such that $\intg{}{c^{M+1}}{\hat{\Lambda}} = W_{c^{M+1}}(\hat{P}^{\otimes (M+1)},\bar{P}') \leq (M+1) \rho$. 
	Consider $\Lambda = \operatorname{pr}_{1:M,1:M}^{M+1,M+1} \# \hat{\Lambda} \in \Gamma(\hat{P}^{\otimes M},\bar{P})$ with $\bar{P} = \operatorname{pr}_{1:M}^{M+1}\# \bar{P}'$.
	We claim that $\intg{}{c^M}{\Lambda} \leq M\rho$.
	For this purpose we note that $\intg{}{c \circ \operatorname{pr}_{1,1}}{\hat{\Lambda}} \leq \rho$.
	Indeed, 
	\begin{align*}
		(M+1) \rho 
		\geq \intg{}{c^{M+1}}{\hat{\Lambda}}
		= \sum_{i=1}^{M+1} \intg{}{c \circ \operatorname{pr}_{i,i}}{\hat{\Lambda}}
		= (M+1) \intg{}{c \circ \operatorname{pr}_{1,1}}{\hat{\Lambda}}\,,
	\end{align*}
	where the first equality follows from $c^{M+1} = \sum_{i=1}^{M+1} c \circ \operatorname{pr}_{i,i}$ and the last equality follows from $\hat{\Lambda}$ being a symmetric coupling.
	Then
	\begin{gather*}
		\intg{}{c^M}{\Lambda}
		= \intg{}{c^M \circ \operatorname{pr}_{1:M,1:M}^{M+1,M+1}}{\hat{\Lambda}}
		= \sum_{i=1}^M \intg{}{c \circ \operatorname{pr}_{i,i}}{\hat{\Lambda}} 
		= M \intg{}{c \circ \operatorname{pr}_{1,1}}{\hat{\Lambda}} 
		\leq M \rho\,.
	\end{gather*}
	Thus we have shown that $W_{c^M}(\hat{P}^{\otimes M},\bar{P}) \leq M \rho$, or equivalently $\bar{P} \in \b{B}_{M\rho}^{c^M}(\hat{P}^{\otimes M})_{\operatorname{sym}}$.
	But then $\tilde{P} = \operatorname{pr}_{1:N}^{M+1} \# \bar{P}' = \operatorname{pr}_{1:N}^M \# \operatorname{pr}_{1:M}^{M+1} \# \bar{P}' = \operatorname{pr}_{1:N}^M \# \bar{P}$ since $\operatorname{pr}_{1:N}^{M+1} \circ \operatorname{pr}_{1:M}^M = \operatorname{pr}_{1:N}^{M+1}$ if $M \geq N$.
	Additionally, it is easily verified that $\bar{P}$ is symmetric, just because $\bar{P}'$ is.
	Hence, by the definition of $\c{U}_M$, we have $\tilde{P} \in \c{U}_M$.  
\end{proof}

\begin{proof}[Proof of Theorem \ref{thm:limitSetContinuous}]
	We have already noted that $\hat{\operatorname{U}}_\infty^{\operatorname{sym}}(\ell) \leq \operatorname{U}_\infty^{\operatorname{sym}}(\ell)$.
	Moreover, in the proof of Theorem \ref{thm:droGeneralSolvability} is it shown that if $\ell$ is upper semi-continuous and $\ell \in \c{G}_{c^N}^<(X^N)$, then $\bar{P} \mapsto \b{E}_{\bar{P}} \ell$ is weakly upper semi-continuous on $\b{B}_{N\rho}^{c^N}(\hat{P}^{\otimes N})$.
	For each $M \geq N$ there exists some $\bar{P}_M \in \c{U}_M$ such that
	\begin{align*}
		\b{E}_{x \sim \bar{P}_M} \ell(x) 
		\geq \sup_{\bar{P} \in \c{U}_M} \b{E}_{x \sim \bar{P}} \ell(x) - \frac{1}{M} 
		= \operatorname{U}_M^{\operatorname{sym}}(\ell) - \frac{1}{M}\,.
	\end{align*}
	Since $\c{U}_{M+1} \subseteq \c{U}_M$ for all $M \geq N$ by Lemma \ref{lem:nestedRelaxationSets}, it follows that $\bar{P}_M \in \c{U}_L$ for any $N \leq L \leq M$.
	Moreover, by the same Lemma each $\c{U}_M$ is weakly compact and thus there exists a subsequence $\bar{P}_{M_k}$ and $\bar{P} \in \bigcap_{M \geq N} \c{U}_M = \c{U}_\infty$ such that $\bar{P}_{M_k} \to \bar{P}$ weakly for $k \to \infty$.
	It follows that
	\begin{align*}
		\begin{gathered}
			\hat{\operatorname{U}}_\infty^{\operatorname{sym}}(\ell)
			\geq \b{E}_{x \sim \bar{P}} \ell(x) 
			\geq \limsup_{k \to \infty} \b{E}_{x \sim \bar{P}_{M_k}} \ell(x) 
			\geq \limsup_{k \to \infty} \left(\operatorname{U}_{M_k}^{\operatorname{sym}}(\ell) - \frac{1}{M_k} \right) \\
			\geq \limsup_{k \to \infty} \operatorname{U}_{M_k}^{\operatorname{sym}}(\ell) - \limsup_{k\to \infty} \frac{1}{M_k} 
			= \operatorname{U}_\infty^{\operatorname{sym}}(\ell)\,,
		\end{gathered}
	\end{align*}
	where the second inequality follows from the weak upper semi-continuity of $\bar{P} \mapsto \b{E}_{\bar{P}} \ell$ and the last inequality holds because $\limsup_k (a_k + b_k) \leq \limsup_k a_k + \limsup_k b_k$ for any two real sequences $(a_k)_{k=1}^\infty$ and $(b_k)_{k=1}^\infty$.
\end{proof}

\begin{proof}[Proof of Lemma \ref{lem:relaxationSetAlternativeDescription}]
First we show $\supseteq$.
If $\bar{P} \in \hat{\c{W}}_\infty$, then we claim that \\ $\operatorname{pr}_{1:N} \# \bar{P} \in \c{U}_M$ for all $M \geq N$.
Indeed, it holds that $\operatorname{pr}_{1:N} \# \bar{P} = \operatorname{pr}_{1:N} \# (\operatorname{pr}_{1:M} \# \bar{P})$ and $\operatorname{pr}_{1:M} \# \bar{P} \in \b{B}_{M \rho}^{c^M}(\hat{P}^{\otimes M})$ by the definition of $\hat{\c{W}}_\infty$.
Moreover, $\operatorname{pr}_{1:M} \# \bar{P}$ is symmetric, since $\bar{P}$ is symmetric.
Hence, $\operatorname{pr}_{1:N} \# \bar{P} \in \c{U}_M$ follows by the definition of $\c{U}_M$.
Now we will show that $\subseteq$ holds.
For this purpose let $\tilde{P} \in \c{U}_M$ for all $M \geq N$.
Then for each $M \geq N$ there exists some $\bar{P}_M \in \b{B}_{M \rho}^{c^M}(\hat{P}^{\otimes M})_{\operatorname{sym}}$ such that $\tilde{P} = \operatorname{pr}_{1:N} \bar{P}_M$.
We have to exhibit a single $\bar{P} \in \c{P}_{\operatorname{sym}}(X^\infty)$ such that $\tilde{P} = \operatorname{pr}_{1:N} \# \bar{P}$ and $\operatorname{pr}_{1:M} \# \bar{P} \in \b{B}_{M \rho}^{c^M}(\hat{P}^{\otimes M})$ for all $M \geq 1$.
For that purpose we notice first that if $\bar{P}_M \in \b{B}_{M \rho}^{c^M}(\hat{P}^{\otimes M})_{\operatorname{sym}}$, then for all $1 \leq L \leq M$ it holds that $\operatorname{pr}_{1:L}\#\bar{P}_M \in \b{B}_{L \rho}^{c^L}(\hat{P}^{\otimes L})_{\operatorname{sym}}$.
Indeed, symmetry is clear from the symmetry of $\bar{P}_M$.
To see why $\operatorname{pr}_{1:L}\#\bar{P}_M \in \b{B}_{L \rho}^{c^L}(\hat{P}^{\otimes L})$, let $\hat{\Lambda} \in \Gamma_{\operatorname{sym}}(\hat{P}^{\otimes M},\bar{P}_M)$ be a coupling (see Lemma \ref{lem:symmetricCoupling}) such that $\intg{}{c^M}{\hat{\Lambda}} = W_{c^M}(\hat{P}^{\otimes M},\bar{P}_M) \leq M\rho$.
Just as in the proof of Lemma \ref{lem:nestedRelaxationSets} one can show, by exploiting symmetry of $\hat{\Lambda}$, that the coupling $\Lambda = \operatorname{pr}_{1:L,1:L}^{M,M} \# \hat{\Lambda}$ satisfies $\Lambda \in \Gamma_{\operatorname{sym}}(\hat{P}^{\otimes L},\operatorname{pr}_{1:L}\#\bar{P}_M)$ and that $\intg{}{c^L}{\Lambda} \leq L\rho$, i.e. $\operatorname{pr}_{1:L}\#\bar{P}_M \in \b{B}_{L \rho}^{c^L}(\hat{P}^{\otimes L})$.
Thus we have established that any $\bar{P}_M \in \b{B}_{M \rho}^{c^M}(\hat{P}^{\otimes M})_{\operatorname{sym}}$ can be projected down to $\operatorname{pr}_{1:L}^M\#\bar{P}_M \in \b{B}_{L \rho}^{c^L}(\hat{P}^{\otimes L})_{\operatorname{sym}}$.
Let us denote $\bs{P}^M = \{\bar{P}' \in \b{B}_{M \rho}^{c^M}(\hat{P}^{\otimes M})_{\operatorname{sym}} \mid \operatorname{pr}_{1:N}\# \bar{P}' = \tilde{P}\}$ and define now the maps $\phi_{M,L}:\bs{P}^M \to \bs{P}^L:\bar{P}' \mapsto \operatorname{pr}_{1:L}^M\#\bar{P}'$ for $N \leq L \leq M$.
Then, since $\phi_{M,K} = \phi_{L,K} \circ \phi_{M,L}$ and $\phi_{M,M} = \operatorname{id}_{\bs{P}^M}$ for all $N \leq K \leq L \leq M$, the family $\{(\bs{P}^M)_{M \geq N},(\phi_{M,L})_{M \geq L \geq N}\}$ is an inverse family\footnote{Also sometimes called projective family} \cite{willard2012general}. 
Hence we can form the inverse limit
\begin{align*}
	\varprojlim_{M \geq N} \bs{P}^M
	= \{(\bar{P}_M')_{M \geq N} \in \bs{P} \mid \phi_{M,L}(\bar{P}_M') = \bar{P}_L'\} \subseteq \bs{P}\,,
\end{align*}
where $\bs{P} = \prod_{M \geq N} \bs{P}^M$ is nonempty, since $(\bar{P}_M)_{M \geq N} \in \bs{P}$.
Since each $\bs{P}^M$ is Hausdorff and weakly compact, by \cite[Theorem 29.11]{willard2012general} the set $\varprojlim_{M \geq N} \bs{P}^M$ is nonempty.
Let $(\bar{P}_M')_{M \geq N} \in \varprojlim_{M \geq N} \bs{P}^M$ be arbitrary.
By the very definition of $\varprojlim_{M\geq N} \bs{P}^M$ it follows that the family $(X^M,\bar{P}_M')_{M \geq N}$ is consistent in the sense of Kolmogorov.
By the Kolmogorov extension theorem \cite[15.27 Corollary]{guide2006infinite} there exists a single Borel probability measure $\bar{P} \in \c{P}(X^{\infty})$ such that $\operatorname{pr}_{1:M} \# \bar{P} = \bar{P}_M' \in \b{B}_{M \rho}^{c^M}(\hat{P}^{\otimes M})_{\operatorname{sym}}$ for all $M \geq N$.
Since each $\bar{P}_M'$ is symmetric, so is $\bar{P}$.
Moreover, for $N = M$ it follows from the definition of $\bs{P}^N$ that $\operatorname{pr}_{1:N} \# \bar{P} = \tilde{P}$.
Hence we have exhibited the desired distribution $\bar{P}$, which shows that $\tilde{P} \in \c{U}_\infty$.
\end{proof}

Before we proceed to the proof of Theorem \ref{thm:mainRelaxedSetRepresentation} will prove establish the following result that is of independent interest.

\begin{theorem}
	\label{thm:wasserseinProductIntegral}
	Let $c:X \times X \to [0,\infty)$ be a proper transportation cost, $\nu \in \c{P}(\c{P}(X))$ and $\hat{P} \in \c{P}_c(X)$.
	Then
	\begin{align}
		\label{eq:wasserseinProductIntegral}
		\sup_{M \in \b{N}} \frac{1}{M} W_{c^M}\left(\intg{\c{P}(X)}{P^{\otimes M}}{\nu(P)},\hat{P}^{\otimes M}\right) = \intg{}{W_c(P,\hat{P})}{\nu(P)}\,.
	\end{align} 
\end{theorem}

\begin{proof}
	First we show ``$\leq$''.
	Let $M \in \b{N}$ be fixed. 
	Then, by \cite[Theorem 4.8]{villani2009optimal}, it holds that
	\begin{align*}
		W_{c^M}\left(\intg{\c{P}(X)}{P^{\otimes M}}{\nu(P)},\hat{P}^{\otimes M}\right)
		&\leq \intg{\c{P}(X)}{W_{c^M}\left(P^{\otimes M},\hat{P}^{\otimes M}\right)}{\nu(P)} \\
		&= M \intg{\c{P}(X)}{W_c(P,\hat{P})}{\nu(P)}
	\end{align*}
	where in the equality we have used the fact that $W_{c^M}\left(\tilde{P}^{\otimes M},P^{\otimes M}\right) = M W_c(\tilde{P},P)$, established in Lemma \ref{lem:wassersteinDistanceProduct}.
	Dividing by $M$ and taking $\sup_{M \in \b{N}}$ on both sides yields ``$\leq$''.
	The argument for ``$\geq$'' is a bit more involved.
	Define $\bar{P}_M := \intg{\c{P}(X)}{P^{\otimes M}}{\nu(P)}$ for $M \in \b{N}$.
	It is sufficient to show for each $\rho > 0$ that 
	\begin{align}
		\label{eq:wassersteinLimitImplication}
		\left[\forall M \in \b{N}\;:\; W_{c^M}\left(\bar{P}_M ,\hat{P}^{\otimes M}\right) \leq M \rho \right]\Longrightarrow \intg{\c{P}(X)}{W_c(P,\hat{P})}{\nu(P)} \leq \rho\,.
	\end{align}
	Indeed, if \eqref{eq:wassersteinLimitImplication} is true and $<$ holds in \eqref{eq:wasserseinProductIntegral}, then by picking $\rho$ that is larger than the left hand side in \eqref{eq:wasserseinProductIntegral} and smaller than the right hand side in \eqref{eq:wasserseinProductIntegral} and applying \eqref{eq:wassersteinLimitImplication} we get a contradiction.
	In the following we will thus show \eqref{eq:wassersteinLimitImplication}.
	First fix $\rho > 0$ and assume that the left side of \eqref{eq:wassersteinLimitImplication} holds.
	We will show that the right hand side of \eqref{eq:wassersteinLimitImplication} holds in steps.\\
%	A brief outline of the proof is as follows:
%	First, departing from the left hand side of \eqref{eq:wassersteinLimitImplication}, the goal is to construct a single symmetric coupling $\Lambda$ between $\intg{}{P^{\otimes \infty}}{\nu(P)}$ and $\hat{P}^{\otimes \infty}$, whose first $X^M$-projected marginal satisfies the inclusion being in the linearly scaled ball $\b{B}_{M\rho}^{c^M}(\hat{P}^{\otimes M})$.
%	Then we disintegrate $\Lambda$ w.r.t. its first marginal into a kernel and show that for each $P \in \c{P}(X)$ the $P^{\otimes \infty}$-integrated kernel posses an interesting property: Its first marginal is $P^{\otimes \infty}$ (by definition) and its second averaged marginal $\hat{P}^\infty$.
%	Using the additional fact that this kernel is symmetric and that the extreme points of the set of symmetric distributions, we then show, using  
%	Then we construct for each $P$
%	
	\emph{Step 1. (Construction of a single coupling $\Lambda$)}.
	Define for each $M \in \b{N}$ the following set
	\begin{align*}
		\bs{\Gamma}^M = \left\{\hat{\Lambda} \in \Gamma_{\operatorname{sym}}(\bar{P}_M,\hat{P}^{\otimes M})\; \middle|\; \intg{X^M \times X^M}{c^M}{\hat{\Lambda}} \leq M\rho \right\} \subseteq \c{P}(X^M \times X^M)\,.
	\end{align*}
	Then, by Lemma \ref{lem:symmetricCoupling}, for each $M \in \b{N}$ it holds that $\bs{\Gamma}^M \neq \emptyset$.
	Moreover, the set $\bs{\Gamma}^M$ is weakly compact, since it can be written as
	\begin{align*}
		\bs{\Gamma}^M
		= \Gamma(\bar{P}_M,P^{\otimes M}) \cap \c{P}_{\operatorname{sym}}(X^M \times X^M) \cap \left\{\hat{\Lambda} \in \c{P}(X^M \times X^M) \mid \intg{}{c^M}{\hat{\Lambda}} \leq M\rho \right\}
	\end{align*}
	with $\c{P}_{\operatorname{sym}}(X^M \times X^M)$ being the set of measures $\Lambda \in \c{P}(X^M \times X^M)$ with $(\pi \oplus \pi)\# \Lambda = \Lambda$ for all $\pi \in \c{S}_M$.
	By \cite[Lemma 4.4, Proof of Theorem 4.1]{villani2009optimal} the set $\Gamma(\bar{P}_M,P^{\otimes M})$ is weakly compact and the latter two sets are weakly closed\footnote{The closedness of the latter comes from the fact that if $\hat{\Lambda}_n \to \hat{\Lambda}$ weakly, then $\intg{}{c^M}{\hat{\Lambda}} \leq \liminf_{n \to \infty} \intg{}{c^M}{\hat{\Lambda}_n}$, which can be shown by Fatou's lemma}, from which weak compactness of $\bs{\Gamma}^M$ follows.  
	Define the (nonempty) product set
	\begin{align*}
		\bs{\Gamma} = {\prod}_{M=1}^\infty \bs{\Gamma}^M\,.
	\end{align*}
	As in the proof of Lemma \ref{lem:nestedRelaxationSets}, for each $M,L \in \b{N}$ with $L \leq M$ define the maps 
	\begin{align*}
		\phi_{M,L}:\bs{\Gamma}^M \to \bs{\Gamma}^L:\hat{\Lambda} \mapsto (\operatorname{pr}_{1:L,1:L}^{M,M})\#\hat{\Lambda}\,.
	\end{align*}
	The fact that $\phi_{M,L}$ are well-defined can be verified by the same computation as in Lemma \ref{lem:nestedRelaxationSets} and by exploiting the symmetry of each coupling in $\bs{\Gamma}^M$.
	Moreover, the family $\{(\bs{\Gamma}^M)_{M \in \b{N}}, (\phi_{M,L})_{M \geq L}\}$ is an inverse family, i.e., it holds that $\phi_{M,N} = \phi_{L,N}\circ \phi_{M,L}$ for all $N \leq L \leq M$ as well as $\phi_{M,M} = \operatorname{id}_{\bs{\Gamma}^M}$.
	Hence we can form the inverse limit
	\begin{align*}
		\varprojlim_{M}\bs{\Gamma}^M := \{(\hat{\Lambda}_M)_{M\in \b{N}} \in \bs{\Gamma} \mid \phi_{M,L}(\hat{\Lambda}_M) = \hat{\Lambda}_L \text{\ for all\ }L \leq M\,,\; M,L \in \b{N}\}\,.
	\end{align*}
	Since each $\bs{\Gamma}^M$ is compact and Hausdorff, the set $\varprojlim_{M}\bs{\Gamma}^M$ is nonempty by \cite[Theorem 29.11]{willard2012general} and we can pick some $(\Lambda_M)_{M\in \b{N}} \in \varprojlim_{M}\bs{\Gamma}^M$.
	By the very definition of $\varprojlim_{M}\bs{\Gamma}^M$ it follows that\footnote{Here we identify each $\Lambda_M$ with a probability measure on $(X\times X)^M \cong X^M \times X^M$, where the isomorphism $\cong$ is provided by the map $((x_i^1,x_i^2))_{i=1}^M \mapsto ((x_i^1)_{i=1}^M,(x_i^2)_{i=1}^M)$} the family $((X\times X)^M,\Lambda_M)_{M\in \b{N}}$ is consistent in the sense of Kolmogorov.
	By the Kolmogorov extension theorem \cite[15.27 Corollary]{guide2006infinite} there exists a single Borel probability measure $\Lambda \in \c{P}(X^\infty \times X^\infty)$ such that $\phi_M(\Lambda) := \operatorname{pr}_{1:M,1:M}\#\Lambda = \Lambda_M$ for each $M \in \b{N}$.
	Then $\Lambda$ is symmetric, since each $\Lambda_M$ is. 
	We claim that $\Lambda \in \Gamma(\bar{P},\hat{P}^{\otimes \infty})$ with $\bar{P} = \intg{}{P^{\otimes \infty}}{\nu(P)}$.
	Indeed, for each $M \in \b{N}$ it holds that $\operatorname{pr}_{1:M,\emptyset}\#\Lambda = \bar{P}_M$ and $\operatorname{pr}_{\emptyset,1:M}\#\Lambda = P^{\otimes M}$.
	Then, evaluation of the marginals $\operatorname{pr}_{1:\infty,\emptyset}\#\Lambda$ and $\operatorname{pr}_{\emptyset,1:\infty}\#\Lambda$ on the finite cylinder sets of the Borel $\sigma$-algebra on $X^\infty$ yields the claim.
	In summary we have obtained a symmetric coupling $\Lambda \in \Gamma_{\operatorname{sym}}(\bar{P},\hat{P}^{\otimes \infty})$ on the countable Cartesian product $X^\infty \times X^\infty$ such that
	\begin{align}
		\label{eq:singleCouplingBallCondition}
		\intg{X^\infty \times X^\infty}{c^M \circ \operatorname{pr}_{1:M,1:M}}{\Lambda} \leq M\rho \text{\ \ for all\ \ } M \in \b{N}\,.
	\end{align} 
	\emph{Step 2. (Construction of a family of couplings between $\hat{P}$ and $P$)}.
	Now we will construct couplings between $\hat{P}$ and $P$ for $P \in \supp \nu \subseteq \c{P}(X)$.
	For this purpose let $\Lambda(\cdot \mid \cdot)$ denote the disintegration kernel of $\Lambda$ w.r.t. the first marginal $\bar{P}$, i.e. $\Lambda(\cdot\mid \eta) \in \c{P}(X^\infty)$ for $\bar{P}$-a.a. $\eta \in X^\infty$, the map $X^\infty \to [0,1]:\eta \mapsto \Lambda(A \mid \eta)$ is Borel-measurable for each Borel set $A \subseteq X^\infty$ and
	\begin{align}
		\label{eq:disintegrationKernel}
		\Lambda(A\times B) = \intg{A}{\Lambda(B\mid \eta)}{\bar{P}(\eta)} \text{\ \ for any Borel sets $A,B \subseteq X^\infty$}\,.
	\end{align}
	Such a disintegration kernel exists according to \cite[Theorem 8.36 and 8.37]{klenke2013probability}, since $X^\infty$ is Polish.\\
	\emph{Step 2.1 (Symmetry of disintegration kernel)}.
	For this disintegration kernel the following identity holds:
	\begin{align}
		\label{eq:kernelPermutationInvariant}
		\pi\#\Lambda(\cdot \mid \eta) = \Lambda(\cdot \mid \pi(\eta)) \text{\ \ for all $\pi \in \c{S}_\infty$ and $\bar{P}$-a.a.\ } \eta \in X^\infty\,.
	\end{align}
	Indeed, let $A,B \subseteq X^\infty$ be Borel.
	Then by symmetry of $\Lambda$ and $\bar{P}$ it follows
	\begin{align*}
		&\intg{A}{(\pi \#\Lambda)(B \mid \eta)}{\bar{P}(\eta)} \\
		&\qquad= \intg{A}{\Lambda(\pi^{-1}(B)\mid \eta)}{\bar{P}(\eta)} 
		= \Lambda(A \times \pi^{-1}(B)) 
		= \Lambda(\pi(A) \times B) \\
		&\qquad= \intg{\pi(A)}{\Lambda(B\mid \eta)}{\bar{P}(\eta)}  
		= \intg{X^\infty}{\bs{1}_{A}(\pi^{-1}(\eta))\Lambda(B\mid \pi(\pi^{-1}(\eta)))}{\bar{P}(\eta)} \\
		&\qquad= \intg{A}{\Lambda(B\mid \pi(\eta))}{\bar{P}(\eta)}\,.
	\end{align*} 
	Hence \eqref{eq:kernelPermutationInvariant} is established. \\
	\emph{Step 2.2 ($P^{\infty}$-integral of disintegration kernel)}.
	Now, for each $P \in \operatorname{supp} \nu \subseteq \c{P}(X)$ define a Borel probability measure on $X^\infty \times X^\infty$ by
	\begin{align}
		\label{eq:integralDisintegrationKernelDefinition}
		\Lambda(A \times B\mid P^{\otimes \infty}) := \intg{A}{\Lambda(B\mid \eta)}{P^{\otimes \infty}(\eta)} \text{\ \ for Borel\ \ }A, B\subseteq X^\infty\,.
	\end{align}
	This is well-defined, since each $\bar{P}$-nullset is also a $P^{\otimes \infty}$-nullset whenever $P \in \operatorname{supp} \nu$.
	Now, it then holds by definition of $\Lambda(\cdot \mid P^{\otimes \infty})$ that 
	\begin{align*}
		\Lambda(A \times X^\infty \mid P^{\otimes \infty})
		= P^{\otimes \infty}(A) \text{\ \ for all Borel\ \ } A \subseteq X^\infty\,,
	\end{align*}
	and hence $\Lambda(\cdot \mid P^{\otimes \infty}) \in \Gamma(P^{\otimes \infty},\cdot)$, i.e. the first marginal of $\Lambda(\cdot \mid P^{\otimes \infty})$ is $P^{\otimes \infty}$. \\
	\emph{Step 2.3 (Second marginal of $\Lambda(\cdot \mid P^{\otimes \infty})$)}.
	Now we will show that the second marginal of $\Lambda(\cdot \mid P^{\otimes \infty})$ is actually $\hat{P}^{\otimes \infty}$ for $\nu$-a.a. $P \in \c{P}(X)$, i.e. that $\Lambda(X^\infty \times B\mid P^{\otimes \infty}) = \hat{P}^{\otimes \infty}(B)$ for all $\nu$-a.a. $P \in \c{P}(X)$ and Borel sets $B \subseteq X^\infty$.\\
	\emph{Step 2.3.1 (Averaged second marginal of $\Lambda(\cdot \mid P^{\otimes \infty})$)}
	First we note that
	\begin{align*}
		\intg{}{\Lambda(X^\infty \times B\mid P^{\otimes \infty})}{\nu(P)} 
		&= \intg{}{\intg{X^{\infty}}{\Lambda(B\mid \eta)}{P^{\otimes \infty}(\eta)}}{\nu(P)} \\
		&= \intg{X^{\infty}}{\Lambda(B\mid \eta)}{\bar{P}(\eta)} \\
		&= \Lambda(X^\infty \times B) \\
		&= \hat{P}^{\otimes \infty}(B)
	\end{align*}
	for all Borel sets $B \subseteq X^\infty$, where the second equality is due to Theorem \ref{thm:disintegrationMixture} and the definition of $\bar{P}$ and the third equality due to \eqref{eq:disintegrationKernel}. 
	This shows that the \emph{averaged} second marginal of $\Lambda(\cdot \mid P^{\otimes \infty})$ is $\hat{P}^{\otimes \infty}$.\\
	\emph{Step 2.3.2 (Symmetry of second marginal of $\Lambda(\cdot \mid P^{\otimes \infty})$)}.
	Let us first show that the measure $B \mapsto \Lambda(X^\infty \times B\mid P^{\otimes \infty}) =: \hat{\Lambda}(B\mid P^{\otimes \infty})$ (i.e. the second marginal of $\Lambda(\cdot \mid P^{\otimes \infty})$) on $X^\infty$ is symmetric for $\nu$-a.a. $P \in \c{P}(X)$. 
	To this end let $\pi \in \c{S}_\infty$ and $B \subseteq X^\infty$ Borel.
	Then it holds for $\nu$-a.a. $P$ that
	\begin{align*}
		\hat{\Lambda}(\pi^{-1}(B)\mid P^{\otimes \infty})
		&= \intg{X^\infty}{\Lambda(\pi^{-1}(B)\mid \eta)}{P^{\otimes \infty}(\eta)} \\
		&= \intg{X^\infty}{\Lambda(B\mid \pi(\eta))}{P^{\otimes \infty}(\eta)} \\
		&= \intg{X^\infty}{\Lambda(B\mid \eta)}{P^{\otimes \infty}(\eta)} \\
		&= \hat{\Lambda}(B\mid P^{\otimes \infty})\,,
	\end{align*}
	since $\Lambda(\pi^{-1}(B)\mid \eta) = \Lambda(B\mid \pi(\eta))$ holds for $\bar{P}$-a.a. $\eta \in X^\infty$ by \eqref{eq:kernelPermutationInvariant} and hence also for $P^{\otimes \infty}$-a.a. $\eta \in X^\infty$ for $\nu$-a.a. $P \in \c{P}(X)$.
	In summary we now know that $\hat{\Lambda}(\cdot \mid P^{\otimes \infty})$ is symmetric for $\nu$-a.a. $P$ and that 
	\begin{align}
		\label{eq:symmetricMeasureConvexCombination}
		\intg{}{\hat{\Lambda}(\cdot\mid P^{\otimes \infty})}{\nu(P)} = \hat{P}^{\otimes \infty}\,.
	\end{align}
	By \cite[Theorem 5.3]{hewitt1955symmetric} the extreme points of $\c{P}_{\operatorname{sym}}(X^\infty)$ are precisely the measures of the form $\tilde{P}^{\otimes \infty}$ with $\tilde{P} \in \c{P}(X)$.
	Hence from \eqref{eq:symmetricMeasureConvexCombination} and Lemma \ref{lem:integralExtremePoints} (taking there $Y = \c{P}(X)$, $Z = \c{P}(X^\infty)$, $\s{Z} = \c{P}_{\operatorname{sym}}(X^\infty)$, $F(P) = \hat{\Lambda}(\cdot\mid P^{\otimes \infty})$) we conclude that $\hat{\Lambda}(\cdot \mid P^{\otimes \infty}) = \hat{P}^{\otimes \infty}$ and thus $\Lambda(\cdot \mid P^{\otimes \infty}) \in \Gamma_{\operatorname{sym}}(P^{\otimes \infty},\hat{P}^{\otimes \infty})$ for $\nu$-a.a. $P$, i.e. the second marginal of $\Lambda(\cdot \mid P^{\otimes \infty})$ is $\hat{P}^{\otimes \infty}$.\\
	\emph{Step 2.4 (Construction of a family of couplings using $\Lambda(\cdot \mid P^{\otimes \infty})$)}.
	Finally define the family of couplings
	\begin{align}
		\label{eq:couplingFamily}
		\Lambda(\cdot \mid P) := \operatorname{pr}_{1,1}\# \Lambda(\cdot \mid P^{\otimes \infty}) \in \Gamma(P,\hat{P}) \text{\ \ for $\nu$-a.a. $P$}\,.
	\end{align}
	\emph{Step 3. (Conclusion of right hand side of of \eqref{eq:wassersteinLimitImplication})}.
	By unfolding all definitions we obtain finally
	\begin{align*}
		\intg{\c{P}(X)}{W_c(P,\hat{P})}{\nu(P)}
		&\leq \intg{\c{P}(X)}{\intg{X \times X}{c}{\Lambda(\cdot \mid P)}}{\nu(P)} \\
		&= \intg{\c{P}(X)}{\intg{X^\infty \times X^\infty}{c \circ \operatorname{pr}_{1,1}}{\Lambda(\cdot \mid P^{\otimes \infty})}}{\nu(P)} \\
		&= \intg{\c{P}(X)}{\intg{X^\infty}{\intg{X^\infty}{c \circ \operatorname{pr}_{1,1}(\eta,\xi)}{\Lambda(\xi \mid \eta)}}{P^{\otimes \infty}(\eta)}}{\nu(P)} \\
		&= \intg{X^\infty}{\intg{X^\infty}{c \circ \operatorname{pr}_{1,1}(\eta,\xi)}{\Lambda(\xi \mid \eta)}}{\bar{P}(\eta)} \\
		&= \intg{X^\infty \times X^\infty}{c \circ \operatorname{pr}_{1,1}}{\Lambda} \\
		&= \intg{X^\infty \times X^\infty}{c^1}{\Lambda}  \\
		&\leq 1\cdot \rho = \rho\,,
	\end{align*}
	which concludes the implication in \eqref{eq:wassersteinLimitImplication} and finishes the proof.
	In the above sequence of (in)equalities (i) the first inequality holds due to $\Lambda(\cdot \mid P)$ being a specific coupling as per \eqref{eq:couplingFamily}, (ii) the first equality holds due to \eqref{eq:couplingFamily}, (iii) the second equality due to \eqref{eq:integralDisintegrationKernelDefinition} and Theorem \ref{thm:disintegrationMixture}, (iv) the third equality holds due to the definition of $\bar{P}$ and Theorem \ref{thm:disintegrationMixture}, (v) the fourth equality holds due to \eqref{eq:disintegrationKernel}, (vi) the last inequality due \eqref{eq:singleCouplingBallCondition}.
\end{proof}

\begin{remark}
	During the preparation of this manuscript the authors became \\ aware of the works \cite{zaev2016ergodic,kolesnikov2017optimal}, where a general result relating the Wasserstein distance of invariant probability measures w.r.t. their ergodic decompositions is established.
	Instead of stating it here in full generality, we restrict ourselves to the special case where the invariance is given w.r.t. a group action of an amenable group.
	We first need the following additional definitions. 
	Let $G$ be an amenable group acting continuously on a Polish space $Z$, i.e. a locally compact group $G$ with a so-called F\o lner sequence $(F_n)_{n \in \b{N}}$ of Borel measurable subsets of $G$ such that
	\begin{align*}
		\lim_{n \to \infty} \frac{\lambda_G((g\cdot F_n) \triangle F_n)}{\lambda_G(F_n)} = 0\,,
	\end{align*}
	where $\lambda_G$ is the left Haar measure of $G$ \cite{rudin2017fourier}.
	Then a probability measure $\bar{P} \in \c{P}(Z)$ is said to be $G$-invariant if $g \# \bar{P} = \bar{P}$ for all $g \in G$\footnote{Here we identify an element $g \in G$ with its action mapping $Z \to Z: z \mapsto g \cdot z$} and $G$-ergodic if for any $G$-invariant $A \in \c{B}(Z)$\footnote{i.e. $g \cdot A = A$ for all $g \in G$} it follows that $\bar{P}(A) \in \{0,1\}$.
	The sets of $G$-invariant and $G$-ergodic Borel probability measures on $Z$ are denoted by $\c{P}_G(Z)$ and $\partial \c{P}_G(Z)$, respectively. 
	It is known that $\c{P}_G(Z)$ is a Dynkin simplex, i.e. that for each $\bar{P} \in \c{P}_G(Z)$ there is a unique decomposition measure $\mu \in \c{P}(\partial \c{P}_G(Z))$ with $\b{E}\mu = \intg{}{\tilde{P}}{\mu(\tilde{P})} = \bar{P}$.
	Moreover, if $\c{A}_G^Z \subseteq \c{B}(Z)$ denotes the $\sigma$-algebra of $G$-invariant sets, then there exists a Markov kernel $\kappa_G^Z:Z \times \c{B}(Z) \to [0,1]$ w.r.t. $\c{A}_G^Z$ such that $\partial \c{P}_G(Z) = \{\kappa_G^Z(z,\cdot) \mid z \in Z\}$ and $\mu(S) = \bar{P}(\{z \in Z \mid \kappa_G^Z(z,\cdot) \in S\})$ for any $S \subseteq \c{P}_G(Z)$, where $\mu$ is the decomposition measure of $\bar{P}$, i.e. $\mu = \kappa_G^Z \# \bar{P}$, where this time we see $\kappa_G^Z:Z \to \c{P}(Z)$.
	Equivalently $\bar{P} = \intg{}{\kappa_G^Z(z,\cdot)}{\bar{P}(z)}$.
	In the case of an amenable group $G$ the kernel $\kappa_G^Z$ is given by
	\begin{align*}
		\kappa_G^Z(z,A)
		= \lim_{n \to \infty} \frac{1}{\lambda_G(F_n)}\intg{F_n}{\delta_{g\cdot z}(A)}{\lambda_G(g)}\,, \quad z \in Z\,, A \in \c{B}(Z)\,.
	\end{align*}
	%Finally, if $G$ is acting on $Z$, then $G$ can be made act on $Z \times Z$ by the diagonal action $g \cdot (z_1,z_2) = (g\cdot z_1,g\cdot z_2)$ for $g \in G$ and $z_1,z_2 \in Z$.
	%In this case the notation $\c{P}_G(Z \times Z)$, $\kappa_G^{Z \times Z}$, etc. will refer to this group action.
	Moreover, let us denote for any lower semi-continuous cost $c_Z:Z \times Z \to [0,\infty]$ the $G$-constrained Wasserstein distance by
	\begin{align*}
		W_{c_Z}^G(\bar{P}_1,\bar{P}_2)
		= \inf_{\Lambda \in \Gamma_G(\bar{P}_1,\bar{P}_2)} \intg{}{c_Z}{\Lambda}\,, \quad
		\bar{P}_1, \bar{P}_2 \in \c{P}_G(Z)\,,
	\end{align*}
	where $\Gamma_G(\bar{P}_1,\bar{P}_2) = \{\Lambda \in \Gamma(\bar{P}_1,\bar{P}_2) \mid \forall g \in G\,:\; (g \oplus g)\#\Lambda = \Lambda \}$ is the set of $G$-invariant couplings.
	Now we can state the result of \cite{zaev2016ergodic}.
	\begin{theorem}[Theorem 4.8 in \cite{zaev2016ergodic}]
		\label{thm:zaevErgodic}
		Let $G$ and $Z$ be as above. 
		Then for any lower semi-continuous $c_Z:Z \times Z \to [0,\infty]$ and $\bar{P}_1, \bar{P}_2 \in \c{P}_G(Z)$ it holds that
		\begin{align*}
			W_{c_Z}^G(\bar{P}_1,\bar{P}_2)
			= W_{C_Z}(\kappa_G^Z \#\bar{P}_1,\kappa_G^Z \# \bar{P}_1)
			= W_{C_Z}(\mu_1,\mu_2)\,,
		\end{align*}
		where $\mu_1, \mu_2 \in \c{P}(\partial \c{P}_G(Z))$ are the decomposition measures of $\bar{P}_1$ and $\bar{P}_2$, respectively and 
		where
		\begin{align*}
			C_Z(\tilde{P}_1,\tilde{P}_2) = W_{c_Z}^G(\tilde{P}_1,\tilde{P}_2) \text{\ \ for\ \ } \tilde{P}_1\,, \tilde{P}_2 \in \partial \c{P}_G(Z)\,.
		\end{align*}
	\end{theorem}
	In other words: The $G$-constrained Wasserstein distance of two $G$-invariant measures on $Z$ w.r.t. a transportation cost $c$ is the Wasserstein distance between their decomposition measures on $\c{P}(Z)$ w.r.t. the transportation cost $W_c^G$. 
	To see how Theorem \ref{thm:wasserseinProductIntegral} relates to Theorem \ref{thm:zaevErgodic}, let $Z = X^\infty$ and let $G = \c{S}_{\infty}$ act on $Z$ by permutation of the coordinates.
	Note that $G$ is amenable with F\o lner sequence $F_n = \c{S}_n$ being the group of permutations acting on the first $n$ coordinates of $Z$. 
	By Theorem \ref{thm:deFinetti} we have $\c{P}_G(Z) = \c{P}_{\operatorname{sym}}(X^\infty)$ and $\partial \c{P}_{G}(Z) = \{P^{\otimes \infty} \mid P \in \c{P}(X)\}$.
	Given any lower semi-continuous transportation cost $c:X \times X \to [0,\infty)$, define $c_Z:Z \times Z \to [0,\infty)$ by
	\begin{align*}
		c_Z(x,y)
		= c(x_1,y_1) \text{\ for\ } x,y \in Z = X^\infty\,. 
	\end{align*}
	Then taking $\bar{P}_1 = \intg{}{P^{\otimes \infty}}{\nu(P)}$ and $\bar{P}_2 = \hat{P}^{\otimes \infty}$ yields $\mu_1 = T \#\nu$ with $T(P) = P^{\otimes \infty}$ and $\mu_2 = \delta_{\hat{P}^{\otimes \infty}}$ and by Theorem \ref{thm:zaevErgodic}\footnote{Here the second equality is due to the fact that for any coupling $\Lambda \in \Gamma(P,\hat{P})$, the product $\Lambda^{\otimes \infty} \in \Gamma_{\c{S}_\infty}(P^{\otimes \infty},\hat{P}^{\otimes \infty})$ is a permutation-invariant coupling.}
	\begin{align*}
		%\sup_{M \in \b{N}} \frac{1}{M} W_{c^M}(\bar{P}_1,\bar{P}_2)
		W_{c_Z}^{\c{S}_\infty}(\bar{P}_1,\bar{P}_2)
		= \intg{}{W_{c_Z}^{\c{S}_\infty}(P^{\otimes \infty},\hat{P}^{\otimes \infty})}{\nu(P)}
		= \intg{}{W_c(P,\hat{P})}{\nu(P)}\,.
	\end{align*}
	Moreover, if we consider $\c{S}_M$ for $M \in \b{N}$ as a subgroup of $\c{S}_\infty$, then clearly
	\begin{align*}
		\frac{1}{M} W_{c^M} \left(\intg{}{P^{\otimes M}}{\nu(P)},\hat{P}^{\otimes M}\right)	
	    = W_{c_Z}^{\c{S}_M}(\bar{P}_1,\bar{P}_2)
		\leq W_{c_Z}^{\c{S}_\infty}(\bar{P}_1,\bar{P}_2)\,,
	\end{align*}
	where the first equality follows by a straightforward calculation.
	Thus, in the framework of \cite{zaev2016ergodic}, Theorem \ref{thm:wasserseinProductIntegral} establishes the continuity property
	\begin{align*}
		\sup_{M \in \b{N}} W_{c_Z}^{\c{S}_M}(\bar{P}_1,\bar{P}_2)
		= W_{c_Z}^{\c{S}_\infty}(\bar{P}_1,\bar{P}_2)\,,
	\end{align*}
	relating the $\c{S}_M$-constrained Wasserstein distance for a finite $M \in \b{N}$ to the infinite case $M = \infty$.
	
%	The boundary of a subset $\s{M} \subseteq \c{P}(X)$, denoted by $\partial \s{M}$, is the set of all points $\hat{P} \in \s{M}$ such that $\mu \in \c{P}(\s{M})$ with $\b{E}\mu := \intg{\s{M}}{P}{\mu(P)} = \hat{P}$ implies $\mu(\{\hat{P}\}) = 1$.
%	Then $\s{M}$ is called a (Dynkin) simplex if $\partial \s{M} \subseteq \c{P}(X)$ is Borel measurable and for each $\hat{P} \in \s{M}$ there is a unique $\mu \in \c{P}(\s{M})$ with $\b{E}\mu = \hat{P}$ and $\mu(\partial \s{M}) = 1$.
%	Let $\c{A}\subseteq \c{B}(X)$ be a $\sigma$-algebra.
%	Then a Markov kernel $\kappa:X \times \c{B}(X) \to [0,1]$ is said to be a decomposition for the triple $(\c{B}(X),\c{A},\s{M})$ if $\kappa$ is the regular conditional probability of $\hat{P}$ w.r.t. $\c{A}$ for any $\hat{P} \in \s{M}$.
%	If such a kernel exists, then $\c{A}$ is called sufficient for $\s{M}$ if additionally $\hat{P}(\{x \in X \mid \kappa(x,\cdot) \in \s{M}\}) = 1$ for all $\hat{P} \in \s{M}$.
\end{remark}

Now we turn to the proof of Theorem \ref{thm:mainRelaxedSetRepresentation}
	
\begin{proof}[Proof of Theorem \ref{thm:mainRelaxedSetRepresentation}]
The claim about $\c{W}_\infty$ is just a restatement of Lemma \ref{lem:weakClosureIntegralRepresentation}.
Hence we come to the proof of the representation of $\hat{\c{W}}_\infty$.
By de Finetti's theorem (Theorem \ref{thm:deFinetti}) and the fact that
\begin{align*}
	\operatorname{pr}_{1:M}\#\intg{}{P^{\otimes \infty}}{\nu(P)} = \intg{}{\operatorname{pr}_{1:M}\# P^{\otimes \infty}}{\nu(P)} = \intg{}{P^{\otimes M}}{\nu(P)}\,,
\end{align*}
it follows that 
\begin{align*}
	\hat{\c{W}}_\infty
	= \left\{ \intg{}{P^{\otimes \infty}}{\nu(P)} \;\middle|\; \nu \in \hat{\c{V}}\right\}
\end{align*}
with
\begin{align*}
	\hat{\c{V}}
	:= \left\{\nu \in \c{P}(\c{P}(X)) \;\middle|\; \forall M \in \b{N}\,:\; \intg{}{P^{\otimes M}}{\nu(P)} \in \b{B}_{M\rho}^{c^M}(\hat{P}^{\otimes M}) \right\}\,.
\end{align*}
We need to show that
\begin{align}
	\label{eq:alternativeRepresentationMixture}
	\hat{\c{V}}
	= \left\{\nu \in \c{P}(\c{P}(X)) \;\middle|\; \b{E}_{P \sim \nu} W_c(\hat{P},P) = \intg{}{W_c(\hat{P},P)}{\nu(P)} \leq \rho \right\}\,.
\end{align}
To see this, denote the set on the right of \eqref{eq:alternativeRepresentationMixture} by $\bar{\c{V}}$.
Let $\nu \in \hat{\c{V}}$.
By definition this is equivalent to
\begin{align*}
	\frac{1}{M} W_{c^M}\left(\intg{}{P^{\otimes M}}{\nu(P)},\hat{P}^{\otimes M}\right) \leq \rho \text{\ \ for all\ } M \in \b{N}\,. 
\end{align*}
By Theorem \ref{thm:wasserseinProductIntegral} the latter is equivalent to
\begin{align*}
	\intg{}{W_c(\hat{P},P)}{\nu(P)} \leq \rho\,,
\end{align*}
which is the same as $\nu \in \bar{\c{V}}$.
\end{proof}

\section{Proofs of Section \ref{sec:structuredDROTightnessOfRelaxationSequence}}

\begin{proof}[Proof of Theorem \ref{thm:concaveLinearStructureExact}]
	First we prove the claims on concavity for $F_\ell$.
	Since $F_\ell(P) = F_{\ell_{\operatorname{sym}}}(P)$, we can assume that $\ell$ is symmetric already, i.e. $\ell = \ell_{\operatorname{sym}}$.
	Now, suppose that $\ell$ is c.n.d. in two variables and let $P_0, P_1 \in \c{P}_c(X)$ and $\alpha \in (0,1)$.
	Let $P_\alpha = (1-\alpha) P_0 + \alpha P_1 = Q_0 + \alpha Q_1$, where $Q_0 = P_0$ and $Q_1 := P_1 - P_0$ is now interpreted as a signed measure with zero total mass.
	We have
	\begin{align*}
		F_\ell(P_\alpha) 
		= \sum_{\epsilon \in \{0,1\}^N} \alpha^{\sum_{i=1}^N \epsilon_i} \intg{}{\ell}{Q_\epsilon}
		%= \sum_{k=0}^N \alpha^k \left(\sum_{\substack{\epsilon \in \{0,1\}^N \\ \operatorname{spin}(\epsilon) = k}} \intg{}{\ell}{Q_\epsilon} \right)
		= \sum_{k=0}^N \alpha^k \binom{N}{k} \intg{}{\ell}{(Q_1^{\otimes k} \otimes Q_0^{\otimes (N-k)})}\,,
	\end{align*}
	with $Q_\epsilon = Q_{\epsilon_1} \otimes \cdots \otimes Q_{\epsilon_N}$.
	Thus, with the identity $k(k-1)\binom{N}{k} =N(N-1) \binom{N-2}{k-2}$  we obtain
	\begin{align*}
		\frac{\operatorname{d}^2}{\operatorname{d}\alpha^2} F_\ell(P_\alpha)
		&= \sum_{k=2}^N k(k-1)\alpha^{k-2} \binom{N}{k} \intg{}{\ell}{(Q_1^{\otimes k} \otimes Q_0^{\otimes (N-k)})} \\
		&= N(N-1)\sum_{k=0}^{N-2} \alpha^k \binom{N-2}{k} \intg{}{\ell}{(Q_1^{\otimes k+2} \otimes Q_0^{\otimes (N-2-k)})} \\
		&= N(N-1) \intg{}{r(x_1,x_2)}{Q_1^{\otimes 2}(x_1,x_2)}
	\end{align*}
	with 
	\begin{align*}
		r(x_1,x_2)
		&= \sum_{k=0}^{N-2} \alpha^k \binom{N-2}{k} \intg{}{\ell(x_1,x_2,\bar{x})}{(Q_1^{\otimes k} \otimes Q_0^{\otimes (N-2-k)})(\bar{x})} \\
		&= \intg{}{\ell(x_1,x_2,\bar{x})}{P_\alpha^{\otimes (N-2)}(\bar{x})}\,.
	\end{align*}
	Now, since $\ell$ is c.n.d. in the first two variables, it holds for any $n \in \b{N}$, $x \in X^n$ and $\gamma \in \b{R}^n$ with $\1^\top \gamma = 0$ that
	\begin{align*}
		\sum_{i,j=1}^n \gamma_i \gamma_j r(x_i,x_j)
		= \intg{}{\sum_{i,j=1}^n \gamma_i \gamma_j \ell(x_i,x_j,\bar{x})}{P_\alpha^{\otimes (N-2)}(\bar{x})}
		\leq 0\,,
	\end{align*}
	i.e. $r$ is also c.n.d. 
	Now, since $\abs{\ell} \in \c{G}_{c^N}(X^N)$ and $P_\alpha \in \c{P}_c(X)$, it holds that $\abs{r} \in \c{G}_{c^2}(X^2)$, i.e. there exists some $C > 0$ and $\hat{x} = (x_0,x_0) \in X^2$ with $\abs{r(x_1,x_2)} \leq C(1+c(x_1,x_0) + c(x_2,x_0))$ for all $x_1,x_2 \in X$.
	Set $q:X \times X \to \b{R}:(x_1,x_2) \mapsto r(x_1,x_0) + r(x_2,x_0) - r(x_1,x_2) - r(x_0,x_0)$.
	Then by \cite[Lemma 2.1]{berg1984harmonic} $q$ is positive definite \cite[Definition 1.5]{berg1984harmonic} and hence has feature map representation in a Hilbert space \cite[Theorem 4.21]{christmann2008support}, i.e. there exists a Hilbert space $\s{H}$ and a (feature) map $\Phi:X \to \s{H}$ with $q(x_1,x_2) = \<\Phi(x_1),\Phi(x_2)\>_{\s{H}}$ for all $x_1,x_2 \in X$.
	Moreover, we have for any $x \in X$ that
	\begin{align*}
		\norm{\Phi(x)}_{\s{H}}
		\leq \sqrt{q(x,x)} 
		= \sqrt{2r(x,x_0) - r(x,x) - r(x_0,x_0)}
		\leq 2\sqrt{C(1+c(x,x_0))}\,.
	\end{align*} 
	As $X$ is Polish, it follows that $\s{H}$ can be picked to be separable \cite[Lemma 4.33]{christmann2008support}.
	Additionally, since $q$ is Borel measurable, the feature map $\Phi$ is measurable \cite[Lemma 4.25]{christmann2008support} and hence in the Bochner class $L^2(X,Q_1;\s{H})$ \cite[Theorem 11.44]{guide2006infinite} due to
	\begin{align*}
		\intg{}{\norm{\Phi(x)}_{\s{H}}^2}{P_\alpha(x)} \leq 4 \intg{}{C(1+c(x,x_0))}{P_\alpha(x)} < \infty \text{\ for\ } \alpha \in [0,1]\,.
	\end{align*}
	Therefore $\intg{}{\Phi(x)}{Q_1(x)} \in \s{H}$ is well-defined and
	\begin{align*}
		&\frac{1}{N(N-1)} \frac{\operatorname{d}^2}{\operatorname{d}\alpha^2} F_\ell(P_\alpha)  \\
		&\quad=\intg{}{r(x_1,x_2)}{Q_1^{\otimes 2}(x_1,x_2)}\\
		&\quad= \intg{}{\left(r(x_1,x_0) + r(x_2,x_0) - q(x_1,x_2) - r(x_0,x_0)\right)}{Q_1^{\otimes 2}(x_1,x_2)} \\
		&\quad= 2 Q_1(X) \intg{}{r(x,x_0)}{Q_1(x)} - \norm*{\intg{}{\Phi(x)}{Q_1(x)}}_{\s{H}}^2 - Q_1(X)^2 r(x_0,x_0) \\
		&\quad\leq 0\,,
	\end{align*}
	since $Q_1(X) = 0$.	
	This establishes convexity of $F_\ell$.
	To establish the second claim let $\nu \in \c{P}(\c{P}(X))$ be such that $\b{E}_\nu W_c(\hat{P},\cdot) \leq \rho$.
	By Lemma \ref{lem:existenceFinitelySupportedMeasure} applied to $F:P \mapsto \intg{}{\ell}{P^{\otimes N}}$, $G:P \mapsto W_c(\hat{P},P) - \rho$ and the probability measure $\nu$ yields a finitely supported measure $\hat{\nu} \in \c{P}(\c{P}(X))$ such that
	\begin{align*}
		\b{E}_{P \sim \hat{\nu}} W_c(\hat{P},P) \leq \rho \text{\ \ and\ \ } \intg{}{F(P)}{\hat{\nu}(P)} = \intg{}{F(P)}{\nu(P)}\,.
	\end{align*}
	Now assume that $\hat{\nu} = \sum_{i=1}^l \alpha_i \delta_{P_i}$ with $\alpha_i \geq 0$, $\sum_{i=1}^l \alpha_i = 1$.
	Since $F_\ell$ is concave, it follows that for $\tilde{P} = \sum_{i=1}^l \alpha_i P_i$ it holds that
	\begin{align*}
		\intg{}{F}{\delta_{\tilde{P}}} 
		= F(\tilde{P}) 
		\geq \intg{}{F(P)}{\hat{\nu}(P)}
	\end{align*}
	and due to \cite[Theorem 4.8]{villani2009optimal} also $W_c(\hat{P},\tilde{P}) \leq \sum_{i=1}^l \alpha_i W_c(\hat{P},P_i) \leq \rho$, i.e. $\nu \in \c{P}(\b{B}_\rho^c(\hat{P}))$.
	Hence for any $\nu$ in the supremum on the right of \eqref{eq:relaxationComparisonIntegral} we can obtain a feasible point $\delta_{\tilde{P}}$ for the supremum on the left \eqref{eq:relaxationComparisonIntegral} with at least the same objective value, which implies $\operatorname{S}(\ell)
	= \hat{\operatorname{U}}_\infty^{\operatorname{sym}}(\ell)$.
	The final statement is then a combination and Theorem \ref{thm:limitSetContinuous}.
\end{proof}

\begin{proof}[Proof of Theorem \ref{thm:relaxationExactConcave}]
As in the previous proof, let $\nu \in \c{P}(\c{P}(X))$ be such that $\b{E}_\nu W_c(\hat{P},\cdot) \leq \rho$.
Applying Lemma \ref{lem:existenceFinitelySupportedMeasure} to the functions $F:P \mapsto \intg{}{\ell}{P^{\otimes N}}$, $G:P \mapsto W_c(\hat{P},P) - \rho$ and the probability measure $\nu$ yields a finitely supported measure $\hat{\nu} \in \c{P}(\c{P}(X))$ such that
\begin{align*}
	\b{E}_{P \sim \hat{\nu}} W_c(\hat{P},P) \leq \rho \text{\ \ and\ \ } \intg{}{F(P)}{\hat{\nu}(P)} = \intg{}{F(P)}{\nu(P)}\,.
\end{align*}
Now assume that $\hat{\nu} = \sum_{i=1}^l \alpha_i \delta_{P_i}$ with $\alpha_i \geq 0$, $\sum_{i=1}^l \alpha_i = 1$ and let $\Lambda_i \in \Gamma(\hat{P},P_i)$ be an optimal transport plan.
Then by the gluing lemma \cite[Gluing lemma]{villani2009optimal} there exists some multi-marginal $\Lambda \in \Gamma(\hat{P},P_1,\ldots,P_l)$ such that\footnote{for notational convenience we see $\hat{P}$ as the $0$-th marginal of couplings in $\Gamma(\hat{P},P_1,\ldots,P_l)$} $\operatorname{pr}_{0,i}\# \Lambda = \Lambda_i$ for $i=1,\ldots,l$.
Now let $T = \sum_{i=1}^l \alpha_i \operatorname{pr}_i^{l+1}:X^{l+1} \to X$ and $\tilde{P} := T \#\Lambda \in \c{P}(X)$.
Note that we have $(T \#\Lambda)^{\otimes N} = \hat{T} \# \Lambda^{\otimes N}$, where $\hat{T} = (T \oplus \cdots \oplus T):(X^{l+1})^N \to X^N: (\bar{x}_1,\ldots,\bar{x}_N) \mapsto (T\bar{x}_1,\ldots,T\bar{x}_N)$ and that $\hat{T} = \sum_{i=1}^l \alpha_i (\operatorname{pr}_i^{l+1} \oplus \cdots \oplus \operatorname{pr}_i^{l+1})$.
Then, by the concavity of $\ell$, 
\begin{align*}
	F(\tilde{P}) 
	&= \intg{}{\ell}{(T\# \Lambda)^{\otimes N}}
	= \intg{}{\ell}{\hat{T}\#\Lambda^{\otimes N}}
	= \intg{}{\ell\circ \hat{T}}{\Lambda^{\otimes N}} \\
	&\geq \intg{}{\sum_{i=1}^l \alpha_i \ell \circ (\operatorname{pr}_i \oplus \cdots \oplus \operatorname{pr}_i)}{\Lambda^{\otimes N}} 
	= \sum_{i=1}^l \alpha_i \intg{}{\ell}{(\operatorname{pr}_i \#\Lambda)^{\otimes N}}\\
	&= \sum_{i=1}^l \alpha_i F(P_i) 
	= \intg{}{F(P)}{\hat{\nu}(P)}
	= \intg{}{F(P)}{\nu(P)}\,.
\end{align*}
Moreover, since $(\operatorname{pr}_0 \oplus T) \# \Lambda \in \Gamma(\hat{P},\tilde{P})$ and $c(x,\cdot)$ is convex for any $x \in X$, it follows that
\begin{align*}
	\begin{gathered}
	W_c(\hat{P},\tilde{P})
	\leq \intg{}{c(x,y)}{(\operatorname{pr}_0 \oplus T) \# \Lambda(x,y)} 
	= \intg{}{c\Big(x,\sum_{i=1}^l \alpha_i y_i\Big)}{\Lambda(x,y_1,\ldots,y_l)} \\
	\leq \sum_{i=1}^l \alpha_i \intg{}{c(x,y_i)}{\Lambda(x,y_1,\ldots,y_l)} 
	= \sum_{i=1}^l \alpha_i \intg{}{c(x,y_i)}{\Lambda_i(x,y_i)} \\
	= \sum_{i=1}^l \alpha_i W_c(\hat{P},P_i) 
	= \b{E}_{\hat{\nu}} W_c(\hat{P},\cdot) 
	\leq \rho\,.
	\end{gathered}
\end{align*}
Thus we have shown once again that for each feasible point $\nu$ in the supremum on the right of \eqref{eq:relaxationComparisonIntegral} we can obtain a feasible point $\delta_{\tilde{P}}$ for the supremum on the left \eqref{eq:relaxationComparisonIntegral} with at least the same objective value. 
This implies that both suprema in \eqref{eq:relaxationComparisonIntegral} are equal and that $\operatorname{S}(\ell)
= \hat{\operatorname{U}}_\infty^{\operatorname{sym}}(\ell)$. 
The second statement is then again a combination of the first statement and Theorem \ref{thm:limitSetContinuous}.
\end{proof}

\section{Proofs of Section \ref{sec:structured_sets_comparison}}

\begin{proof}[Proof of Theorem \ref{thm:multitransportSymmetricInequality}]
Let $\mu^* \in [0,\infty)$ be a point where the right hand side of \eqref{eq:droMultitransportDuality} is attained.
Then, by the very definition of $\hat{\phi}_{\mu^*}$ we have
\begin{align*}
	\ell(x) \leq \sum_{i=1}^N \mu_i^* c(x_i,z_i) + \hat{\phi}_{\mu^*}(z) \text{\ \ for all\ } x,z \in X^N\,.
\end{align*}
Let $\bar{P} \in \b{B}_{N\rho}^{c^N}(\hat{P}^{\otimes N})_{\operatorname{sym}}$.
Then by Lemma \ref{lem:symmetricCoupling} there exists some $\Lambda \in \Gamma_{\operatorname{sym}}(\hat{P}^{\otimes N},\bar{P})$ with $\intg{}{c^N}{\Lambda} \leq N\rho$.
As in the proof of Lemma \ref{lem:nestedRelaxationSets} one can show using the symmetry of $\Lambda$ that $\intg{}{c\circ (\operatorname{pr}_{i,i}^N)}{\Lambda} \leq \rho$ for all $i=1,\ldots,N$.
Then, integrating the above inequality w.r.t. $\Lambda$ and using the fact that $\check{P} = \hat{P}^{\otimes N}$, we obtain 
\begin{align*}
	\b{E}_{x \sim \bar{P}} \ell(x)
	= \intg{}{\ell(x)}{\Lambda(z,x)} 
	&\leq \sum_{i=1}^N \mu_i^* \intg{}{c(x_i,z_i)}{\Lambda(x,z)} + \intg{}{\hat{\phi}_{\mu^*}(z)}{\Lambda(x,z)} \\
	&= \sum_{i=1}^N \mu_i^* \intg{}{c\circ (\operatorname{pr}_{i,i}^N)}{\Lambda} + \intg{}{\hat{\phi}_{\mu^*}(z)}{\hat{P}^{\otimes N}(z)} \\
	&\leq \sum_{i=1}^N \mu_i^* \rho + \intg{}{\hat{\phi}_{\mu^*}(z)}{\check{P}(z)} \\
	&= \sup_{\bar{P} \in \c{W}_{\operatorname{multi}}} \b{E}_{x \sim \bar{P}} \ell(x)\,.
\end{align*}
Since $\bar{P} \in \b{B}_{N\rho}^{c^N}(\hat{P}^{\otimes N})_{\operatorname{sym}}$ was arbitrary, the proof is complete.
\end{proof}

\section{Proofs of Section \ref{sec:structuredDROOuter}}

In the following we denote for any convex $f:\b{R}^p \to \bar{\b{R}}$ its domain by $\operatorname{dom}(f) = \{x \in \b{R}^p \mid f(x) < \infty\}$.

\begin{proof}[Proof of Theorem \ref{thm:structuredDROOuterTightness}]
%	To prove \eqref{eq:structuredDROOuterValueConvergence}, we note that $\leq$ always holds, since $\Psi_{\operatorname{S}}(\theta) \leq \Psi_{\operatorname{U}}^M(\theta)$ for all $\theta \in \Theta$ and $M \geq N$.
%	Let now suppose that (i) holds and let $\epsilon > 0$ be given and pick $\bar{\theta} \in \Theta$ such that $\Psi_{\operatorname{S}}(\bar{\theta}) < \inf_\Theta \Psi_{\operatorname{S}} + \epsilon$. 
%	Then by Theorem \ref{thm:relaxationExactConcave} that $\lim_{M \to \infty} \Psi_{\operatorname{U}}^M(\bar{\theta}) = \Psi_{\operatorname{S}}(\bar{\theta})$ and hence there exists some $\bar{M} \geq N$ such that $\Psi_{\operatorname{U}}^M(\bar{\theta}) < \inf_\Theta \Psi_{\operatorname{S}} + \epsilon$ for all $M \geq \bar{M}$.
%	This implies 
%	\begin{align*}
%		\inf_\Theta \Psi_{\operatorname{U}}^M \leq \Psi_{\operatorname{U}}^M(\bar{\theta}) < \inf_\Theta \Psi_{\operatorname{S}} + \epsilon \text{\ \ for all\ } M \geq \bar{M}\,.
%	\end{align*}
%	As $\epsilon > 0$ was arbitrary, this shows $\geq$ in \eqref{eq:structuredDROOuterValueConvergence} and thus finishes the proof of the first part.
	To prove \eqref{eq:structuredDROOuterValueConvergence}, we note that
	\begin{align*}
		\lim_{M \to \infty} \inf_{\theta \in \Theta} \Psi_{\operatorname{U}}^M(\theta)
		= \inf_{M \geq N} \inf_{\theta \in \Theta} \Psi_{\operatorname{U}}^M(\theta)
		= \inf_{\theta \in \Theta} \inf_{M \geq N} \Psi_{\operatorname{U}}^M(\theta)
		= \inf_{\theta \in \Theta} \Psi_{\operatorname{S}}(\theta)\,,
	\end{align*}
	where the first equality is due to Corollary \ref{cor:decreasingRelaxationSequence} and the last due to the first set of assumptions on $\ell$ and Theorem \ref{thm:relaxationExactConcave}.
	Now suppose that both sets of assumptions on $\ell$ and $\Theta$ hold.
	First we show that $\Psi_{\operatorname{U}}^M$ for each $M \geq N$ is finite and convex on $\Theta$.
	Note that $\Psi_{\operatorname{U}}^M(\theta) = \sup_{\bar{P} \in \c{U}_M} \b{E}_{x \sim \bar{P}} \ell(\theta,x)$ and so finiteness follows from the condition $\ell(\theta,\cdot) \in \c{G}_{c^N}^<(X^N)$ for each $\theta \in \Theta$ and Theorem \ref{thm:droGeneralSolvability}.
	Convexity then follows from the convexity of $\ell$ in the first argument and the fact that $\Psi_{\operatorname{U}}^M$ is the supremum of the convex functions $\theta \mapsto \b{E}_{x \sim \bar{P}} \ell(\theta,x)$.
	Similarly, $\Psi_{\operatorname{S}}$ is finite and convex on $\Theta$.
	Furthermore, $\Psi_{\operatorname{S}}$ is lower semi-continuous on $\Theta$.
	To see this we first note that, since pointwise suprema preserve lower semi-continuity, it is sufficient to show that for each $\bar{P} \in \c{W}$ the function $\Theta \to \b{R}:\theta \mapsto \b{E}_{x \sim \bar{P}} \ell(\theta,x)$ is lower semi-continuous.
	Now, if $(\theta_n)_{n=1}^\infty \subseteq \Theta$ is any sequence with $\theta_n \to \theta \in \Theta$, then
	\begin{align*}
		\liminf_{n \to \infty} \b{E}_{x \sim \bar{P}} \ell(\theta_n,x)
		\geq \b{E}_{x \sim \bar{P}} \liminf_{n \to \infty} \ell(\theta_n,x)
		\geq \b{E}_{x \sim \bar{P}} \ell(\theta,x)\,,
	\end{align*}
	showing that $\theta \mapsto \b{E}_{x \sim \bar{P}} \ell(\theta,x)$ is lower semi-continuous.
	Here the first inequality holds by Fatou's lemma and the fact that $\ell(\theta_n,\cdot)$ is $\bar{P}$-integrable uniformly in $\bar{P} \in \c{W}$, while the second inequality holds due to the lower semi-continuity of $\ell(\cdot,x)$.
	Hence $\Psi_{\operatorname{S}}$ is lower semi-continuous and in particular the infimum on the right hand side in \eqref{eq:structuredDROOuterValueConvergence} is a minimum.
	Now, extending $\Psi_{\operatorname{S}}$ and $\Psi_{\operatorname{U}}^M$ to be $\infty$ outside of $\Theta$, we consider them as convex functions on $\b{R}^p$ with $\Theta = \operatorname{dom} \Psi_{\operatorname{S}} = \operatorname{dom} \Psi_{\operatorname{U}}^M$ for each $M \geq N$.
	Moreover, by condition (i) and Theorem \ref{thm:relaxationExactConcave} it holds that $\Psi_{\operatorname{U}}^M(\theta) \to \Psi_{\operatorname{S}}(\theta)$ pointwise on $\Theta$.
	By \cite[Corollary 7.18]{rockafellar2009variational} it follows that $\Psi_{\operatorname{U}}^M \to \Psi_{\operatorname{S}}$ uniformly on each compact subset of $\operatorname{int} \Theta$.
	Thus by \cite[Theorem 7.17]{rockafellar2009variational} the functions $\Psi_{\operatorname{U}}^M$ converge to $\Psi_{\operatorname{S}}$ in the epigraphical sense \cite[Definition 7.1]{rockafellar2009variational} for $M \to \infty$.
	Then applying \cite[Theorem 7.31 (b)]{rockafellar2009variational} yields the claim.
\end{proof}

\section{Proofs of Section \ref{sec:structuredDROPolyhedralLoss}}

Before we proceed to the proof of Theorem \ref{thm:polyhedralConcaveLossProgram}, we need to establish the following two results.
The first one concerns the Legendre-Fenchel conjugate of $(-\ell\circ \operatorname{pr}_{1:N}^M)_{\operatorname{sym}}$.

\begin{lemma}
	\label{lem:conjugateSymmetrization}
	It holds for $M \geq N$ and any $z \in (\b{R}^n)^M$ that
	\begin{align*}
		((-\ell\circ \operatorname{pr}_{1:N}^M)_{\operatorname{sym}})^*(z)
		&= -\max\; \frac{(M-N)!}{M!} \sum_{\bs{l} \in \bs{\c{L}}} b_{\bs{l}}
	\end{align*}
	where the maximum is taken over all decision variables $(a_{\bs{l}},b_{\bs{l}})_{\bs{l} \in \bs{\c{L}}} \in (\b{R}^n)^N \times \b{R}$ under the constraints
	\begin{align*}
		z = \frac{(M-N)!}{M!} \sum_{\bs{l} \in \bs{\c{L}}} E_{\bs{l}}^\top a_{\bs{l}}\,, \quad 
		\mat{a_{\bs{l}} \\ b_{\bs{l}}} \in -\c{H}  &\quad \forall \bs{l} \in \bs{\c{L}}\,,
	\end{align*}
	and where $(-\ell\circ \operatorname{pr}_{1:N}^M)_{\operatorname{sym}}^*$ is the Legendre-Fenchel conjugate of $(-\ell\circ \operatorname{pr}_{1:N}^M)_{\operatorname{sym}}$.
\end{lemma}

\begin{proof}
We have
\begin{align*}
	\begin{gathered}
	(-\ell\circ \operatorname{pr}_{1:N}^M)_{\operatorname{sym}}(x)
	= \frac{1}{M!}\sum_{\pi \in \c{S}_M} -\ell(x_{\pi(1)},\ldots,x_{\pi(N)}) 
	%= \frac{1}{M!} \sum_{\bs{l} \in \bs{\c{L}}} \sum_{\substack{\pi \in \c{S}_M \\ \pi(1) = \bs{l}_1,\ldots,\pi(N) = \bs{l}_N}} -\ell(x_{\pi(1)},\ldots,x_{\pi(N)}) \\
	%= \frac{1}{M!} \sum_{\bs{l} \in \bs{\c{L}}} \sum_{\substack{\pi \in \c{S}_M \\ \pi(1) = \bs{l}_1,\ldots,\pi(N) = \bs{l}_N}} -\ell(x_{\bs{l}_1},\ldots,x_{\bs{l}_N}) \\
	%&= -\frac{1}{M!} \sum_{\bs{l} \in \bs{\c{L}}} (M-N)! \ell(x_{\bs{l}_1},\ldots,x_{\bs{l}_N})\\
	= \frac{(M-N)!}{M!} \sum_{\bs{l} \in \bs{\c{L}}} -\ell(x_{\bs{l}_1},\ldots,x_{\bs{l}_N}) \\
	= \frac{(M-N)!}{M!} \sum_{\bs{l} \in \bs{\c{L}}} \max_{h \in -\c{H}} h^\top \mat{x_{\bs{l}} \\ 1} 
	= \frac{(M-N)!}{M!} \max_{\bs{h} \in (-\c{H})^{\bs{\c{L}}}} \sum_{\bs{l} \in \bs{\c{L}}} \bs{h}_{\bs{l}}^\top \mat{E_{\bs{l}} x \\ 1}  \\
	= \frac{(M-N)!}{M!} \max_{\bs{h} \in (-\c{H})^{\bs{\c{L}}}} \left(\sum_{\bs{l} \in \bs{\c{L}}} \mat{E_{\bs{l}}^\top & 0 \\ 0 & 1}\bs{h}_{\bs{l}}\right)^\top \mat{x \\ 1} 
	=\max_{\hat{h} \in \hat{\c{H}}_M} \hat{h}^\top \mat{x \\ 1}\,, 
	\end{gathered}
\end{align*}
where 
\begin{align*}
	\hat{\c{H}}_M = -\frac{(M-N)!}{M!} \sum_{\bs{l} \in \bs{\c{L}}} \mat{E_{\bs{l}}^\top & 0 \\ 0 & 1}\c{H}
\end{align*}
is understood in the Minkowski sense.
By Lemma \ref{lem:conjugatePolyhedral} we obtain
\begin{align*}
	&((-\ell \circ \operatorname{pr}_{1:N}^M)_{\operatorname{sym}})^*(z) \\
	&\quad= -\max\left\{\hat{b} \mid \mat{z \\ \hat{b}} \in \hat{\c{H}}_M\right\} \\
	&\quad= -\max\left\{\hat{b} \mid \exists (\bs{h}_{\bs{l}})_{\bs{l}} \in (-\c{H})^{\bs{\c{L}}}\,:\; \mat{z \\ \hat{b}} = \frac{(M-N)!}{M!} \sum_{\bs{l} \in \bs{\c{L}}} \mat{E_{\bs{l}}^\top & 0 \\ 0 & 1}\bs{h}_{\bs{l}} \right\} \\
	&\quad= -\max\left\{\frac{(M-N)!}{M!} \sum_{\bs{l} \in \bs{\c{L}}} b_{\bs{l}} \mid \exists (a_{\bs{l}})_{\bs{l} \in \bs{\c{L}}}\in ((\b{R}^n)^N)^{\bs{\c{L}}} \,:\; \substack{z = \frac{(M-N)!}{M!} \sum_{\bs{l} \in \bs{\c{L}}} E_{\bs{l}}^\top a_{\bs{l}}\,, \\ \forall \bs{l} \in \bs{\c{L}}\,:\; \smat{a_{\bs{l}} \\ b_{\bs{l}}} \in -\c{H}}\right\}\,.
\end{align*}
\end{proof}

\begin{theorem}
	\label{thm:polyhedralConcaveLoss}
	Let $\rho > 0$, $\hat{P} = \sum_{i=1}^{n_{\hat{P}}} \hat{p}_i \delta_{\hat{\xi}_i}$, the loss be given by \eqref{eq:polyhedralConcaveLoss} and $c(x,y) = \norm{x-y}$ for some norm $\norm{\cdot}$ on $X = \b{R}^n$.
	Then
	\begin{align*}
		\operatorname{U}_M^{\operatorname{sym}}(\ell)
		= \inf \; \mu M \rho + \sum_{\bs{\iota} \in \bs{\c{I}}} \hat{p}_{\bs{\iota}} \sigma_{\bs{\iota}}
		\text{\ \ s.t.\ }
		\begin{cases}
			f^*(z_{\bs{\iota}}) - z_{\bs{\iota}}^\top \hat{\xi}_{\upsilon(\bs{\iota})} \leq \sigma_{\bs{\iota}}\,, & \forall\, \bs{\iota} \in \bs{\c{I}}\,, \\
			\norm{z_{\bs{\iota}}^j}_* \leq \mu\quad \forall j=1,\ldots,M\,, &\forall \bs{\iota} \in \bs{\c{I}}\,, \\
			\mu \geq 0\,,\; \sigma_{\bs{\iota}} \in \b{R} \,, \; z_{\bs{\iota}} \in (\b{R}^n)^M\,, &\forall \bs{\iota} \in \bs{\c{I}}\,,
		\end{cases}
	\end{align*}	
	where $f = ((-\ell\circ \operatorname{pr}_{1:N}^M)_{\operatorname{sym}})$, $\norm{\cdot}_*$ is the dual norm, $\hat{p}_{[\bs{i}]} = \abs{[\bs{i}]} \prod_{k=1}^M \hat{p}_{\bs{i}_k}$, $\hat{\xi}_{\bs{i}} = (\hat{\xi}_{\bs{i}_1},\ldots,\hat{\xi}_{\bs{i}_M}) \in (\b{R}^n)^M$ and the optimal value is independent of the choice of the selection $\upsilon$ in the first constraint.
\end{theorem}

\begin{proof}
First we note that for any $M \in \b{N}$ the cost $c^M$ is given by $c^M(x,y) = \norm{x-y}_M$, where $\norm{x}_M := \sum_{i=1}^M \norm{x_i}$ for $x \in (\b{R}^n)^M$.
Moreover, the dual norm of the latter is given by $\norm{x}_{M,*} = \max\{\norm{x_i}_* \mid i=1,\ldots,M\}$.
Additionally, if $\bs{I} = \{1,\ldots,n_{\hat{P}}\}^M$, then $\hat{P}^{\otimes M} = \sum_{\bs{i} \in \bs{J}} \hat{p}_{\bs{i}} \delta_{\hat{\xi}_{\bs{i}}}$ with $\hat{p}_{\bs{i}} = \prod_{k=1}^M \hat{p}_{\bs{i}_k}$ and $\hat{\xi}_{\bs{i}} = (\hat{\xi}_{\bs{i}_1},\ldots,\hat{\xi}_{\bs{i}_M}) \in (\b{R}^n)^M$.
Using \cite[Theorem 4.2]{mohajerin2018data} we obtain
\begin{align*}
	\operatorname{U}_M^{\operatorname{sym}}(\ell)
	&= \sup_{\bar{P} \in \b{B}_{M\rho}^{c^M}(\hat{P}^{\otimes M})} \b{E}_{\bar{P}} f \\
	&= \inf \; \mu M \rho + \sum_{\bs{i} \in \bs{J}} \hat{p}_{\bs{i}} \sigma_{\bs{i}} 
	\text{\ \ s.t.\ }
	\begin{cases}
		f^*(z_{\bs{i}}) - z_{\bs{i}}^\top \hat{\xi}_{\bs{i}} \leq \sigma_{\bs{i}}\,, & \forall\, \bs{i} \in \bs{I}\,, \\
		\norm{z_{\bs{i},j}}_* \leq \mu\quad \forall j=1,\ldots,M\,, &\forall \bs{i} \in \bs{I}\,, \\
		\mu \geq 0\,,\; \sigma_{\bs{i}} \in \b{R} \,, \; z_{\bs{i}} \in (\b{R}^n)^M\,, &\forall \bs{i} \in \bs{I}\,.
	\end{cases}
\end{align*}
Now suppose that $(\mu,(\sigma_{\bs{i}})_{\bs{i} \in \bs{I}},(z_{\bs{i}})_{\bs{i} \in \bs{I}})$ is any feasible point to the above program.
For each $\bs{\iota} \in \bs{\c{I}}$ set $\bs{j} = \argmin_{\bs{j} \in \bs{\iota}} \sigma_{\bs{j}}$ and set $\sigma_{\bs{\iota}} = \sigma_{\bs{j}}$ and $z_{\bs{\iota}} = z_{\bs{j}}$.
This yields a feasible point $(\mu,(\sigma_{\bs{\iota}})_{\bs{\iota} \in \bs{\c{I}}},(z_{\bs{\iota}})_{\bs{\iota} \in \bs{\c{I}}})$ for the program \eqref{eq:polyhedralConcaveLossProgram} with objective value
\begin{align*}
	\mu M \rho + \sum_{\bs{\iota} \in \bs{\c{I}}} \hat{p}_{\bs{\iota}} \sigma_{\bs{\iota}}
	\leq \mu M \rho + \sum_{\bs{\iota} \in \bs{\c{I}}} \sum_{\bs{i} \in \bs{\iota}} \hat{p}_{\bs{i}} \sigma_{\bs{i}} 
	= \mu M \rho + \sum_{\bs{i} \in \bs{J}} \hat{p}_{\bs{i}} \sigma_{\bs{i}} \,.
\end{align*}
%Thus the optimal value of \eqref{eq:polyhedralConcaveLossProgram} is not larger than the optimal value of the above program.
Conversely, suppose that $(\mu,(\sigma_{\bs{\iota}})_{\bs{\iota} \in \bs{\c{I}}},(z_{\bs{\iota}})_{\bs{\iota} \in \bs{\c{I}}})$ is a feasible point for \eqref{eq:polyhedralConcaveLossProgram} with first constraint containing $\hat{\xi}_{\bs{i}}$ for a fixed member $\bs{i} = \nu(\bs{\iota})$ for each class $\bs{\iota} \in \bs{\c{I}}$.
Setting $z_{\bs{i}} = z_{\bs{\iota}}$ and $\sigma_{\bs{i}} = \sigma_{\bs{\iota}}$ we obtain that 
\begin{align*}
	f^*(z_{\bs{i}}) - z_{\bs{i}}^\top \hat{\xi}_{\bs{i}} \leq \sigma_{\bs{i}} \text{\ and\ } \norm{z_{\bs{i},j}}_* \leq \mu
\end{align*}
for each such member $\bs{i}$ of $\bs{\iota}$.
For other members $\bs{j} \in \bs{\iota}$ we set $z_{\bs{j}} = (\pi \otimes I_n) z_{\bs{\iota}}$ and $\sigma_{\bs{j}} = \sigma_{\bs{i}}$, where $\pi \in \c{S}_M$ is such that $\pi(\bs{j}) = \bs{i}$.
Since $f$ is symmetric, so is $f^*$ and from $(\pi \otimes I_n) \hat{\xi}_{\bs{j}} = \hat{\xi}_{\bs{i}}$ we obtain
\begin{align*}
	\sigma_{\bs{j}}
	= \sigma_{\bs{i}}
	\geq f^*(z_{\bs{i}}) - z_{\bs{i}}^\top \hat{\xi}_{\bs{i}}
	= f^*(z_{\bs{j}}) - z_{\bs{j}}^\top \hat{\xi}_{\bs{j}}
\end{align*}
as well as $\norm{((\pi \otimes I_n)z_{\bs{j}})_j}_* = \norm{z_{\bs{j},\pi(j)}}_* \leq \mu$ for any $j=1,\ldots,M$. Thus \\
$(\mu,(\sigma_{\bs{i}})_{\bs{i} \in \bs{I}},(z_{\bs{i}})_{\bs{i} \in \bs{I}})$ is a feasible point with objective value
\begin{align*}
	\mu M \rho + \sum_{\bs{i} \in \bs{J}} \hat{p}_{\bs{i}} \sigma_{\bs{i}}
	= \mu M \rho + \sum_{\bs{\iota} \in \bs{\c{I}}} \sum_{\bs{i} \in \bs{\iota}} \hat{p}_{\bs{i}} \sigma_{\bs{i}}
	= \mu M \rho + \sum_{\bs{\iota} \in \bs{\c{I}}} \sigma_{\bs{\iota}} \sum_{\bs{i} \in \bs{\iota}} \hat{p}_{\bs{i}}
    =\mu M \rho + \sum_{\bs{\iota} \in \bs{\c{I}}} \hat{p}_{\bs{\iota}} \sigma_{\bs{\iota}}\,.
\end{align*}
\end{proof}

\begin{proof}[Proof of Theorem \ref{thm:polyhedralConcaveLossProgram}]
Combine Theorem \ref{thm:polyhedralConcaveLoss} with Lemma \ref{lem:conjugateSymmetrization}.
\end{proof}

%\section{Proofs of Section \ref{sec:structuredDROTightnessOfRelaxationSequence}}

\section{Auxiliary results}

Here we collect some results that, while not technically sophisticated or new, were difficult to locate in the existing literature in the stated form.

\begin{lemma}
	\label{lem:wassersteinWeakTriangleInequality}
	For any transportation cost $c:X \times X \to [0,\infty]$ satisfying the weak triangle inequality, the corresponding Wasserstein distance $W_c$ also satisfies the weak triangle inequality (with the same constant).
\end{lemma}

\begin{proof}
	Let $P_1,P_2,P_3 \in \c{P}(X)$, $\epsilon > 0$ and let $\Lambda_{12} \in \Gamma(P_1,P_2)$, $\Lambda_{23} \in \Gamma(P_2,P_3)$ be such that $\intg{}{c}{\Lambda_{12}} \leq W_c(P_1,P_2) + \epsilon$ and $\intg{}{c}{\Lambda_{23}} \leq W_c(P_2,P_3) + \epsilon$.
	By the gluing lemma \cite[Gluing lemma]{villani2009optimal} there exists a multimarginal coupling $\Lambda \in \Gamma(P_1,P_2,P_3)$ such that $\operatorname{pr}_{1,2}\# \Lambda = \Lambda_{12}$ and $\operatorname{pr}_{2,3}\# \Lambda = \Lambda_{23}$.
	Then $\operatorname{pr}_{1,3}\# \Lambda \in \Gamma(P_1,P_3)$ and thus
	\begin{align*}
		W_c(P_1,P_3) 
		&\leq \intg{}{c}{\operatorname{pr}_{1,3}\# \Lambda} 
		= \intg{}{c(x_1,x_3)}{\Lambda(x_1,x_2,x_3)} \\
		&\leq \intg{}{K(c(x_1,x_2)+c(x_2,x_3))}{\Lambda(x_1,x_2,x_3)} \\
		&= K \left(\intg{}{c}{\Lambda_{23}} + \intg{}{c}{\Lambda_{12}}\right) \\
		&\leq K (W_c(P_1,P_2) + W_c(P_2,P_3)) + 2K\epsilon\,.
	\end{align*}
	Since $\epsilon > 0$ was arbitrary, the result follows.
\end{proof}

\begin{lemma}
	\label{lem:wassersteinBallUniformlyBounded}
	For any transportation cost $c:X \times X \to [0,\infty)$ satisfying the weak triangle inequality, $\hat{P} \in \c{P}_c(X)$ and $x_0 \in X$ it holds that
	\begin{align*}
		\sup_{P \in \b{B}_\rho^c(\hat{P})} \intg{}{c(x_0,\cdot)}{P} < \infty\,.
	\end{align*}
\end{lemma}

For transportation costs of the form $c = d^p$ this was shown in \cite[Lemma 1]{yue2022linear}.

\begin{proof}
	Let $K > 0$ be the constant from the weak triangle inequality. 
	It holds by Lemma \ref{lem:wassersteinWeakTriangleInequality} that
	\begin{align*}
		\intg{}{c(x_0,\cdot)}{P}
		= W_c(\delta_{x_0},P)
		\leq K W_c(\delta_{x_0},\hat{P}) + K W_c(P,\hat{P})
		\leq K W_c(\delta_{x_0},\hat{P}) + K \rho 
	\end{align*}
	for any $P \in \b{B}_\rho^c(\hat{P})$.
\end{proof}

\begin{lemma}
	\label{lem:properCostCompactWassersteinBall}
	Let $X$ be a Polish space and $c$ a proper transportation cost on $X$.
	Then for any $\rho > 0$ and $\hat{P} \in \c{P}_c(X)$ the Wasserstein ball $\b{B}_\rho^c(\hat{P})$ is weakly compact.	
\end{lemma}

In \cite[Theorem 1]{yue2022linear} this result is shown for the special case of $c = d^p$ for some metric $d$ on $X$ and $p \in [1,\infty)$.

\begin{proof}
	The proof follows the proof of \cite[Theorem 1]{yue2022linear}.
	Weak closedness follows from the fact that $\b{B}_\rho^c(\hat{P})$ is the sublevel set of the weakly lower-semicontinuous map $P \mapsto W_c(P,\hat{P})$.
	To establish compactness, it suffices to show tightness of $\b{B}_\rho^c(\hat{P})$.
	Let $x_0 \in X$ be arbitrary.
	By Lemma \ref{lem:wassersteinBallUniformlyBounded} the quantity $C = \sup_{P \in \b{B}_\rho^c(\hat{P})} \intg{}{c(x_0,\cdot)}{P}$ is finite.
	Let $\epsilon > 0$ and set $K = \{c(x_0,\cdot) > \frac{C}{\epsilon}\} \subseteq X$, which is compact, because $c$ is proper.
	Then for any $P \in \b{B}_\rho^c(\hat{P})$ it holds
	\begin{align*}
		P(\{c(x_0,\cdot) > \tfrac{C}{\epsilon}\})
		\leq \frac{\intg{}{c(x_0,\cdot)}{P}}{C/\epsilon}
		= \epsilon\,,
	\end{align*}
	which shows that $\b{B}_\rho^c(\hat{P})$ is tight.
\end{proof}

\begin{definition}
	\label{def:symmetricCoupling}
	Let $P_1,P_2 \in \c{P}(X^N)$.
	Then a coupling $\Lambda \in \Gamma(P_1,P_2)$ is called symmetric if for all permutations $\pi \in \c{S}_N$ it holds $(\pi \oplus \pi)\# \Lambda = \Lambda$.
	The set of symmetric couplings between $P_1, P_2$ is denoted by $\Gamma_{\operatorname{sym}}(P_1,P_2)$.
\end{definition}

\begin{lemma}
	\label{lem:symmetricCoupling}
	Let $c:X \times X \to [0,\infty)$ be any transportation cost and $P_1,P_2 \in \c{P}(X^N)$.
	Then for each coupling $\Lambda \in \Gamma(P_1,P_2)$ there exists a symmetric coupling $\hat{\Lambda} \in \Gamma_{\operatorname{sym}}((P_1)_{\operatorname{sym}},(P_2)_{\operatorname{sym}})$ such that $\intg{}{c^N}{\Lambda} = \intg{}{c^N}{\hat{\Lambda}}$.
\end{lemma}

\begin{proof}
	Set $\hat{\Lambda} = \frac{1}{N!} \sum_{\pi \in \c{S}_N} (\pi \oplus \pi)\#\Lambda$.
	Then clearly $\hat{\Lambda} \in \Gamma_{\operatorname{sym}}((P_1)_{\operatorname{sym}},(P_2)_{\operatorname{sym}})$ and 
	\begin{align*}
		\intg{}{c^N}{\hat{\Lambda}}
		= \frac{1}{N!} \sum_{\pi \in \c{S}_N} \intg{}{c^N}{(\pi \oplus \pi)\#\Lambda}
		= \frac{1}{N!} \sum_{\pi \in \c{S}_N} \intg{}{c^N}{\Lambda}
		= \intg{}{c^N}{\Lambda}\,,
	\end{align*}
	where the second equality holds because $c^N \circ (\pi \oplus \pi) = c^N$ for any $\pi \in \c{S}_N$.
\end{proof}

\begin{lemma}
	\label{lem:powerMeasureWeaklyContinuous}
	The map $\c{P}(X) \to \c{P}(X^N): P \mapsto P^{\otimes N}$ is continuous in the weak topology for $N \in \b{N} \cup \{\infty\}$.
\end{lemma}

\begin{proof}
	If $N \in \b{N}$, then continuity can be shown by using characteristic functions, Levy's continuity theorem \cite{dudley2018real} and exploiting the fact that the characteristic function of a product distribution is the product of characteristic functions.
	Now suppose that $N = \infty$.
	Since the weak topology on $\c{P}(X)$ is metrizable \cite[Theorem 15.15]{guide2006infinite}, it is enough to show that the above map is sequentially continuous. 
	Let $P_n \in \c{P}(X)$ be a sequence of measures such that $P_n \to P \in \c{P}(X)$ weakly as $n \to \infty$. 
	Then $\{P_n\}_{n \in \b{N}}$ is tight and from this one can show that $\{P_n^{\otimes \infty}\}_{n \in \b{N}}$ is tight as well.
	Indeed, let $\epsilon > 0$ be given.
	Pick for each $l \in \b{N}$ some compact $K_l \subseteq X$ such that $P_n(K_l) \geq 1-\epsilon/2^l$ for all $n \in \b{N}$.
	Then $K = \prod_{l=1}^\infty K_l \subseteq X^\infty$ is compact and $P_n^{\otimes \infty}(K) \geq \prod_{l=1}^\infty (1-\epsilon/2^l) \geq 1-\epsilon$ for any $n \in \b{N}$.
	Therefore there exist some subsequence $P_{n_k}^{\otimes \infty}$ with $P_{n_k}^{\otimes \infty} \to \bar{P} \in \c{P}(X^\infty)$ weakly as $k \to \infty$.
	For the same subsequence we also have $P_{n_k}^{\otimes N} \to P^{\otimes N}$ weakly as $k \to \infty$.
	Therefore any finite-coordinate marginal of $\bar{P}$ is of the form $P^{\otimes N}$ and since the system of all cylinder sets forms a $\pi$-system on which $\bar{P}$ and $P^{\otimes \infty}$ agree, it follows that $\bar{P} = P^{\otimes \infty}$ \cite{dudley2018real}.
	Hence $P \mapsto P^{\otimes \infty}$ is weakly continuous.
\end{proof}

\begin{lemma}
	\label{lem:weakClosureIntegralRepresentation}
	Let $c:X \times X \to [0,\infty)$ be a proper transportation cost.
	Then for any $\rho > 0$ and $\hat{P} \in \c{P}_c(X)$ and $N \in \b{N} \cup \{\infty\}$ it holds
	\begin{align*}
		\convw \c{W}
		= \left\{\intg{}{P^{\otimes \infty}}{\nu(P)} \; \middle| \; \nu \in \c{P}(\c{P}(X))\,,\; \nu(\b{B}_\rho^c(\hat{P})) = 1 \right\}\,,
	\end{align*}  
	where $\c{W} = \{P^{\otimes \infty} \mid P \in \b{B}_\rho^c(\hat{P})\}$.
\end{lemma}

\begin{proof}
	Denote the set on the right by $\s{M}$.
	Then, by taking finitely supported mixing measures $\nu \in \c{P}(\c{P}(X))$, we see that $\operatorname{conv} \c{W} \subseteq \s{M}$.
	Moreover, $\s{M}$ is weakly closed.
	Indeed, denote $J(\nu) = \intg{}{P^{\otimes \infty}}{\nu(P)}$ for $\nu \in \c{P}(\b{B}_\rho^c(\hat{P}))$.
	Then let $J(\nu_n) \to \bar{P}$ weakly for $n \to \infty$.
	Since $\b{B}_\rho^c(\hat{P}) \subseteq \c{P}(X)$ is weakly compact \cite[Theorem 1]{yue2022linear}, the entire set $\c{P}(\b{B}_\rho^c(\hat{P}))$ is weakly compact in $\c{P}(\c{P}(X))$ \cite[Theorem 15.11]{guide2006infinite} and it follows that there exists some subsequence $\nu_{n_k}$ with $\nu_{n_k} \to \nu \in \c{P}(\b{B}_\rho^c(\hat{P}))$ weakly for $k \to \infty$.
	If $f \in C(X^\infty)$ is bounded, then
	\begin{gather*}
		\intg{}{f}{\bar{P}} 
		= \lim_{k \to \infty} \intg{}{f}{J(\nu_{n_k})}
		= \lim_{k \to \infty} \intg{}{\intg{}{f}{P^{\otimes \infty}}}{\nu_{n_k}(P)} \\
		= \intg{}{\intg{}{f}{P^{\otimes \infty}}}{\nu(P)}
		= \intg{}{f}{J(\nu)}\,,
	\end{gather*}
	where in the third inequality we have used the fact that $F:P \mapsto \intg{}{f}{P^{\otimes \infty}}$ satisfies $F \in C_b(\c{P}(X))$, since it is the composition of the continuous (in the respective topologies) maps $P \mapsto P^{\otimes \infty}$ (Lemma \ref{lem:powerMeasureWeaklyContinuous}) and $\tilde{P} \mapsto \intg{}{f}{\tilde{P}}$ (by definition). 
	Thus $\bar{P} = J(\nu)$ and $\s{M}$ is weakly closed and $\convw \c{W} \subseteq \s{M}$.
	Now take any $J(\nu) \in \s{M}$.
	Then, since the finitely supported measures are weakly dense in $\c{P}(\b{B}_\rho^c(\hat{P}))$ \cite[15.10 Density Theorem]{guide2006infinite}, it follows that there is a sequence $\nu_n$ of such measures such that $\nu_n \to \nu$ weakly.
	Clearly it holds that $J(\nu_n) \in \operatorname{conv} \c{W}$ for each $n \in \b{N}$.
	Moreover, it holds that $J(\nu_n) \to J(\nu)$ weakly in $\c{P}(X)$, which can be shown by similar arguments as above.
	Thus $J(\nu) \in \convw \c{W}$ and hence $\s{M} \subseteq \convw \c{W}$.
\end{proof}

The following result was difficult to locate in the literature, but was proven in \cite{mathoverflow164842}.
Due to the nature of the source, we include the proof here.

\begin{lemma}[\cite{mathoverflow164842}]
	\label{lem:integralConvexHull}
	Let $(X,\mu)$ be a probability space and $h:X \to \b{R}^n$ be Borel measurable and $\mu$-integrable.
	Then $\intg{}{h}{\mu} \in \operatorname{conv} h(X)$.\footnote{Note that $h(X)$ itself does not need to be Borel measurable here}
\end{lemma}

\begin{proof}
	We proceed by induction on $n$. 
	For $n = 1$ this is true.
	Indeed, by the definition of the Lebesgue integral it holds $\intg{}{h}{\mu} \in \bar{\operatorname{conv}} h(X)$.
	Since $I = \bar{\operatorname{conv}} h(X)$ is a closed (possibly unbounded) interval, it suffices to verify that if $\intg{}{h}{\mu}$ is an endpoint $a \in \partial I$, then $a \in h(X)$.
	But if $\intg{}{h}{\mu} = a \in \partial I$, then either $h(X) \subseteq [a,\infty)$ or $h(X) \subseteq (-\infty,a]$.
	If, say, $h(x) > a$ for all $x \in X$, then $0 < \intg{}{h-a}{\mu} = \intg{}{h}{\mu} - a$, a contradiction. 
	Now assume that the claim is true for some $n-1 \geq 1$.
	Let $y = \intg{}{h}{\mu}$.
	If $y \notin \operatorname{conv} h(X)$, then there exists a separating hyperplane $z \in \b{R}^n \setminus \{0\}$ such that $\<z,y\> \leq \<z,h(x)\>$ for all $x \in X$.
	Now the set $E_+ = \{x \in X \mid \<z,y\> < \<z,h(x)\>\}$ is a $\mu$-nullset, since if not, then $0 > \intg{E}{\<z,y\> - \<z,h(x)\>}{\mu(x)} = 0$, a contradiction.
	Similarly $E_- = \{x \in X \mid \<z,y\> > \<z,h(x)\>\}$ is $\mu$-nullset.
	Pick an arbitrary $x_0 \in X$ and let $\tilde{h}(x) = h(x)$ for $x \notin E_+ \cup E_-$ and $\tilde{h}(x) = h(x_0)$ if $x \in E_+ \cup E_-$.
	We obtain that $\intg{}{\tilde{h}}{\mu} = y$, and $\tilde{h}(X) \subseteq h(X)$ as well as $\tilde{h}(X) \subseteq U = \{w \in \b{R}^n \mid \<w-y,z\> = 0\}$.
	Let $T:U \to \b{R}^{n-1}$ be an invertible linear transformation and set $\hat{h} = T\circ \tilde{h}:X \to \b{R}^{n-1}$.
	Since $\operatorname{dim} U = n-1$, such a transformation exists.
	By induction hypothesis $Ty = \intg{}{\hat{h}}{\mu} \in \operatorname{conv} \hat{h}(X) = T \operatorname{conv} \tilde{h}(X) \subseteq T \operatorname{conv} h(X)$ and thus $y \in \operatorname{conv} h(X)$, a contradiction.
	Hence $y \in \operatorname{conv} h(X)$ in the first place, which finishes the induction step.
\end{proof}

\begin{lemma}
	\label{lem:existenceFinitelySupportedMeasure}
	Let $X$ be a measurable space and let $f,g:X \to \b{R}$ be Borel measurable.
	If there exists some Borel probability measure $\mu \in \c{P}(X)$ such that $f,g \in L^1(X,\mu)$, then the set
	\begin{align*}
		\left\{\nu \in \c{P}(X) \; \middle| \; f,g \in L^1(X,\nu) \text{\ and\ } \intg{}{g}{\nu} = \intg{}{g}{\mu}\,,\; \intg{}{f}{\nu} = \intg{}{f}{\mu} \right\}
	\end{align*} 
	contains a finitely supported measure.
\end{lemma}

\begin{proof}
	Let $h:X \mapsto \b{R}^2:x \mapsto (f(x),g(x))$.
	Then clearly $\intg{}{h}{\mu} \in \bar{\operatorname{conv}} h(X)$, while $\operatorname{conv} h(X) = \{\intg{}{h}{\nu} \mid \nu \in \c{P}(X)\,,\; \abs{\operatorname{supp} \nu} < \infty \}$.
	Thus the task is to show that actually $\intg{}{h}{\mu} \in \operatorname{conv} h(X)$, which follows from Lemma \ref{lem:integralConvexHull}.
\end{proof}

\begin{lemma}
	\label{lem:integrationFunctionalBorel}
	Let $X$ be a Polish space and $\ell:X \to \b{R}$ be Borel measurable.
	Define $F:\c{P}(X) \mapsto \bar{\b{R}}:\bar{P} \mapsto \intg{}{\ell}{\bar{P}}$, where we set $F(\bar{P}) = \infty$ if $\ell$ is not $\bar{P}$-integrable.
	Then $F$ is Borel measurable. 
\end{lemma}

\begin{proof}
	Let $\c{P}_\ell(X) = \{\bar{P} \in \c{P}(X) \mid \ell \in L^1(X,\bar{P})\} = \operatorname{dom}(F)$ and observe that $\c{P}_\ell(X)$ is convex and $F$ satisfies $F(\alpha P_1 + (1-\alpha)P_2) = \alpha F(P_1) + (1-\alpha)F(P_2)$ for all $\alpha \in [0,1]$ and $P_1,P_2 \in \c{P}_\ell(X)$.
	The set $\c{P}_\ell(X)$ is a Borel set w.r.t. the weak topology.
	This follows from the fact that $\c{P}_\ell(X) = \bigcup_{N \in \b{N}} \bigcap_{M \in \b{N}} \{\bar{P} \in \c{P}(X) \mid \intg{}{\abs{\ell}_{M}}{\bar{P}} \leq N\}$, where $\abs{\ell}_M = \min\{\abs{\ell},M\}$, and that $\bar{P} \mapsto \intg{}{\abs{\ell}_{M}}{\bar{P}}$ is Borel measurable for each $M$ by \cite[Theorem 15.13]{guide2006infinite}.
	Moreover, we also note that 
	\begin{align*}
		F(\bar{P}) = \lim_{M \to \infty} \int{\ell\mathbbm{1}_{|\ell| \leq M}}{\bar{P}} \text{\ \ for all\ } \bar{P} \in \c{P}_\ell(X)\,,
	\end{align*}
	where each $\bar{P} \mapsto \int{\ell\mathbbm{1}_{|\ell|\leq M}}{\bar{P}}$ is again Borel measurable by \cite[Theorem 15.13]{guide2006infinite}.
	The Borel measurability of $F$ follows now from the fact that pointwise limits of Borel measurable functions are again Borel measurable \cite{bogachev2007measure}.
\end{proof}

Finally we will also need the following

\begin{lemma}
	\label{lem:integralExtremePoints}
	Let $Y, Z$ be Polish spaces and let $\s{Z} \subseteq \c{P}(Z)$ be a weakly closed, convex set with its extreme points denoted by $\operatorname{ext} \s{Z}$.
	Let $\nu \in \c{P}(Y)$ and $F:Y \to \c{P}(Z)$ be measurable w.r.t. Borel-$\sigma$-algebra on $\c{P}(Z)$ such that $F(y) \in \s{Z}$ for $\nu$-almost all $y \in Y$.
	If
	\begin{align*}
		\intg{Y}{F(y)}{\nu(y)} \in \operatorname{ext} \s{Z}\,,
	\end{align*} 
	then there exists $\zeta \in \operatorname{ext} \s{Z}$ such that $F(y) = \zeta$ for $\nu$-almost all $y \in Y$.
\end{lemma}

\begin{proof}
	If there exists a single $\zeta \in \c{P}(Z)$ such that $F(y) = \zeta$ for $\nu$-almost all $y \in Y$, then the claim is obvious.
	Suppose now that $F$ is not constant $\nu$-almost everywhere.
	Then there exist Borel subsets $Y_i \subseteq Y$ such that $\nu(Y_i) > 0$ for $i=1,2$ and 
	\begin{align*}
		\intg{Y_1}{F(y)}{\nu(y)} \neq \intg{Y_2}{F(y)}{\nu(y)}\,.
	\end{align*}
	Then we obtain
	\begin{align*}
		\intg{Y}{F(y)}{\nu(y)} 
		= \nu(Y_1) \underbrace{\intg{Y_1}{F(y)}{\nu_1(y)}}_{\in \s{Z}} + \nu(Y_2) \underbrace{\intg{Y_2}{F(y)}{\nu_2(y)}}_{\in \s{Z}}
	\end{align*}
	with $\nu_i \in \c{P}(Y_i)$ defined by $\nu_i(A) = \frac{\nu(A)}{\nu(Y_i)}$ for $A \subseteq Y_i$ Borel and $i=1,2$.
	Since $\nu(Y_1) + \nu(Y_2) = \nu(Y) = 1$, this contradicts $\intg{Y}{F(y)}{\nu(y)} \in \operatorname{ext} \s{Z}$ and hence finishes the proof.
\end{proof}

\begin{lemma}
	\label{lem:wassersteinDistanceProduct}
	For any transportation cost $c:X \times X \to [0,\infty]$, $P_1,P_2 \in \c{P}(X)$ and $N \in \b{N}$ it holds that $W_{c^N}(P_1^{\otimes N},P_2^{\otimes N}) = N W_c(P_1,P_2)$.
\end{lemma}

\begin{proof}
	First we show $\leq$.
	Let $\epsilon > 0$ be given. 
	Then we can find a coupling $\tilde{\Lambda} \in \Gamma(P_1,P_2)$ such that $\intg{}{c}{\tilde{\Lambda}} \leq W_c(P_1,P_2) + \epsilon$.
	Define the distribution $\Lambda \in \c{P}(X^N \times X^N)$ by $\Lambda = \tilde{\Lambda}^{\otimes N}$, where $(X\times X)^N$ is identified with $X^N \times X^N$.
	Clearly $\Lambda \in \Gamma(P_1^{\otimes N},P_2^{\otimes N})$ and it holds that
	\begin{align*}
		&\intg{}{c^N(x_1,\ldots,x_N,y_1,\ldots,y_N)}{\Lambda(x_1,\ldots,x_N,y_1,\ldots,y_N)} \\
		&\quad= \sum_{i=1}^N \intg{}{c(x_i,y_i)}{\tilde{\Lambda}^{\otimes N}(x_1,\ldots,x_N,y_1,\ldots,y_N)} \\
		&\quad= \sum_{i=1}^N \intg{}{c(x_i,y_i)}{\tilde{\Lambda}(x_i,y_i)} \\
		&\quad\leq N (W_c(P_1,P_2) + \epsilon)\,.
	\end{align*}
	Since $\epsilon$ was arbitrary, $\leq$ follows. 
	Now we show $\geq$.
	Let again $\epsilon > 0$ and $\Lambda \in \Gamma(P_1^{\otimes N},P_2^{\otimes N})$ be a coupling such that $\intg{}{c^N}{\Lambda} \leq W_{c^N}(P_1^{\otimes N},P_2^{\otimes N}) + \epsilon$.
	Then it is easy to see that $\tilde{\Lambda}_i := \operatorname{pr}_{i,i} \# \Lambda \in \Gamma(P_1,P_2)$ for all $i=1,\ldots,N$.
	Moreover, for $\tilde{\Lambda} = \frac{1}{N}\sum_{i=1}^N \tilde{\Lambda}_i$ it holds that $\tilde{\Lambda} \in \Gamma(P_1,P_2)$, since averages of couplings are again couplings.
	Then
	\begin{align*}
		W_c(P_1,P_2)
		\leq \intg{}{c}{\tilde{\Lambda}}
		= \frac{1}{N} \sum_{i=1}^N \intg{}{c}{\tilde{\Lambda}_i}
		= \frac{1}{N} \intg{}{c^N}{\Lambda} 
		\leq \frac{1}{N} (W_{c^N}(P_1^{\otimes N},P_2^{\otimes N}) + \epsilon)\,. 
	\end{align*}
	Since $\epsilon$ was again arbitrary, this finishes the proof.
\end{proof}

\begin{lemma}
\label{lem:conjugatePolyhedral}
	Suppose $f:\b{R}^n \to \b{R}$ is a polyhedral function given by
	\begin{align*}
		f(x) = \max_{h \in \c{H}} h^\top \mat{x \\ 1}\,, \quad x \in \b{R}^n\,, 
	\end{align*}	
	where $\c{H} \subseteq \b{R}^{n+1}$ is a polytope.
	Then the Legendre-Fenchel conjugate is given by
	\begin{align*}
		f^*(z) = -\max\left\{ b \mid \mat{z \\ b} \in \c{H}\right\}\,, \quad z \in \b{R}^n\,.
	\end{align*}
\end{lemma}

\begin{proof}
	We have 
	\begin{gather*}
		f^*(z)
		= \sup_{x \in \b{R}^n} x^\top z - \left(\max_{h \in \c{H}} h^\top \mat{x \\ 1}\right) 
		= \sup_{x \in \b{R}^n} \mat{z \\ 0}^\top \mat{x \\ 1} + \left(\min_{h \in \c{H}} -h^\top \mat{x \\ 1}\right) \\
		= \sup_{x \in \b{R}^n} \min_{h \in \c{H}} \left(\mat{z \\ 0} - h \right)^\top \mat{x \\ 1} 
		= \min_{h \in \c{H}} \sup_{x \in \b{R}^n}  \left(\mat{z \\ 0} - h \right)^\top \mat{x \\ 1} \\
		= \min\left\{-b \mid \mat{z \\ b} \in \c{H} \right\} 
		= -\max\left\{b \mid \mat{z \\ b} \in \c{H} \right\} 
	\end{gather*}
	where in the fourth equality we have used the Sions minimax theorem and in the fifth the fact that if $h = \smat{a \\ b} \in \c{H}$, then the inner supremum is finite if and only if $a = z$.
\end{proof}

%% file: bibliography.bib
@book{villani2009optimal,
  title={Optimal transport: old and new},
  author={Villani, C{\'e}dric and others},
  volume={338},
  year={2009},
  publisher={Springer}
}

@incollection{kuhn2019wasserstein,
  title={Wasserstein distributionally robust optimization: Theory and applications in machine learning},
  author={Kuhn, Daniel and Esfahani, Peyman Mohajerin and Nguyen, Viet Anh and Shafieezadeh-Abadeh, Soroosh},
  booktitle={Operations research \& management science in the age of analytics},
  pages={130--166},
  year={2019},
  publisher={Informs}
}

@article{shafieezadeh2023new,
  title={New perspectives on regularization and computation in optimal transport-based distributionally robust optimization},
  author={Shafieezadeh-Abadeh, Soroosh and Aolaritei, Liviu and D{\"o}rfler, Florian and Kuhn, Daniel},
  journal={arXiv preprint arXiv:2303.03900},
  year={2023}
}

@conference{chaouach2022tightening,
title = "Tightening ambiguity set characterizations for data-driven distributionally robust optimization",
author = "Lotfi Chaouach and Dimitris Boskos and Tom Oomen",
year = "2022",
language = "English",
pages = "189--189",
publisher={41st Benelux Meeting on Systems and Control 2022},
}

@article{chaouach2023structured,
  title={Structured ambiguity sets for distributionally robust optimization},
  author={Chaouach, Lotfi M and Oomen, Tom and Boskos, Dimitris},
  journal={arXiv preprint arXiv:2310.20657},
  year={2023}
}

@book{bogachev2007measure,
  title={Measure theory},
  author={Bogachev, Vladimir Igorevich and Ruas, Maria Aparecida Soares},
  volume={1},
  year={2007},
  publisher={Springer}
}

@article{gao2023distributionally,
  title={Distributionally robust stochastic optimization with Wasserstein distance},
  author={Gao, Rui and Kleywegt, Anton},
  journal={Mathematics of Operations Research},
  volume={48},
  number={2},
  pages={603--655},
  year={2023},
  publisher={INFORMS}
}

@article{blanchet2019quantifying,
  title={Quantifying distributional model risk via optimal transport},
  author={Blanchet, Jose and Murthy, Karthyek},
  journal={Mathematics of Operations Research},
  volume={44},
  number={2},
  pages={565--600},
  year={2019},
  publisher={INFORMS}
}

@book{willard2012general,
  title={General topology},
  author={Willard, Stephen},
  year={2012},
  publisher={Courier Corporation}
}

@book{guide2006infinite,
  title={Infinite dimensional analysis: A hitchhiker’s guide},
  author={Aliprantis, Charalambos and Border, Kim},
  year={2006},
  publisher={Springer Science \& Business Media}
}

@book{dudley2018real,
  title={Real analysis and probability},
  author={Dudley, Richard M},
  year={2018},
  publisher={Chapman and Hall/CRC}
}

@article{hewitt1955symmetric,
  title={Symmetric measures on Cartesian products},
  author={Hewitt, Edwin and Savage, Leonard},
  journal={Transactions of the American Mathematical Society},
  volume={80},
  number={2},
  pages={470--501},
  year={1955},
  publisher={JSTOR}
}

@article{yue2022linear,
  title={{On linear optimization over Wasserstein balls}},
  author={Yue, Man-Chung and Kuhn, Daniel and Wiesemann, Wolfram},
  journal={Mathematical Programming},
  volume={195},
  number={1},
  pages={1107--1122},
  year={2022},
  publisher={Springer}
}

@article{rahimian2019distributionally,
  title={Distributionally robust optimization: A review},
  author={Rahimian, Hamed and Mehrotra, Sanjay},
  journal={arXiv preprint arXiv:1908.05659},
  year={2019}
}

@article{kantorovitch1958translocation,
  title={On the translocation of masses},
  author={Kantorovitch, Leonid},
  journal={Management science},
  volume={5},
  number={1},
  pages={1--4},
  year={1958},
  publisher={INFORMS}
}

@article{bogachev2012monge,
  title={The Monge-Kantorovich problem: achievements, connections, and perspectives},
  author={Bogachev, Vladimir and Kolesnikov, Aleksandr},
  journal={Russian Mathematical Surveys},
  volume={67},
  number={5},
  pages={785},
  year={2012},
  publisher={IOP Publishing}
}

@article{vershik2006kantorovich,
  title={Kantorovich metric: Initial history and little-known applications},
  author={Vershik, Anatoly},
  journal={Journal of Mathematical Sciences},
  volume={133},
  pages={1410--1417},
  year={2006},
  publisher={Springer}
}

@book{panaretos2020invitation,
  title={An invitation to statistics in Wasserstein space},
  author={Panaretos, Victor and Zemel, Yoav},
  year={2020},
  publisher={Springer Nature}
}

@book{shapiro2021lectures,
  title={Lectures on stochastic programming: modeling and theory},
  author={Shapiro, Alexander and Dentcheva, Darinka and Ruszczynski, Andrzej},
  year={2021},
  publisher={SIAM}
}

@article{zhang2024short,
  title={{A short and general duality proof for Wasserstein distributionally robust optimization}},
  author={Zhang, Luhao and Yang, Jincheng and Gao, Rui},
  journal={Operations Research},
  year={2024},
  publisher={INFORMS}
}

@book{rudin2017fourier,
  title={Fourier analysis on groups},
  author={Rudin, Walter},
  year={2017},
  publisher={Courier Dover Publications}
}

@article{bertsekas1978mathematical,
  title={Mathematical issues in dynamic programming},
  author={Bertsekas, Dimitri and Shreve, Steven},
  journal={unpublished paper},
  year={1978}
}

@book{bertsekas1996stochastic,
  title={Stochastic optimal control: the discrete-time case},
  author={Bertsekas, Dimitri and Shreve, Steven},
  volume={5},
  year={1996},
  publisher={Athena Scientific}
}

@book{ambrosio2008gradient,
  title={Gradient flows: in metric spaces and in the space of probability measures},
  author={Ambrosio, Luigi and Gigli, Nicola and Savar{\'e}, Giuseppe},
  year={2008},
  publisher={Springer Science \& Business Media}
}

@article{fournier2015rate,
  title={On the rate of convergence in Wasserstein distance of the empirical measure},
  author={Fournier, Nicolas and Guillin, Arnaud},
  journal={Probability theory and related fields},
  volume={162},
  number={3},
  pages={707--738},
  year={2015},
  publisher={Springer}
}

@article{mohajerin2018data,
  title={Data-driven distributionally robust optimization using the Wasserstein metric: Performance guarantees and tractable reformulations},
  author={Mohajerin Esfahani, Peyman and Kuhn, Daniel},
  journal={Mathematical Programming},
  volume={171},
  number={1},
  pages={115--166},
  year={2018},
  publisher={Springer}
}

@book{klenke2013probability,
  title={Probability theory: a comprehensive course},
  author={Klenke, Achim},
  year={2013},
  publisher={Springer Science \& Business Media}
}

@article{gao2017distributionally,
  title={Distributionally robust stochastic optimization with dependence structure},
  author={Gao, Rui and Kleywegt, Anton},
  journal={arXiv preprint arXiv:1701.04200},
  year={2017}
}

@article{zaev2016ergodic,
	title={{On ergodic decompositions related to the Kantorovich problem}},
	author={Zaev, Danila},
	journal={Journal of Mathematical Sciences},
	volume={216},
	number={1},
	pages={65--83},
	year={2016},
	publisher={Springer}
}

@article{kolesnikov2017optimal,
	title={{Optimal transportation of processes with infinite Kantorovich distance: Independence and symmetry}},
	author={Kolesnikov, Alexander and Zaev, Danila},
	year={2017},
	volume={57},
	number={2},
	pages={293--324},
	journal={Kyoto Journal of Mathematics},
	publisher={Duke University Press}
}

@incollection{bertsimas2006robust,
  title={Robust and data-driven optimization: modern decision making under uncertainty},
  author={Bertsimas, Dimitris and Thiele, Aur{\'e}lie},
  booktitle={Models, methods, and applications for innovative decision making},
  pages={95--122},
  year={2006},
  publisher={INFORMS}
}

@article{popescu2007robust,
  title={Robust mean-covariance solutions for stochastic optimization},
  author={Popescu, Ioana},
  journal={Operations Research},
  volume={55},
  number={1},
  pages={98--112},
  year={2007},
  publisher={INFORMS}
}

@article{delage2010distributionally,
  title={Distributionally robust optimization under moment uncertainty with application to data-driven problems},
  author={Delage, Erick and Ye, Yinyu},
  journal={Operations research},
  volume={58},
  number={3},
  pages={595--612},
  year={2010},
  publisher={INFORMS}
}

@article{ben2013robust,
  title={Robust solutions of optimization problems affected by uncertain probabilities},
  author={Ben-Tal, Aharon and Den Hertog, Dick and De Waegenaere, Anja and Melenberg, Bertrand and Rennen, Gijs},
  journal={Management Science},
  volume={59},
  number={2},
  pages={341--357},
  year={2013},
  publisher={INFORMS}
}

@incollection{bayraksan2015data,
  title={Data-driven stochastic programming using phi-divergences},
  author={Bayraksan, G{\"u}zin and Love, David K},
  booktitle={The operations research revolution},
  pages={1--19},
  year={2015},
  publisher={INFORMS}
}

@article{tzortzis2015dynamic,
  title={Dynamic programming subject to total variation distance ambiguity},
  author={Tzortzis, Ioannis and Charalambous, Charalambos D and Charalambous, Themistoklis},
  journal={SIAM Journal on Control and Optimization},
  volume={53},
  number={4},
  pages={2040--2075},
  year={2015},
  publisher={SIAM}
}

@article{bertsimas2020predictive,
  title={From predictive to prescriptive analytics},
  author={Bertsimas, Dimitris and Kallus, Nathan},
  journal={Management Science},
  volume={66},
  number={3},
  pages={1025--1044},
  year={2020},
  publisher={INFORMS}
}

@article{hannah2010nonparametric,
  title={Nonparametric density estimation for stochastic optimization with an observable state variable},
  author={Hannah, Lauren and Powell, Warren and Blei, David},
  journal={Advances in Neural Information Processing Systems},
  volume={23},
  year={2010}
}

@misc{mathoverflow164842,
    TITLE = {Is an integral against a probability measure in the convex hull of the range? (version: 2014-04-30)},
    AUTHOR = {Alexander Pruss},
    HOWPUBLISHED = {MathOverflow},
    URL = {https://mathoverflow.net/q/164842}
}

@article{packard1993complex,
  title={The complex structured singular value},
  author={Packard, Andrew and Doyle, John},
  journal={Automatica},
  volume={29},
  number={1},
  pages={71--109},
  year={1993},
  publisher={Elsevier}
}

@book{rockafellar2009variational,
  title={Variational analysis},
  author={Rockafellar, Tyrrell and Wets, Roger},
  volume={317},
  year={2009},
  publisher={Springer Science \& Business Media}
}

@book{wendland2004scattered,
  title={Scattered data approximation},
  author={Wendland, Holger},
  volume={17},
  year={2004},
  publisher={Cambridge university press}
}

@book{schwartz1975lectures,
  title={Lectures on disintegration of measures},
  author={Schwartz, Laurent},
  year={1975},
  publisher={Bombay Tata Institute of fundamental research}
}

@book{christmann2008support,
  title={Support vector machines},
  author={Christmann, Andreas and Steinwart, Ingo},
  year={2008},
  publisher={Springer}
}

@book{berg1984harmonic,
  title={Harmonic analysis on semigroups: Theory of positive definite and related functions},
  author={Berg, Christian and Christensen, Jens Peter Reus and Ressel, Paul},
  volume={100},
  year={1984},
  publisher={Springer}
}
